\pdfoutput=1
\documentclass[letterpaper]{amsart}

\usepackage[utf8]{inputenc}
\usepackage{lmodern}
\usepackage[T1]{fontenc}
\usepackage{amssymb}
\usepackage{enumitem}
\usepackage{mathtools}

\newif\ifarxiv
\arxivtrue
\ifarxiv
\usepackage{tikz-cd}
\usepackage[pdfusetitle]{hyperref}
\else

\usepackage{tikz-cd}
\usepackage[pdfusetitle,dvipdfmx]{hyperref}
\fi
\tikzcdset{arrow style=Latin Modern}
\usepackage{cleveref}

\mathtoolsset{mathic}
\DeclareMathAlphabet\mathbfit{OML}{cmm}{b}{it}

\newlist{enumarabic}{enumerate}{1}
\setlist[enumarabic]{font=\normalfont,label=(\arabic*),leftmargin=0.3in}
\newlist{enumroman}{enumerate}{1}
\setlist[enumroman]{font=\normalfont,label=(\roman*),leftmargin=0.3in}

\def\citehgashc#1{\cite[#1]{Franz:hgashc}}

\def\cf{\emph{cf.}}

\def\doi#1{\href{http://doi.org/#1}{doi:#1}}

\makeatletter
\g@addto@macro\bfseries{\boldmath}
\makeatother

\newcommand*\xbar[1]{%
   \hbox{%
     \vbox{%
       \hrule height 0.35pt 
       \kern0.35ex
       \hbox{%
         \kern-0.15em
         \ensuremath{#1}%
         \kern-0.15em
       }%
     }%
   }%
} 

\let\newterm\emph

\def\N{\mathbb N}
\def\Z{\mathbb Z}
\def\R{\mathbb R}
\def\C{\mathbb C}

\DeclareMathOperator{\Hom}{Hom}
\DeclareMathOperator{\End}{End}
\DeclareMathOperator{\Tor}{Tor}

\def\cupone{\mathbin{\cup_1}}
\def\cuptwo{\mathbin{\cup_2}}

\def\deg#1{|#1|}

\let\shuffle\nabla

\def\kk{\Bbbk}
\def\Ll{\boldsymbol\Lambda}
\def\Sl{\mathbf S}
\def\Kl{\mathbf K}

\def\tildecc{\tilde c}
\def\ff{F}
\def\ffbar{f}

\let\epsilon\varepsilon
\let\phi\varphi

\def\RP{\mathbb{RP}}

\def\susp{\mathbf{s}}
\def\desusp{\susp^{-1}}

\def\transp#1{#1^{*}}
\let\transpp\transp

\def\aa{\mathbfit{a}}
\def\bb{\mathbfit{b}}

\def\xx{\mathbfit{x}}

\def\aaa{\mathfrak{a}}
\def\bbb{\mathfrak{b}}
\def\ccc{\mathfrak{c}}
\def\kkk{\mathfrak{k}}

\def\AW{AW}

\def\AWu#1{\AW_{\mkern -1mu #1}}

\def\EE{\mathbf{E}}
\def\EEE{\mathfrak{E}}
\def\FF{\mathbf{F}}
\def\ggg{\mathfrak{g}}

\def\ax{\kkk}
\def\pmax{\pmod\ax}
\def\pmkerf{\pmod{\ker\transp\ffbar}}

\def\ha{h^{a}}
\def\hc{h^{c}}

\def\invhc{k^{c}}
\def\barinvhc{\bar k^{c}}

\def\Laa{\Lambda^{G}}
\def\Lbb{\Lambda^{K}}
\def\Lbbone{\Lbb_{(1)}}
\def\FLbb{\Gamma_{\!\Lbb}}
\def\QBT{\Pi}

\def\Bar#1#2{\BB(\kk,#2,#1)}
\def\BarEl#1#2{#2\otimes#1}
\def\bigBar#1#2{\BB\bigl(\kk,#2,#1\bigr)}

\def\Ac#1#2{A^{#1}_{#2}}
\def\ac#1#2#3{\Ac{#1}{#2}(#3)}

\let\emptyset\varnothing

\def\setzero#1{[#1]}
\def\setone#1{\underline{#1}}

\def\iter#1#2{#1^{[#2]}}
\def\tGM{t_{\mathrm{GM}}}
\def\RC{U}

\def\BB{\mathbf{B}}
\def\BBone{\mathbf{1}}
\def\OM{\boldsymbol{\Omega}}

\let\KS\varkappa
\def\eqKS{\stackrel{\KS}{=}}

\DeclareMathOperator{\Sq}{Sq}

\def\mm{r}
\def\nn{r}

\setlistdepth{8}
\newlist{caselist}{enumerate}{8}
\setlist[caselist]{label*={\bf\arabic*.},wide,leftmargin=0.1in}

\def\thpair#1#2{Pair~\ref{#1}~\(\leftrightarrow\)~\ref{#2}\ \kern0pt}
\def\ii{\mathbfit{i}}
\def\jj{\mathbfit{j}}

\def\cancelswith#1{$\to$~\ref{#1}}

\newcommand{\nocontentsline}[3]{}
\newcommand{\tocless}[2]{\bgroup\let\addcontentsline=\nocontentsline{}\egroup}

\theoremstyle{plain}
\newtheorem{theorem}{Theorem}[section]
\newtheorem{proposition}[theorem]{Proposition}
\newtheorem{lemma}[theorem]{Lemma}
\newtheorem{corollary}[theorem]{Corollary}

\theoremstyle{definition}

\newtheorem{remark}[theorem]{Remark}
\newtheorem{example}[theorem]{Example}

\theoremstyle{remark}
\newtheorem*{acknowledgements}{Acknowledgements}

\numberwithin{equation}{section}

\allowdisplaybreaks[4]

\begin{document}

\title{The cohomology rings of homogeneous spaces}
\author{Matthias Franz}
\thanks{The author was supported by an NSERC Discovery Grant}
\address{Department of Mathematics, University of Western Ontario,
  London, Ont.\ N6A\;5B7, Canada}
\email{mfranz@uwo.ca}

\subjclass[2020]{Primary 57T15; secondary 16E45, 57T30, 57T35}

\begin{abstract}
  Let \(G\) be a compact connected Lie group and \(K\) a closed connected subgroup.
  Assume that the order of any torsion element in the integral cohomology of~\(G\) and~\(K\)
  is invertible in a given principal ideal domain~\(\kk\).
  It is known that in this case the cohomology of the homogeneous space~\(G/K\) with coefficients in~\(\kk\)
  and the torsion product of~\(H^{*}(BK)\) and~\(\kk\) over~\(H^{*}(BG)\)
  are isomorphic as \(\kk\)-modules. We show that this isomorphism
  is multiplicative and natural in the pair~\((G,K)\)
  provided that \(2\) is invertible in~\(\kk\).
  The proof uses homotopy Gerstenhaber algebras in an essential way.
  In particular, we show that the normalized singular cochains
  on the classifying space of a torus are formal as a homotopy Gerstenhaber algebra.
\end{abstract}

\maketitle

\tableofcontents

\section{Introduction}

In~1950, H.~Cartan gave the first uniform description of the cohomology
of homogeneous spaces of Lie groups.
Using a differential-geometric approach, he established the following result
for a compact connected Lie group~\(G\) and a closed connected subgroup~\(K\subset G\) \cite[Thm.~5]{Cartan:1951b}.

\begin{theorem}[H.~Cartan]
  \label{thm:Cartan}
  There is an isomorphism of graded algebras
  \begin{equation*}
    H^{*}(G/K;\R) \cong \Tor^{*}_{H^{*}(BG;\R)}\bigl(\R,H^{*}(BK;\R)\bigr).
  \end{equation*}
\end{theorem}

A topological way to look at this formula is the following:
One has a fibre bundle
\begin{equation}
  G/K \hookrightarrow EG/K = BK \to BG,
\end{equation}
and there is an associated Eilenberg--Moore spectral sequence
\begin{equation}
  E_{2} = \Tor^{*}_{H^{*}(BG;\R)}\bigl(\R,H^{*}(BK;\R)\bigr) \;\Rightarrow\; H^{*}(G/K).
\end{equation}
In this language, Cartan's result says that the spectral sequence collapses
at the second page and that the product on that page agrees with the one on~\(H^{*}(G/K)\).

The real  cohomology of the classifying space of a connected Lie group
is a polynomial algebra on even-degree generators.
An obvious question is whether a result analogous to Cartan's holds
for other principal ideal domains~\(\kk\) for which \(H^{*}(BG)\) and~\(H^{*}(BK)\)
have this property.
An equivalent condition is that the orders of the torsion subgroups of~\(H^{*}(G;\Z)\) and~\(H^{*}(K;\Z)\)
are invertible in~\(\kk\), and we assume this throughout.
It holds in many cases, for example for~\(U(n)\), \(SU(n)\) and~\(Sp(n)\) over any~\(\kk\),
and for~\(SO(n)\) and~\(\textit{Spin}(n)\) if \(2\) is invertible in~\(\kk\).

In his 1952~thesis, Borel studied the case where \(G\) and~\(K\) have the same rank and
established a multiplicative isomorphism \cite[Prop.~30.2]{Borel:1953}
\begin{equation}
  H^{*}(G/K)\cong H^{*}(BK) \bigm/ H^{>0}(BG)\cdot H^{*}(BK).
\end{equation}
The Leray--Hirsch theorem then implies that \(H^{*}(BK)\) is free over~\(H^{*}(BG)\), so that Borel's formula can be written as
\begin{equation}
  \label{eq:intro-iso-GK-Tor}
  H^{*}(G/K) \cong \Tor^{*}_{H^{*}(BG)}\bigl(\kk,H^{*}(BK)\bigr).
\end{equation}
Another step forward was achieved in~1968 by Baum, who proved that for field coefficients,
the Eilenberg--Moore spectral sequence
collapses at the second page for any~\(G\) and~\(K\) satisfying a certain `deficiency condition'
\cite[Thm.~7.4]{Baum:1968}.
This yields an additive isomorphism of the form~\eqref{eq:intro-iso-GK-Tor}.
Shortly afterwards,
May~\cite[p.~335]{May:1968a} announced that the Eilenberg--Moore spectral sequence collapses
for any~\(\kk\), independently of the deficiency condition.
Details appeared in Gugenheim--May~\cite[Thm.~A]{GugenheimMay:1974},
where additionally the extension problem was solved. This gives the following result.

\begin{theorem}
  \label{thm:intro:additive}
  If \(H^{*}(BG)\) and~\(H^{*}(BK)\) are polynomial algebras on even-degree generators,
  then there is an isomorphism of graded \(\kk\)-modules
  \begin{equation*}
    H^{*}(G/K) \cong \Tor^{*}_{H^{*}(BG)}\bigl(\kk,H^{*}(BK)\bigr).
  \end{equation*}
\end{theorem}

Munkholm~\cite[Thm.]{Munkholm:1974}\ provided a different proof of the isomorphism,
and Huse\-moller--Moore--Stash\-eff~\cite[Thm.~IV.8.2]{HusemollerMooreStasheff:1974}
a further one for the collapse of the Eilenberg--Moore spectral sequence. 
For field coefficients, yet another proof was published by Wolf~\cite[Thm.~B]{Wolf:1977}.
Later, May--Neumann~\cite{MayNeumann:2002} extended \Cref{thm:intro:additive} to generalized homogeneous spaces
(see also \Cref{rem:gen-homog}), and Barthel--May--Riehl~\cite{BarthelMayRiehl:2014} put
Gugenheim--May's approach into a model-theoretic framework.

Apart from one special case~\cite[Cor.~7.5]{Baum:1968},
the product structure is not addressed in any of the works after Borel.
In their introduction~\cite[p.~viii]{GugenheimMay:1974},
Gugenheim and May remark:

\begin{quote}\small
  Multiplicatively, however, we are left with an extension problem; our
  results will compute the associated graded algebras of~\(H^{*}(G/K)\)
  and [\dots] with respect to suitable filtrations. Refinements of our
  algebraic theory could conceivably yield precise procedures for the
  computation of these cohomology algebras. When \(\kk=\Z_{2}\), there are
  examples where the extensions are non-trivial. There are no such
  examples known when \(\kk\) is a field of characteristic~\(\ne2\).%
\footnote{We have aligned the original notation with ours.}
\end{quote}

The examples alluded to are the projective unitary groups~\(PU(n)=U(n)/U(1)\)
for~\(n\equiv 2 \pmod 4\), see \Cref{rem:PU2}. To the author's knowledge,
no progress on the multiplicative structure has been made since these words were written.
In the present paper we prove the following:

\begin{theorem}
  \label{thm:main}
  Assume that \(2\) is invertible in~\(\kk\).
  If \(H^{*}(BG)\) and~\(H^{*}(BK)\) are polynomial algebras,
  then there is an isomorphism of graded \(\kk\)-algebras
  \begin{equation*}
    H^{*}(G/K) \cong \Tor^{*}_{H^{*}(BG)}\bigl(\kk,H^{*}(BK)\bigr),
  \end{equation*}
  natural with respect to maps of pairs~\((G,K)\to(G',K')\).
\end{theorem}

The central difficulty one faces when proving an isomorphism of the form~\eqref{eq:intro-iso-GK-Tor}
is the lack of commutativity of the singular cochain algebra. At some point one
has to pass from cochains to cohomology, and unlike in the case of differential forms,
the assignment of representatives~\(a_{i}\in C^{*}(BG)\)
to generators~\(x_{i}\in H^{*}(BG)\) does not extend to
a morphism of differential graded algebras (dgas).
To address this, all approaches after Baum
resorted to some `up to homotopy' structure,
as suggested by Stasheff--Halperin \cite[p.~575]{StasheffHalperin:1970}.

Munkholm for example further develops the idea of strongly homotopy commutative (shc) algebras
introduced by Stasheff--Halperin. The only additional ingredient he then needs is that
both~\(BG\) and~\(BK\) have polynomial cohomology,
and his result holds more generally for the fibre of bundles where both the total space and the base have this property.

In contrast to this, Husemoller--Moore--Stasheff, Gugenheim--May and Wolf
rely on the existence of a maximal torus~\(T\subset K\) to reduce the problem to that of a homogeneous space~\(G/T\).
This was already done by Baum~\cite{Baum:1968}, who observed that \(H^{*}(G/K)\) injects into~\(H^{*}(G/T)\),
compare \Cref{thm:HPsi}\,\ref{thm:HPsi-2}.
A crucial result in this direction, also used by Wolf, is the following \cite[Thm.~4.1]{GugenheimMay:1974}.

\begin{theorem}[Gugenheim--May]
  \label{thm:gugenheim-may}
  There is a quasi-isomorphism of dgas~\(C^{*}(BT)\to H^{*}(BT)\)
  annihilating all \(\cupone\)-products.
\end{theorem}

We are going to extend \Cref{thm:gugenheim-may} to homotopy Gerstenhaber algebras (hgas),
which were introduced by Voronov--Gersten\-haber~\cite{VoronovGerstenhaber:1995}.
An hga structure on a dga~\(A\) is essentially a family of operations~%
\( 
  E_{k}\colon A^{\otimes(k+1)} \to A
\) 
that allow to define a product on the bar construction~\(\BB A\) compatible with the coalgebra structure.
Based on a result of Baues~\cite{Baues:1980}, the former authors also noted
that singular cochain algebras are endowed with this structure \cite{GerstenhaberVoronov:1995}.
In this case, the first hga operation~\(E_{1}\) is the usual \(\cupone\)-product, up to sign.
We strengthen the Gugenheim--May result as follows.

\begin{theorem}
  \label{thm:intro:hga-formality}
  There is a quasi-isomorphism of dgas~\(C^{*}(BT)\to H^{*}(BT)\)
  annihilating all hga operations.
  In particular, \(C^{*}(BT)\) is formal as an hga.
\end{theorem}

See \Cref{thm:ffbar-hga-formal}.
This seems to be the first time that the hga formality of a non-trivial space is established.
The quasi-isomorphism from \Cref{thm:intro:hga-formality} actually annihilates even more operations,
including the ones identified by Kadeishvili~\cite{Kadeishvili:2003} to construct a \(\cupone\)-product on~\(\BB C^{*}(BT)\).
The only exception is the \(\cuptwo\)-product on~\(C^{*}(BT)\), but we can show
that also \(\cuptwo\)-products of cocycles are in the kernel of the formality map
provided that \(2\) is invertible in~\(\kk\) (\Cref{thm:ffbar-cuptwo}).
We call an hga having a \(\cuptwo\)-product as well as the other additional operations ``extended''.

The following result from the companion paper~\cite{Franz:hgashc}
allows us to combine \Cref{thm:intro:hga-formality}
with Munk\-holm's techniques, see \Cref{thm:shc-mod-k}.

\begin{theorem}
  \label{thm:intro-hga-shc}
  Any extended hga is naturally an shc algebra in the sense of Munk\-holm.
\end{theorem}

In a nutshell, our strategy to prove \Cref{thm:main} is the following:
By the Eilenberg--Moore theorem, \(H^{*}(G/K)\) is naturally isomorphic to the differential torsion product
\begin{equation}
  \label{eq:intro:dg-Tor}
  \Tor_{C^{*}(BG)}\bigl(\kk,C^{*}(BK)\bigr).
\end{equation}
Kadeishvili--Saneblidze~\cite{KadeishviliSaneblidze:2005} observed that
the hga structure on cochains permits to define a product on the one-sided bar construction
underlying \eqref{eq:intro:dg-Tor};
the Eilenberg--Moore isomorphism then becomes multiplicative.
Imitating mostly Munkholm, we first construct a \(\kk\)-module isomorphism
\begin{equation}
  H^{*}(\Theta)\colon \Tor_{H^{*}(BG)}\bigl(\kk,H^{*}(BK)\bigr) \to \Tor_{C^{*}(BG)}\bigl(\kk,C^{*}(BK)\bigr)
\end{equation}
where we use the shc algebra structure given by \Cref{thm:intro-hga-shc}.
In order to show that our map is multiplicative and natural, we look at the composition
\begin{multline}
  \Tor_{C^{*}(HG)}\bigl(\kk,H^{*}(BK)\bigr) \xrightarrow{H^{*}(\Theta)} \Tor_{C^{*}(BG)}\bigl(\kk,C^{*}(BK)\bigr) \\
    \lhook\joinrel\longrightarrow \Tor_{C^{*}(BG)}\bigl(\kk,C^{*}(BT)\bigr) \stackrel{\cong}{\longrightarrow} \Tor_{C^{*}(BG)}\bigl(\kk,H^{*}(BT)\bigr).
\end{multline}
The last map involves the quasi-isomorphism from \Cref{thm:intro:hga-formality}
in the same way as Wolf applied the formality map constructed by Gugenheim--May. 
This leads to a dramatic simplification of the formulas and allows us to complete the proof
of \Cref{thm:main}, see \Cref{sec:homog}.

Along the way we exhibit an explicit homotopy between the two possible definitions
of a tensor product of two \(A_{\infty}\)-maps (\Cref{thm:tensor-shm}).

\begin{acknowledgements}
  Maple and Sage~\cite{Sage} were used to derive the formulas in Sections~\ref{sec:shm-tensor} and~\ref{sec:BT-formality}.
  The connection between tensor products of \(A_{\infty}\)-maps and hypercubes (\Cref{rem:cube})
  was discovered by consulting the OEIS~\cite{OEIS}.
  I thank Jeff Carlson and Xin Fu for carefully reading parts of an earlier version of this paper.
\end{acknowledgements}

\section{Preliminaries}
\label{sec:tw}

\subsection{Differential algebra}

We work over a fixed commutative ring~\(\kk\) with unit,
which will be assumed to be a principal ideal domain in Sections~\ref{sec:bundles}, \ref{sec:homog} and~\ref{sec:examples}.
Since we will mostly deal with cohomological complexes, we assume a cohomological grading throughout this review section.
The identity map on a complex~\(M\) is denoted \(1_{M}\).
The suspension map on a complex is denoted by~\(\susp\) and the desuspension by~\(\desusp\).
All tensor products are over~\(\kk\) unless otherwise indicated.

Given two \(\Z\)-graded complexes~\(A\) and~\(B\),
the complex~\(\Hom(A,B)\) consists in degree~\(n\in\Z\) of all linear maps~\(f\colon A\to B\)
raising degrees by~\(n\). The differential of such a map is
\begin{equation}
  d(f) = d\,f - (-1)^{n}f\, d.
\end{equation}

We write
\begin{equation}
  \label{eq:transposition}
  T=T_{A,B}\colon A\otimes B\to B\otimes A,
  \qquad
  a\otimes b\mapsto (-1)^{\deg{a}\deg{b}}\,b\otimes a
\end{equation}
for the transposition of factors in a tensor product.
This illustrates the Koszul sign rule, 
according to which swapping two objects of degrees~\(m\) and~\(n\)
incurs the sign~\((-1)^{mn}\). Another incarnation of it is the definition
\begin{equation}
  \label{eq:tensor-maps}
  f\otimes g\colon A\otimes B\to C\otimes D,
  \qquad
  a\otimes b \mapsto (-1)^{\deg{g}\deg{a}}\,f(a)\otimes g(b)
\end{equation}
of the tensor product of two maps~\(f\colon A\to C\) and~\(g\colon B\to D\).
This implies that for maps~\(f_{i}\colon A_{i}\to A_{i+1}\) and~\(g_{i}\colon B_{i}\to B_{i+1}\), \(i=1\),~\(2\), we have
\begin{equation}
  \label{eq:KS-maps}
  (f_{1}\otimes g_{1})\,(f_{2}\otimes g_{2})
  = (-1)^{\deg{f_{2}}\deg{g_{1}}}\, f_{1}\,f_{2} \otimes g_{1}\,g_{2}.
\end{equation}

We refer to~\cite[\S\S 1.1,~1.2,~1.11]{Munkholm:1974}
for the definitions of differential graded algebras (dgas) and dga maps 
as well as for differential graded coalgebras (dgcs), dgc maps and coalgebra homotopies.
By an \newterm{ideal}~\(\aaa\) of a dga~\(A\), we mean a two-sided differential ideal~\(\aaa\lhd A\).
We write augmentations as~\(\epsilon\) and coaugmentations as~\(\eta\).
The augmentation ideal of a dga~\(A\) is denoted by~\(\bar A\);
for any~\(a\in A\) we define \(\bar a=a-\eta\,\epsilon(a)\in\bar A\).
A dga~\(A\) is connected if it is \(\N\)-graded and \(\eta_{A}\colon\kk\to A^{0}\) is an isomorphism;
it is simply connected if additionally \(A^{1}=0\). A connected or simply connected dgc~\(C\) is defined similarly.

For~\(n\ge0\), we write
\begin{equation}
  \iter{\mu_{A}}{n}\colon A^{\otimes n}\to A
\end{equation}
for the iterated multiplication of a dga~\(A\), so that \(\iter{\mu_{A}}{0}=\eta_{A}\), \(\iter{\mu_{A}}{1}=1_{A}\) and~\(\iter{\mu_{A}}{2}=\mu_{A}\).
The iterations~\(\iter{\Delta}{n}\) are defined analogously.
A dgc~\(C\) is cocomplete if for any~\(c\in C\) there is an~\(n\ge0\) such that \((1_{C}-\epsilon_{C})^{\otimes n}\iter{\Delta}{n}(c)=0\).
Any connected dgc is cocomplete.

Given two ideals~\(\aaa\lhd A\) and~\(\bbb\lhd B\) where \(A\) and~\(B\) are dgas,
we define the ideal
\begin{equation}
  \aaa\boxtimes\bbb=\aaa\otimes B+A\otimes\bbb\lhd A\otimes B
\end{equation}
as well as
\begin{equation}
  \aaa^{\boxtimes 0}=0\lhd A^{\otimes 0}=\kk,
  \qquad
  \aaa^{\boxtimes 1}=\aaa,
  \qquad
  \aaa^{\boxtimes(n+1)}=\aaa^{\boxtimes n}\boxtimes\aaa\lhd A^{\otimes(n+1)}
\end{equation}
inductively for~\(n\ge1\).

We will make heavy use of the (reduced) bar construction
\begin{equation}
  \BB A = \bigoplus_{k\ge 0}\BB_{k}A,
  \qquad
  \BB_{k}A = (\desusp \bar A)^{\otimes k}
\end{equation}
of an augmented dga~\(A\), which is a cocomplete coaugmented dgc, connected if \(A\) is simply connected,
see~\cite[Sec.~II.3]{HusemollerMooreStasheff:1974} or~\cite[\S 1.6]{Munkholm:1974}.
We write \(\BBone_{\BB A}=1\in\kk=\BB_{0}A\) for the counit of~\(\BB A\).
The canonical map
\begin{equation}
  \label{eq:twc-bar}
  t_{A}\colon \BB A \rightarrow \BB_{1}A = \desusp\bar A \stackrel{\susp}{\longrightarrow} \bar A \hookrightarrow A
\end{equation}
is a twisting cochain in the sense of the following definition.

For an augmented dga~\(A\) and a coaugmented dgc~\(C\),
the complex~\(\Hom(C,A)\) is an augmented dga with cup product
\begin{equation}
  f\cup g = \mu_{A}\,(f\otimes g)\,\Delta_{C},
\end{equation}
unit element~\(\eta_{A}\,\epsilon_{C}\) and augmentation~\(\epsilon(f)=(\epsilon_{A}\,f\,\eta_{C})(1)\).
Note that for~\(f\),~\(g\) as before and any dgc map~\(k\colon B\to C\) we have
\begin{equation}
  \label{eq:Hom-cup-dgc}
  (f\cup g)\circ k = (f\circ k) \cup (g\circ k).
\end{equation}

A \newterm{twisting cochain} is an element~\(t\in\Hom^{1}(C,A)\) such that
\begin{gather}
  \label{eq:def-tw-cochain}
  d(t) = t\cup t, \\
  \label{eq:tw-normalization}
  \epsilon_{A}\,t = 0
  \qquad\text{and}\qquad
  t\,\eta_{C} = 0.
\end{gather}
If \(C\) is cocomplete, then the assignment \(f\mapsto t_{A}f\) sets up a bijection
between the dgc maps~\(C\to \BB A\) and the twisting cochains~\(C\to A\),
compare~\cite[Prop.~1.9]{Munkholm:1974}.

\begin{example}
  \label{ex:shuffle-map}
  Let \(A\) and~\(B\) be augmented dgas. The shuffle map
  \begin{equation}
    \shuffle=\shuffle_{A,B}\colon \BB A \otimes \BB B \to \BB(A\otimes B)
  \end{equation}
  is the dgc map with associated twisting cochain~\(t_{A}\otimes\eta_{B}\,\epsilon_{\BB B}+\eta_{A}\,\epsilon_{\BB A}\otimes t_{B}\),
  \begin{equation}
    [a_{1}|\dots|a_{k}] \otimes [b_{1}|\dots|b_{l}] \mapsto
    \begin{cases}
      a_{1}\otimes 1 & \text{if \(k=1\) and~\(l=0\),} \\
      1\otimes b_{1} & \text{if \(k=0\) and~\(l=1\),} \\
      0 & \text{otherwise.} \\
    \end{cases}
  \end{equation}
  The shuffle map is associative and also commutative in the sense that the diagram
  \begin{equation}
    \begin{tikzcd}
      \BB A\otimes\BB B \arrow{d}[left]{T_{\BB A,\BB B}} \arrow{r}{\shuffle_{A,B}} & \BB(A\otimes B) \arrow{d}{\BB T_{A,B}} \\
      \BB B\otimes\BB A \arrow{r}{\shuffle_{B,A}} & \BB(B\otimes A)
    \end{tikzcd}
  \end{equation}
  commutes. 
  
  If \(A\) is commutative, then the composition
  \begin{equation}
    \mu_{\BB A}=\BB\mu_{A}\,\shuffle_{A,A}\colon \BB A \otimes \BB A \to \BB A
  \end{equation}
  turns \(\BB A\) into a \newterm{dg~bialgebra}, that is, into a coaugmented dgc
  with an associative product that is a morphism of dgcs.
\end{example}

An element~\(h\in\Hom_{0}(C,A)\) is a \newterm{twisting cochain homotopy}
from the twisting cochain~\(t\colon C\to A\) to the twisting cochain~\(u\colon C\to A\), in symbols \(h\colon t\simeq u\), if
\begin{gather}
  \label{eq:tw-homotopy}
  d(h) = t\cup h - h\cup u, \\
  \label{eq:tw-h-normalization}
  \epsilon_{A}\,h = \epsilon_{C}
  \qquad\text{and}\qquad
  h\,\eta_{C} = \eta_{A}.
\end{gather}
Assume again that \(C\) is cocomplete, and let \(f\),~\(g\colon C\to\BB A\) be two dgc maps.
The assignment~\(h\mapsto 1+t_{A}\,h\) then is a bijection between the coalgebra homotopies
from~\(f\) to~\(g\) and the twisting cochain homotopies from~\(t_{A}\,f\) to~\(t_{A}\,g\),
see~\cite[\S 1.11]{Munkholm:1974}.

Let \(h\colon C\to A\) be a twisting cochain homotopy, and let \(\aaa\lhd A\) be an ideal.
If \(h\) is congruent to~\(1=\eta_{A}\,\epsilon_{C}\) modulo~\(\aaa\lhd A\),
we say that \(h\) as well as the associated coalgebra homotopy~\(C\to \BB A\) is \newterm{\(\aaa\)-trivial}.\footnote{%
  The importance of this notion as well as that of \(\bbb\)-strict shm maps defined in~\eqref{eq:def-shm-strict} below
  will only become evident in Section~\ref{sec:homog}. Readers may wish to ignore them on a first reading.}  
By the first normalization condition~\eqref{eq:tw-h-normalization} any twisting cochain homotopy~\(h\colon C\to A\) is \(\bar A\)-trivial.

\begin{lemma}
  \label{thm:tw-h-equiv-rel}
  Let \(\aaa\lhd A\) be an ideal, and let \(C\) be a cocomplete dgc.
  Being related by an \(\aaa\)-trivial homotopy is an equivalence relation among twisting cochains~\(C\to A\).
  More precisely:
  \begin{enumroman}
  \item
    \label{thm:cup-trivial-h}
    Let \(h\colon t\simeq u\) and \(k\colon u\simeq v\)
    be \(\aaa\)-trivial twisting cochain homotopies.
    Then \(h\cup k\) is an \(\aaa\)-trivial homotopy from~\(t\) to~\(v\).
  \item
    \label{thm:inverse-trivial}
    Let \(h\colon t\simeq u\) be an \(\aaa\)-trivial twisting cochain homotopy
    Then \(h\) is invertible in~\(\Hom_{0}(C,A)\), and its inverse
    \begin{equation*}
      h^{-1}=\sum_{n=0}^{\infty}(1-h)^{\cup n} \colon C\to A
    \end{equation*}
    is an \(\aaa\)-trivial homotopy from~\(u\) to \(t\).
  \end{enumroman}
\end{lemma}

In particular, we may unambiguously speak of an ``\(\aaa\)-trivial homotopy between twisting cochains~\(t\) and~\(u\)''
without specifying the direction of the homotopy.

\begin{proof}
  The first part follows immediately from the definition of the cup product.
  Apart from the obvious \(\aaa\)-triviality, the second claim is \cite[\S 1.12]{Munkholm:1974}.
\end{proof}

\subsection{Notation}

The Koszul signs~\eqref{eq:transposition} and~\eqref{eq:tensor-maps}
quickly tend to clutter more complex formulas, as do the arguments
of multilinear maps. For instance, in \Cref{sec:def-hga} 
we will encounter the formula
\begin{equation}
  \label{eq:kappa-example-withsign}
  E_{k}(a_{1}a_{2};b_{1},\dots,b_{k})
  = \sum_{l+m=k}\! (-1)^{\epsilon}\,E_{l}(a_{1};b_{1},\dots,b_{l})
  \, E_{m}(a_{2};b_{l+1},\dots,b_{k})
\end{equation}
for certain multilinear operations~\(E_{k}\colon A^{\otimes(k+1)}\to A\) of degree~\(-k\)
on a dga~\(A\) and elements~\(a_{1}\),~\(a_{2}\), \(b_{1}\),~\dots,~\(b_{k}\in A\).
Here the sign exponent is
\begin{equation}
  \label{eq:kappa-example-sign}
  \epsilon = \deg{a_{2}}\,\bigl(\deg{b_{1}} + \dots + \deg{b_{l}}\bigr)
  - m\,\bigl(\deg{a_{1}}+\deg{b_{1}} + \dots + \deg{b_{l}}\bigr),
\end{equation}
and it it completely determined by the Koszul sign rule.
Alternatively, one could dispense with the arguments and write \eqref{eq:kappa-example-withsign}
more concisely as the identity of functions
\begin{equation}
  \label{eq:kappa-example-maps}
  E_{k}\,(\mu_{A}\otimes 1^{\otimes k}) = \sum_{l+m=k} \mu_{A}\,(E_{l}\otimes E_{m})\,\pi_{l}
\end{equation}
where \(\mu_{A}\) is the multiplication in~\(A\) and
\(\pi_{l}\colon A^{\otimes(k+2)}\to A^{\otimes(k+2)}\) the permutation of factors
corresponding to the cycle~\((l+2,l+1,\dots,2)\in S_{k+2}\).
The advantage of such a notation is that it is easy to compute the differential of a map
because compositions as well as tensor products of maps obey the graded Leibniz rule.
For example, the differential of a term~\(\mu_{A}(E_{l}\otimes E_{m})\pi_{l}\) is
\begin{equation}
  \label{eq:example-diff-maps}
  d\bigl(\mu_{A}\,(E_{l}\otimes E_{m})\,\pi_{l}\bigr) =
  \mu_{A}\,\bigl(d(E_{l})\otimes E_{m}\bigr)\,\pi_{l} + (-1)^{l}\,\mu_{A}\,\bigl(E_{l}\otimes d(E_{m})\bigr)\,\pi_{l}.
\end{equation}
However, we feel that it is very hard to grasp the meaning of formulas of the
form~\eqref{eq:kappa-example-maps}, mostly because the effect of the permutations remains opaque.

We therefore propose another notation that aims to combine the advantages
of~\eqref{eq:kappa-example-withsign} and~\eqref{eq:kappa-example-maps}.
We write maps in the form~\eqref{eq:kappa-example-withsign}, but without all
Koszul signs involving the degrees of variables. To indicate that these signs need
to be added, we write ``\(\eqKS\)'' instead of an equality sign.
For example, the identity~\eqref{eq:kappa-example-withsign} is written as
\begin{equation}
  E_{k}(a_{1}a_{2};b_{1},\dots,b_{k})
  \eqKS \sum_{l+m=k} \!\! E_{l}(a_{1};b_{1},\dots,b_{l}) \, E_{m}(a_{2};b_{l+1},\dots,b_{k}).
\end{equation}
In this case, the sign to be added is exactly \eqref{eq:kappa-example-sign}.
Stated differently, we really describe formulas of the form~\eqref{eq:kappa-example-maps},
but use variables to specify the permutations of arguments like~\(\pi_{l}\) that are to be applied
before the maps that are spelt out.

To make the notation even more compact, we abbreviate sequences of variables with a bullet, as in
\begin{equation}
  \label{eq:example-bullets}
  E_{k}(a_{1}a_{2};b_{\bullet})
  \eqKS \sum_{l+m=k} \!\! E_{l}(a_{1};b_{\bullet}) \, E_{m}(a_{2};b_{\bullet}).
\end{equation}
The number of elements in each sequence is to be inferred from the maps and may be zero.
Since \(E_{k}\) takes \(k\)~arguments in addition to the leading~\(a_{1}a_{2}\),
the first occurrence of~\(b_{\bullet}\) above stands for \(k\)~arguments~\(b_{1}\),~\dots,~\(b_{k}\).
Throughout a product the order of ``bullet variables'' is always maintained.
Thus, the first~\(b_{\bullet}\) on the right-hand side of~\eqref{eq:example-bullets} stands for~\(b_{1}\),~\dots,~\(b_{l}\)
(as \(E_{l}\) takes \(l\)~arguments in addition to~\(a_{1}\)) and
the last~\(b_{\bullet}\) for the \(m\)~arguments~\(b_{l+1}\),~\dots,~\(b_{k}\).
A tensor product like~\(a_{\bullet}\otimes b_{\bullet}\) indicates
a sequence of tensors~\(a_{1}\otimes b_{1}\),~\(a_{2}\otimes b_{2}\),~\ldots\,.

Composition of maps is distributed over tensor products, so that the
linear order in which the maps appear in a formula is maintained when translating between our
and the corresponding function notation. For example, the formula
\begin{equation}
  F(a,b) \eqKS f_{1}(f_{2}(a))\otimes g_{1}(g_{2}(b))
\end{equation}
stands for the identity of functions~\(F = f_{1}\,f_{2} \otimes g_{1}\,g_{2}\).
Note that our ``\(\eqKS\)''~notation does not incorporate the Koszul sign~\eqref{eq:KS-maps}
involving only maps and no variables.

\section{Strongly homotopy multiplicative maps}
\label{sec:shm}

Our discussion is based on the treatment in~\cite[\S 3.1]{Munkholm:1974} and~\cite[Sec.~1\,(c)]{Wolf:1977}.

Let~\(A\) and~\(B\) be augmented dgas. By definition, an \newterm{\(A_{\infty}\)-map} or \newterm{strongly homotopy multiplicative (shm) map}\footnote{%
  We prefer the term ``shm map'' used by Munkholm over the nowadays more popular terminology ``\(A_{\infty}\)-map''
  because it pairs better with the ``shc algebras'' to be introduced in \Cref{sec:shc}.}~%
\( 
  f\colon A\Rightarrow B
\) 
is a twisting cochain~\(f\colon \BB A\to B\). We write the corresponding dgc map as~\(\BB f\colon \BB A\to \BB B\).
It is given by
\begin{multline}
  \label{eq:shm-to-dgc}
  \BB f\bigl([a_{1}|\dots|a_{n}]\bigr) = \\ \sum_{k\ge0}\,\sum_{i_{1}+\dots+i_{k}=n} \!\!\!
  \bigl[\, f[a_{1}|\dots|a_{i_{1}}] \bigm| f[a_{i_{1}+1}|\dots|a_{i_{2}}] \bigm| \ldots \bigm|  f[a_{n-i_{k}+1}|\dots|a_{n}] \,\bigr],
\end{multline}
where the second sum is over all decompositions of~\(n\) into \(k\)~positive integers.

Following Munkholm~\cite[Appendix]{Munkholm:1974},
we define for~\(n\ge0\) the map%
\footnote{This definition leads to a sign convention different from Wolf's~\cite[p.~319]{Wolf:1977}.}
\begin{equation}
  f_{(n)} 
  \colon \bar A^{\otimes n} \xrightarrow{(\desusp)^{\otimes n}} \BB_{n}A \stackrel{f}{\longrightarrow} B
\end{equation}
of degree~\(1-n\) and extend it to~\(A^{\otimes n}\) by setting
\begin{align}
  \label{eq:tw-fam-1}
  f_{(1)}(1) &= 1, \\
  \label{eq:tw-fam-1-bis}
  f_{(n)}(a_{1}\otimes\dots\otimes a_{n}) &= 0
  \quad\text{if \(n\ge2\) and \(a_{k}=1\) for some~\(k\).}
\end{align}
The twisting cochain conditions~\eqref{eq:def-tw-cochain} and~\eqref{eq:tw-normalization} for~\(f\) translate into
\begin{align}
  f_{(0)} &=
  \epsilon_{B}\,f_{(n)} = 0, \\
  \label{eq:tw-fam-2}
  d(f_{(n)})(a_{\bullet}) &\eqKS \sum_{k=1}^{n-1}(-1)^{k}\,
  \bigl(f_{(k)}(a_{\bullet})\,f_{(n-k)}(a_{\bullet})
  - f_{(n-1)}(a_{\bullet},a_{k}a_{k+1},a_{\bullet})\bigr)
\end{align}
for all~\(n\ge1\). In~\eqref{eq:tw-fam-2} we have used the symbol~``\(\eqKS\)'' to indicate the Koszul sign
and also the notation~``\(a_{\bullet}\)'' to denote a sequence
of \(a\)-variables, ordered by their indices.

We call a family of multilinear functions
\begin{equation}
  f_{(n)}\colon A^{\otimes n}\to B
\end{equation}
of degree~\(1-n\) satisfying \eqref{eq:tw-fam-1}--\eqref{eq:tw-fam-2} a \newterm{twisting family}.
Twisting families correspond bijectively to shm maps~\(A\Rightarrow B\)
and therefore to dgc maps~\(\BB A\to\BB B\). The dgc map determined by the twisting family~\(f_{(n)}\) can be read of from~\eqref{eq:shm-to-dgc},
using the identity
\begin{equation}
  \label{eq:tw-fam-to-tw}
  f\bigl([a_{1}|\dots|a_{n}]\bigr) = (-1)^{\epsilon}\,f_{(n)}(a_{1}\otimes\dots\otimes a_{n})
\end{equation}
for~\(n\ge0\), where
\begin{equation}
  \label{eq:tw-fam-to-tw-sign}
  \epsilon = \sum_{k=1}^{n}(n-k)\bigl(\deg{a_{k}}-1\bigr).
\end{equation}

It follows from~\eqref{eq:tw-fam-2} that the component~\(f_{(1)}\colon A\to B\) is a chain map which is multiplicative up to homotopy since
\begin{equation}
  d(f_{(2)}) = f_{(1)}\,\mu_{A} - \mu_{B}\,(f_{(1)}\otimes f_{(1)}).
\end{equation}
The map
\begin{equation}
  H^{*}(f) \coloneqq H^{*}(f_{(1)})\colon H^{*}(A)\to H^{*}(B)
\end{equation}
therefore is a morphism of graded algebras.

Any dga morphism~\(f\colon A\to B\) induces an shm map~\(\tilde f\colon A\Rightarrow B\)
with \(\tilde f_{(1)}=f\) and \(\tilde f_{(n)}=0\) for~\(n\ge2\). We call such an shm map \newterm{strict}.
Note that \(H^{*}(\tilde f)=H^{*}(f)\) in this case.
We will not distinguish between a dga map and its induced strict shm~map.

More generally, we say that an shm map~\(f\colon A\Rightarrow B\) is \newterm{\(\bbb\)-strict} for some~\(\bbb\lhd B\) if
\begin{equation}
  \label{eq:def-shm-strict}
  f_{(n)}\equiv0\pmod{\bbb}\qquad\text{for all~\(n\ge2\).}
\end{equation}
Then \(f\) is \(0\)-strict if and only if it is strict, and every~\(f\colon A\Rightarrow B\) is \(\bar B\)-strict.
Any \(\bbb\)-strict shm~map~\(f\colon A\Rightarrow B\) induces a strict map~\(A\to B/\bbb\).

A twisting cochain homotopy~\(h\colon f\simeq g\) from an shm map~\(f\colon A\Rightarrow B\) to another shm map~\(g\colon A\Rightarrow B\)
is called an \newterm{shm homotopy}.
Based on~\(h\) we define the maps
\begin{equation}
  \label{eq:def-tw-h-fam}
  h_{(n)} = h\,(\desusp)^{\otimes n}\colon \bar A^{\otimes n}\to B
\end{equation}
of degree~\(-n\) for~\(n\ge0\) and extend them to~\(A^{\otimes n}\) by
\begin{equation}
  \label{eq:tw-h-family-1}
  h_{(n)}(a_{1}\otimes\dots\otimes a_{n}) = 0
  \quad\text{if \(a_{k}=1\) for some~\(k\).}  
\end{equation}
The normalization conditions~\eqref{eq:tw-h-normalization} mean
\begin{equation}
  \label{eq:tw-h-family-2}
  h_{(0)} = \eta_{B}
  \qquad\text{and}\qquad
  \epsilon_{B}\,h_{(n)} = 0
  \quad\text{for~\(n\ge1\),}
\end{equation}
and condition~\eqref{eq:tw-homotopy} is equivalent to
\begin{align}
  \label{eq:tw-h-family-3}
  d(h_{(n)})(a_{\bullet}) &\eqKS \sum_{k=1}^{n-1}(-1)^{k}\,h_{(n-1)}(a_{\bullet},a_{k}a_{k+1},a_{\bullet}) \\*
  \notag
  &\qquad + \sum_{k=0}^{n}\Bigl( f_{(k)}(a_{\bullet})\,h_{(n-k)}(a_{\bullet})
  - (-1)^{k}\,h_{(k)}(a_{\bullet})\, g_{(n-k)}(a_{\bullet}) \Bigr)
\end{align}
for all~\(n\ge0\).
In particular, \(h_{(1)}\colon g_{(1)}\simeq f_{(1)}\), so that \(H^{*}(f)=H^{*}(g)\).

We call a family of multilinear functions
\begin{equation}
  h_{(n)}\colon A^{\otimes n}\to B
\end{equation}
of degree~\(-n\)
satisfying \eqref{eq:tw-h-family-1}--
\eqref{eq:tw-h-family-3}
a \newterm{twisting homotopy family} from the twisting family~\(f_{(*)}\) to~\(g_{(*)}\).
Twisting homotopy families correspond bijectively to homotopies between twisting cochains.
We also write \(Bh\colon \BB A\to \BB B\) for the coalgebra homotopy induced by the twisting cochain homotopy~\(h\colon \BB A\to B\).

A twisting homotopy family~\(h_{(*)}\) as above is called \(\bbb\)-trivial for some~\(\bbb\lhd B\)
if the twisting cochain homotopy~\(\BB A\to B\) is so. Equivalently,
\begin{equation}
  \label{eq:h-trivial-def}
  h_{(n)}\equiv0\pmod{\bbb}\qquad\text{for all~\(n\ge1\).}
\end{equation}

Let \(f\colon A\Rightarrow B\) and~\(g\colon B\Rightarrow C\) be shm~maps.
We define the composition
\begin{equation}
  g\circ f\colon A\Rightarrow C
\end{equation}
to be the twisting cochain~\(g\,\BB f\) associated to the dgc map~\(\BB g\,\BB f\colon \BB A\to \BB C\).
Using \eqref{eq:shm-to-dgc} and~\eqref{eq:tw-fam-to-tw}, one sees that
the corresponding twisting family is given by
\begin{equation}
  \label{eq:twc-composition}
  (g\circ f)_{(n)}(a_{\bullet}) \eqKS \sum_{k\ge1}\sum_{i_{1}+\dots+i_{k}=n} \!\! (-1)^{\epsilon}\:
  g_{(k)}\bigl(f_{(i_{1})}(a_{\bullet}),\dots,f_{(i_{k})}(a_{\bullet})\bigr)
\end{equation}
for~\(n\ge0\), where the second sum is over all decompositions of~\(n\) into \(k\)~positive integers and
\begin{equation}
  \label{eq:twc-composition-sign}
  \epsilon = \sum_{s=1}^{k}(k-s)(i_{s}-1).
\end{equation}
Note that the composition of an shm map with the canonical twisting cochain~\eqref{eq:twc-bar} of a bar construction,
considered as another shm map, is tautological in the sense that \(t_{B}\circ f=f\) and~\(g\circ t_{B}=g\).

The composition of an shm map and an shm homotopy is similarly defined
as the shm homotopy associated to the composition of the corresponding maps
between bar constructions.

\begin{lemma}
  \label{thm:trivial-h-comp}
  \( \)
  \begin{enumroman}
  \item \label{thm:trivial-h-comp-0}
    Let \(f\colon A\Rightarrow B\) be a \(\bbb\)-strict shm map, and let \(g\colon B\Rightarrow C\) a \(\ccc\)-strict shm map.
    If \(g_{(1)}(\bbb)\subset\ccc\), then \(g\circ f\) is \(\ccc\)-strict.
  \item \label{thm:trivial-h-comp-1}
    Let \(h\colon C\to A\) be an \(\aaa\)-trivial twisting cochain homotopy,
    and let \(f\colon A\to B\) be a \(\bbb\)-strict shm map.
    If \(f_{(1)}(\aaa)\subset\bbb\), then \(f\circ h\) is \(\bbb\)-trivial.
  \item \label{thm:trivial-h-comp-2}
    Let \(h\colon C\to A\) be an \(\aaa\)-trivial twisting cochain homotopy, and let \(g\colon D\to C\) be a map of coaugmented dgcs.
    Then \(h\circ g\) is \(\aaa\)-trivial.
  \end{enumroman}
\end{lemma}

\begin{proof}
  The first two claims are readily verified, and the last one is trivial.
\end{proof}

\section{Tensor products of shm~maps}
\label{sec:shm-tensor}

In this section, \(A\),~\(B\),~\(A'\) and~\(B'\) denote augmented dgas.
We write \(a_{\bullet}\otimes b_{\bullet}\) for a sequence~\(a_{1}\otimes b_{1}\),~\(a_{2}\otimes b_{2}\),~\dots\ 
in~\(A\otimes B\) whose length is given by the context.

Let \(f\colon A\Rightarrow A'\) be an shm~map,
and let \(g\colon B\rightarrow B'\) be a dga map. Then
\begin{equation}
  \label{eq:f-otimes-g-mult}
  (f\otimes g)_{(n)}(a_{\bullet}\otimes b_{\bullet})
  \eqKS f_{(n)}(a_{\bullet})\otimes g\,\iter{\mu}{n}(b_{\bullet})
\end{equation}
is a twisting family, hence defines an shm map
\begin{equation}
  f\otimes g\colon A\otimes B \Rightarrow A'\otimes B'.
\end{equation}
If \(h\) is an \(\aaa\)-trivial homotopy from~\(f\) to another shm~map~\(\tilde f\), then
\begin{equation}
  \label{eq:def-h-g-mult}
  (h\otimes g)_{(n)}(a_{\bullet}\otimes b_{\bullet}) \eqKS h_{(n)}(a_{\bullet})\otimes g\,\iter{\mu}{n}(b_{\bullet})
\end{equation}
defines an \(\aaa\otimes B\)-trivial shm homotopy~\(h\otimes g\) from~\(f\otimes g\) to~\(\tilde f\otimes g\).

Similarly, if \(f\colon A\rightarrow A'\) is a dga map
and \(g\colon B\Rightarrow B'\) an shm~map, then
\begin{equation}
  \label{eq:f-mult-otimes-g}
  (f\otimes g)_{(n)}(a_{\bullet}\otimes b_{\bullet})
  \eqKS f\,\iter{\mu}{n}(a_{\bullet})\otimes g_{(n)}(b_{\bullet})
\end{equation}
defines an shm map
\begin{equation}
  f\otimes g\colon A\otimes B \Rightarrow A'\otimes B'.
\end{equation}
If \(h\) is a \(\bbb\)-trivial homotopy from~\(g\) to another shm~map~\(\tilde g\), then
\begin{equation}
  \label{eq:def-f-mult-h}
  (f\otimes h)_{(n)}(a_{\bullet}\otimes b_{\bullet}) \eqKS f\,\iter{\mu}{n}(a_{\bullet})\otimes h_{(n)}(b_{\bullet})
\end{equation}
defines an \(A\otimes\bbb\)-trivial shm homotopy~\(f\otimes h\) from~\(f\otimes g\) to~\(f\otimes\tilde g\).

Now let both \(f\colon A\Rightarrow A'\) and \(g\colon B\Rightarrow B'\) be shm~maps.
Then the two shm maps
\begin{equation}
  (f\otimes1_{B'})\circ(1_{A}\otimes g)
  \qquad\text{and}\qquad
  (1_{A'}\otimes g)\circ(f\otimes 1_{B})
\end{equation}
are not equal in general.
In fact, for any~\(n\ge0\) one has
\begin{equation}
  \label{eq:f-1-1-g}
  \bigl((f\otimes 1)\circ(1\otimes g)\bigr)_{(n)}(a_{\bullet}\otimes b_{\bullet}) \eqKS \sum_{l\ge1}\,\sum_{j_{1}+\dots+j_{l}=n}\!\!(-1)^{\epsilon}\,F\otimes G
\end{equation}
where the sum is over all decompositions of~\(n\) into \(l\)~positive integers and
\begin{align}
  F &= f_{(l)}\,\bigl(\iter{\mu}{j_{1}}(a_{\bullet}),\dots,\iter{\mu}{j_{l}}(a_{\bullet})\bigr), \\
  G &= \iter{\mu}{l}\,\bigl(g_{(j_{1})}(b_{\bullet}),\dots,g_{(j_{l})}(b_{\bullet})\bigr), \\
  \epsilon &= \sum_{t=1}^{l}(l-t)(j_{t}-1),
\end{align}
compare \eqref{eq:twc-composition} and~\eqref{eq:twc-composition-sign},
while
\begin{equation}
  \label{eq:1-g-f-1}
  \bigl((1\otimes g)\circ(f\otimes 1)\bigr)_{(n)}(a_{\bullet}\otimes b_{\bullet}) \eqKS \sum_{k\ge1}\,\sum_{i_{1}+\dots+i_{k}=n}\!\!(-1)^{\epsilon}\,F\otimes G
\end{equation}
where the sum is analogously over all decompositions of~\(n\) into \(k\)~positive integers and
\begin{align}
  F &= \iter{\mu}{k}\,\bigl(f_{(i_{1})}(a_{\bullet}),\dots,f_{(i_{k})}(a_{\bullet})\bigr), \\
  G &= g_{(k)}\,\bigl(\iter{\mu}{i_{1}}(b_{\bullet}),\dots,\iter{\mu}{i_{k}}(b_{\bullet})\bigr), \\
  \label{eq:1-g-f-1-epsilon}
  \epsilon &= \sum_{s=1}^{k}(s-1)(i_{s}-1).
\end{align}
Note that if \(f\) or~\(g\) is strict, then \eqref{eq:f-1-1-g} and~\eqref{eq:1-g-f-1} coincide and agree with the formulas given previously.
Following Munkholm~\cite[Prop.~3.3]{Munkholm:1974}, we define
\begin{equation}
  \label{eq:def-tensor-Ainfty-maps}
  f\otimes g = (f\otimes1)\circ(1\otimes g)
\end{equation}
in the general case and compare it to the other composition. 

\begin{proposition}
  \label{thm:tensor-shm}
  Assume that \(f\) is \(\aaa\)-strict for some~\(\aaa\lhd A'\) and that \(g\) is \(\bbb\)-strict for some~\(\bbb\lhd B'\).
  Then the two shm maps
  \begin{equation*}
    f\otimes g=(f\otimes1)\circ(1\otimes g)
    \qquad\text{and}\qquad
    (1\otimes g)\circ(f\otimes 1)
  \end{equation*}
  are homotopic via an \(\aaa\otimes\bbb\)-trivial homotopy.
  In particular, if \(f\) or~\(g\) is strict, then the two compositions agree.
\end{proposition}

\begin{proof}
  Instead of using Munkholm's theory of trivialized extensions \cite[Sec.~2]{Munkholm:1974},
  we exhibit an explicit homotopy from~\((1\otimes g)\circ(f\otimes 1)\) to~\((f\otimes1)\circ(1\otimes g)\).
  It is given by~\(h_{(0)}=\eta_{A'}\otimes\eta_{B'}\) and
  \begin{equation}
    \label{eq:def-Hn}
    h_{(n)}(a_{\bullet}\otimes b_{\bullet}) \eqKS \sum_{k,l\ge1}\mkern-0mu\sum_{\substack{i_{1}+\dots+i_{k}+\\j_{1}+\dots+j_{l}=n}}\mkern-5mu
    (-1)^{\epsilon}\,F \otimes G
  \end{equation}
  for~\(n\ge1\),
  where the second sum is over all decompositions of~\(n\) into~\(k+l\) positive integers,
  \begin{align}
    \label{eq:def-Hn-F}
    F &= \iter{\mu}{k}\Bigl(f_{(i_{1})}(a_{\bullet}),\dots,f_{(i_{k-1})}(a_{\bullet}),f_{(i_{k}+l)}\bigl(a_{\bullet},\iter{\mu}{j_{1}}(a_{\bullet}),\dots,\iter{\mu}{j_{l}}(a_{\bullet})\bigr)\Bigr), \\
    \label{eq:def-Hn-G}
    G &= \iter{\mu}{l}\Bigl(g_{(k+j_{1})}\bigl(\iter{\mu}{i_{1}}(b_{\bullet}),\dots,\iter{\mu}{i_{k}}(b_{\bullet}),b_{\bullet}\bigr),g_{(j_{2})}(b_{\bullet}),\dots,g_{(j_{l})}(b_{\bullet})\Bigr), \\
    \label{eq:def-epsilon-ij}
    \epsilon &= \sum_{s=1}^{k}s\,(i_{s}-1) + \sum_{t=1}^{l}(l-t)(j_{t}-1) + k\,(l-1)+1.
  \end{align}
  Verifying that \(h\) is a homotopy as claimed is lengthy, but elementary, see \Cref{sec:pf-tensor-prod-Ai}.
  That the homotopy is \(\aaa\otimes\bbb\)-trivial follows from the assumptions on~\(f\) and~\(g\)
  and the inequalities~\(i_{k}+l\ge2\) and~\(k+j_{1}\ge2\).
  In particular, \(h\) takes values in \(\bar A'\boxtimes\bar B'\supset\bar A'\otimes\bar B'\) since \(f\) and~\(g\) are \(\bar A\)-strict and \(\bar B\)-strict, respectively.
  This proves the second part of the normalization condition~\eqref{eq:tw-h-family-2}.

  Let us verify the condition~\eqref{eq:tw-h-family-1}:
  Assume that \(a_{i}=b_{i}=1\) for some~\(i\) and consider a term~\(F\otimes G\) of the sum~\eqref{eq:def-Hn}.
  Let \(m\) be the index such that \(a_{i}\) appears is the \(m\)-th \(f\)-term of~\(F\).
  If \(i_{s}>1\) or~\(s=k\), this term vanishes by~\eqref{eq:tw-fam-2}. Otherwise, the product inside~\(g_{(k+j_{1})}\) containing~\(b_{m}\) is \(b_{m}\)~itself,
  so that this term vanishes again by~\eqref{eq:tw-fam-2}. In any case we have \(F\otimes G=0\).
  
  The last part of the statement is the special case~\(\aaa=0\) or~\(\bbb=0\)
  and has already been observed above.
\end{proof}

Omitting the arguments~\(a_{\bullet}\otimes b_{\bullet}\),
formula~\eqref{eq:def-Hn} looks as follows in small degrees.
\begin{align}
  h_{(1)} &= 0, \\
  h_{(2)} &\eqKS
  - f_{(2)} ( a_{{1}},a_{{2}} ) \otimes g_{(2)} ( b_{{1}},b_{{2}} ),
  \\
  h_{(3)} &\eqKS
  - f_{(1)} ( a_{{1}} ) \, f_{(2)} ( a_{{2}},a_{{3}} ) \otimes g_{(3)} ( b_{{1}},b_{{2}},b_{{3}} )
\\ \notag &\quad\,+ f_{(3)} ( a_{{1}},a_{{2}},a_{{3}} ) \otimes g_{(2)} ( b_{{1}},b_{{2}} ) \, g_{(1)} ( b_{{3}} ) 
\\ \notag &\quad\,+ f_{(3)} ( a_{{1}},a_{{2}},a_{{3}} ) \otimes g_{(2)} ( b_{{1}}\, b_{{2}},b_{{3}} ) 
\\ \notag &\quad\,+ f_{(2)} ( a_{{1}},a_{{2}}\, a_{{3}} ) \otimes g_{(3)} ( b_{{1}},b_{{2}},b_{{3}} ).
\end{align}

\begin{remark}
  \label{rem:cube}
  The summands appearing in~\eqref{eq:f-1-1-g} and~\eqref{eq:1-g-f-1}
  are in bijection with the vertices of an \((n-1)\)-dimensional cube.
  For example, the vertex of~\([0,1]^{n-1}\) corresponding to the decomposition~\(i_{1}+\dots+i_{k}=n\)
  is given by
  \begin{equation}
    \bigl(\underbrace{0,\dots,0,1}_{i_{1}},\dots,\underbrace{0,\dots,0,1}_{i_{k-1}},\underbrace{0,\dots,0}_{i_{k}-1}\bigr).
  \end{equation}
  Similarly, the summands appearing in~\eqref{eq:def-Hn} are in bijection
  with the edges of an \((n-1)\)-dimensional cube. Here
  the summand corresponding
  to the decomposition \(i_{1}+\dots+i_{k}+j_{1}+\dots+j_{l}=n\)
  is identified with the edge
  \begin{equation}
    \bigl(\underbrace{0,\dots,0,1}_{i_{1}},\dots,\underbrace{0,\dots,0,1}_{i_{k-1}},
    \underbrace{0,\dots,0}_{i_{k}-1},*,
    \underbrace{0,\dots,0}_{j_{1}-1},\underbrace{1,0,\dots,0}_{j_{2}},\dots,\underbrace{1,0,\dots,0}_{j_{l}}
    \bigr),
  \end{equation}
  where ``\(*\)'' denotes the free parameter.
\end{remark}

\begin{corollary}
  \label{thm:comp-shm-tensor}
  Let \(f_{1}\colon A_{0}\Rightarrow A_{1}\),~\(f_{2}\colon A_{1}\Rightarrow A_{2}\), \(g_{1}\colon B_{0}\Rightarrow B_{1}\) and \(g_{2}\colon B_{1}\Rightarrow B_{2}\)
  be shm maps. Assume that \(f_{1}\) is \(\aaa_{1}\)-trivial, \(f_{2}\) \(\aaa_{2}\)-trivial and \(g_{2}\) \(\bbb_{2}\)-trivial
  and that \((f_{2})_{(1)}(\aaa_{1})\subset\aaa_{2}\)
  for ideals~\(\aaa_{1}\lhd A_{1}\),~\(\aaa_{2}\lhd A_{2}\) and~\(\bbb_{2}\lhd B_{2}\).
  Then the two shm maps
  \begin{equation*}
    (f_{2}\otimes g_{2})\circ (f_{1}\otimes g_{1})
    \qquad\text{and}\qquad
    (f_{2}\circ f_{1})\otimes(g_{2}\circ g_{1})
  \end{equation*}
  are homotopic via an \(\aaa_{2}\otimes\bbb_{2}\)-trivial homotopy.
  If \(f_{1}\) or~\(g_{2}\) are strict, then the two maps agree.
\end{corollary}

\begin{proof}
  This follows by writing the maps as
  \begin{align}
    (f_{2}\otimes g_{2})\circ (f_{1}\otimes g_{1})
    &= (f_{2}\otimes 1)\circ(1\otimes g_{2})\circ (f_{1}\otimes 1)\circ(1\otimes g_{1}) \\
    (f_{2}\circ f_{1})\otimes(g_{2}\circ g_{1})
    &= (f_{2}\otimes 1)\circ(f_{1}\otimes 1)\circ(1\otimes g_{2})\circ(1\otimes g_{1})
  \end{align}
  and applying \Cref{thm:tensor-shm} and \Cref{thm:trivial-h-comp}.
  The second identity above is a consequence of the formulas~\eqref{eq:f-otimes-g-mult},~\eqref{eq:f-mult-otimes-g} and~\eqref{eq:twc-composition}.
\end{proof}

\begin{lemma}
  \label{thm:shuffle-natural-shm}
  The shuffle map is natural with respect to shm~maps. In other words, the diagram
  \begin{equation*}
    \begin{tikzcd}
      \BB A\otimes \BB B \arrow{d}[left]{\BB f\otimes \BB g} \arrow{r}{\shuffle} & \BB(A\otimes B) \arrow{d}{\BB(f\otimes g)} \\
      \BB A'\otimes \BB B' \arrow{r}{\shuffle} & \BB(A'\otimes B')
    \end{tikzcd}
  \end{equation*}
  commutes for all shm~maps~\(f\colon A\Rightarrow A'\) and~\(g\colon B\Rightarrow B'\).
\end{lemma}

\begin{proof}
  Since all morphisms involved are dgc~maps and the bar construction cocomplete, it suffices to compare the associated twisting cochains.
  Let \(\aa\otimes\bb=[a_{1}|\dots|a_{k}]\otimes[b_{1}|\dots|b_{l}]\in \BB_{k}A\otimes \BB_{l}B\).

  Assume \(g=1_{B}\). Then both twisting cochains vanish
  on~\(\aa\otimes\bb\) if \(k\ge1\) and~\(l\ge1\).
  For~\(l=0\) both twisting cochains yield \(f(\aa)\otimes1\), and for~\(k=0\)
  they give \(1\otimes b_{1}\) if~\(l=1\) and \(0\) otherwise, compare \Cref{ex:shuffle-map}.

  The case~\(f=1_{A}\) is analogous, and the general case follows by combining the two
  and using the definition~\eqref{eq:def-tensor-Ainfty-maps}.
\end{proof}

Now let \(f_{i}\colon A_{i}\Rightarrow B_{i}\) be a family of shm maps, \(1\le i\le m\).
Generalizing \eqref{eq:def-tensor-Ainfty-maps}, we define the shm map
\begin{multline}
  f_{1}\otimes\dots\otimes f_{m} = \\
  (f_{1}\otimes 1\otimes\dots\otimes 1)\circ(1\otimes f_{2}\otimes 1\otimes\dots\otimes 1)\circ\dots\circ(1\otimes\dots\otimes 1\otimes f_{m}).
\end{multline}
If one of the maps is instead an shm homotopy~\(f_{i}=h\), we use the same definition.
The resulting map is an shm homotopy in this case.
We observe that this convention is compatible with the definitions~\eqref{eq:def-h-g-mult} and~\eqref{eq:def-f-mult-h}.

\begin{lemma}
  \label{thm:tensor-shm-h}
  Let \(h\colon A\to B\) be an shm homotopy.
  \begin{enumroman}
  \item For any dga map~\(f\colon A'\to B'\) we have
    \begin{equation*}
      h\otimes f = (1_{B}\otimes f)\circ(h\otimes 1_{A'})
      \qquad\text{and}\qquad
      f\otimes h = (1_{B'}\otimes h)\circ(f\otimes 1_{A}).
    \end{equation*}
  \item For any shm map~\(g\colon C\to A\) and any dga~\(D\) we have
    \begin{equation*}
      (h\otimes 1_{D})\circ(g\otimes 1_{D}) = (h\circ g)\otimes 1_{D}.
    \end{equation*}
  \end{enumroman}
\end{lemma}

\begin{proof}
  The first part follows from inspection of the formulas~\eqref{eq:def-h-g-mult} and~\eqref{eq:def-f-mult-h}.
  The second claim additionally uses that formula~\eqref{eq:twc-composition} remains valid
  for the shm homotopy~\(h\) instead of the shm map~\(f\).
\end{proof}

\section{Strongly homotopy commutative algebras}
\label{sec:shc}

Let \(A\) be an augmented dga.
According to Stasheff--Halperin~\cite[Def.~8]{StasheffHalperin:1970},
\(A\) is a \newterm{strongly homotopy commutative (shc) algebra} if
\begin{enumroman}
\item
  \label{shc-shm}
  the multiplication map~\(\mu_{A}\colon A\otimes A\to A\) extends to an shm~morphism
  \begin{equation*}
    \Phi\colon A\otimes A\Rightarrow A,
  \end{equation*}
  where ``extending'' means that \(\Phi_{(1)}=\mu_{A}\).
\end{enumroman}
Munkholm~\cite[Def.~4.1]{Munkholm:1974} additionally requires the following:
\begin{enumroman}[resume]
\item
  \label{shc-unit}
  The map~\(\eta_{A}\) is a unit for~\(\Phi\), that is,
  \begin{equation*}
    \Phi\circ(1_{A}\otimes\eta_{A}) = \Phi\circ(\eta_{A}\otimes 1_{A}) = 1_{A}\colon A\Rightarrow A.
  \end{equation*}
\item
  \label{shc-ass}
  The shm map~\(\Phi\) is homotopy associative, that is,
  \begin{equation*}
    \Phi\circ(\Phi\otimes 1_{A}) \simeq \Phi\circ(1_{A}\otimes\Phi)\colon A\otimes A\otimes A\Rightarrow A.
  \end{equation*}
  We write \(\ha\) for a homotopy from~\(\Phi\circ(\Phi\otimes 1)\) to~\(\Phi\circ(1\otimes\Phi)\).
\item
  \label{shc-com}
  The map~\(\Phi\) is homotopy commutative, that is,
  \begin{equation*}
    \Phi\circ T_{A,A} \simeq \Phi\colon A\otimes A\Rightarrow A.
  \end{equation*}
  We write \(\hc\) for a homotopy from~\(\Phi\circ T\) to~\(\Phi\).
\end{enumroman}
Whenever we speak of an shc algebra, we mean one satisfying all four properties unless otherwise indicated.
Any commutative dga is canonically an shc algebra.

Let \(A\) and~\(B\) be shc algebras, and let \(\bbb\lhd B\).
A \newterm{morphism of shc algebras} is an shm map~\(f\colon A\Rightarrow B\)
such that the diagram
\begin{equation}
  \label{eq:def-natural-shc-map}
  \begin{tikzcd}
    A\otimes A \arrow[Rightarrow]{d}[left]{\Phi_{A}} \arrow[Rightarrow]{r}{f\otimes f} & B\otimes B \arrow[Rightarrow]{d}{\Phi_{B}} \\
    A \arrow[Rightarrow]{r}{f} & B
  \end{tikzcd}
\end{equation}
commutes up to homotopy.\footnote{%
  Munkholm also requires the identity~\(f\circ\eta_{A}=\eta_{B}\).
  Given the normalization condition~\eqref{eq:tw-normalization},
  this holds automatically as both maps necessarily represent \(0\) as twisting cochains~\(\kk=\BB\kk\to B\).}
It is called \newterm{\(\bbb\)-strict} if it is so as an shm~map,
and \newterm{\(\bbb\)-natural}
if there is a \(\bbb\)-trivial homotopy
making \eqref{eq:def-natural-shc-map} commute.

Recall from~\cite[Prop.~4.2]{Munkholm:1974} that the tensor product of two shc algebras~\(A\) and~\(B\)
is again an shc algebra with structure map
\begin{equation}
  \Phi_{A\otimes B}\colon A\otimes B\otimes A\otimes B \xrightarrow{1\otimes T_{B,A}\otimes 1} A\otimes A\otimes B\otimes B \xRightarrow{\Phi_{A}\otimes\Phi_{B}} A\otimes B.
\end{equation}

The following is a variant of the result just cited.

\begin{lemma}
  \label{thm:tensor-prod-shc}
  Let \(f_{i}\colon A_{i}\to B_{i}\) be strict \(\bbb_{i}\)-natural shc maps for~\(i=1\),~\(2\).
  Their tensor product~\(f_{1}\otimes f_{2}\colon A_{1}\otimes A_{2}\to B_{1}\otimes B_{2}\) is a strict \(\bbb_{1}\boxtimes\bbb_{2}\)-natural shc map.
\end{lemma}

\begin{proof}
  We have to show that there is a \(\bbb_{1}\boxtimes\bbb_{2}\)-trivial homotopy
  for the diagram~\eqref{eq:def-natural-shc-map},
  which in the present setting reads
  \begin{equation}
    \begin{tikzcd}[column sep=6em]
      A_{1}\otimes A_{2}\otimes A_{1}\otimes A_{2} \arrow{d}{1\otimes T\otimes 1} \arrow{r}{f_{1}\otimes f_{2}\otimes f_{1}\otimes f_{2}} & B_{1}\otimes B_{2}\otimes B_{1}\otimes B_{2} \arrow{d}{1\otimes T\otimes 1} \\
      A_{1}\otimes A_{1}\otimes A_{2}\otimes A_{2} \arrow[Rightarrow]{d}{1\otimes 1\otimes\Phi} \arrow{r}{f_{1}\otimes f_{1}\otimes f_{2}\otimes f_{2}} & B_{1}\otimes B_{1}\otimes B_{2}\otimes B_{2} \arrow[Rightarrow]{d}{1\otimes 1\otimes\Phi} \\
      A_{1}\otimes A_{1}\otimes A_{2} \arrow[Rightarrow]{d}{\Phi\otimes 1} \arrow{r}{f_{1}\otimes f_{1}\otimes f_{2}} & B_{1}\otimes B_{1}\otimes B_{2} \arrow[Rightarrow]{d}{\Phi\otimes 1} \\
      A_{1}\otimes A_{2} \arrow{r}{f_{1}\otimes f_{2}} & B_{1}\otimes B_{2} \mathrlap{.}
    \end{tikzcd}
  \end{equation}
  Since \(f_{1}\) and~\(f_{2}\) are strict, the top square commutes.
  If \(h_{i}\) denotes a \(\bbb_{i}\)-trivial naturality homotopy for~\(f_{i}\),
  then \(f_{1}\otimes f_{1}\otimes h_{2}\) is a \(B_{1}\otimes B_{1}\otimes\bbb_{2}\)-natural homotopy making the middle diagram commute,
  and \(h_{1}\otimes f_{2}\) is a \(\bbb_{1}\otimes B_{2}\)-natural one for the bottom square.
  Hence the cup product of
  \begin{equation}
    (\Phi\otimes 1) \circ (f_{1}\otimes f_{1}\otimes h_{2}) \circ (1\otimes T\otimes 1)
  \end{equation}
  and
  \begin{equation}
    (h_{1}\otimes f_{2}) \circ (1\otimes 1\otimes\Phi) \circ (1\otimes T\otimes 1)
  \end{equation}
  yields the required homotopy by Lemmas~\ref{thm:trivial-h-comp} and~\ref{thm:tw-h-equiv-rel}\,\ref{thm:cup-trivial-h}.
\end{proof}

Let \(A\) be an shc algebra with structure map~\(\Phi\colon A\otimes A \Rightarrow A\).
Following \cite[p.~30]{Munkholm:1974}, we define the shm~map
\begin{equation}
  \label{eq:def-Phi-n}
  \iter{\Phi}{n}\colon A^{\otimes n}\Rightarrow A
\end{equation}
for~\(n\ge0\) by
\begin{equation}
  \Phi^{[0]} = \eta_{A},
  \qquad
  \Phi^{[1]} = 1_{A},
  \qquad
  \Phi^{[2]} = \Phi,
  \qquad
  \iter{\Phi}{n+1} = \Phi\circ (\iter{\Phi}{n}\otimes 1_{A})
\end{equation}
for~\(n\ge2\). Note that
\begin{equation}
  \bigl(\iter{\Phi}{n}\bigr)_{(1)} = \iter{\mu_{A}}{n}.
\end{equation}
If \(\Phi\) is \(\aaa\)-strict for some ideal~\(\aaa\lhd A\), then so is \(\iter{\Phi}{n}\) for any~\(n\ge0\)
by \Cref{thm:trivial-h-comp}\,\ref{thm:trivial-h-comp-0}.

For the next result, compare \cite[Prop.~4.6]{Munkholm:1974}.

\begin{lemma}
  \label{thm:Phi-n-homotopy}
  Let \(A\) and~\(B\) be shc algebras with ideals~\(\aaa\lhd A\) and~\(\bbb\lhd B\).
  Assume that \(\Phi_{A}\) is \(\aaa\)-strict and \(\Phi_{B}\) \(\bbb\)-strict.
  Let \(f\colon A\Rightarrow B\) be a \(\bbb\)-strict and \(\bbb\)-natural map of shc algebras such that \(f_{(1)}(\aaa)\subset\bbb\).
  Then the diagram
  \begin{equation*}
    \begin{tikzcd}
      A^{\otimes n} \arrow[Rightarrow]{d}[left]{\iter{\Phi}{n}} \arrow[Rightarrow]{r}{f^{\otimes n}} & B^{\otimes n} \arrow[Rightarrow]{d}{\iter{\Phi}{n}} \\
      A \arrow[Rightarrow]{r}{f} & B
    \end{tikzcd}
  \end{equation*}
  commutes up to a \(\bbb\)-trivial homotopy for any~\(n\ge0\).
\end{lemma}

\begin{proof}
  The claim is trivial for~\(n\le2\). Assume it proven for~\(n\) and consider the diagram
  \begin{equation}
    \begin{tikzcd}[column sep=large]
      A^{\otimes n}\otimes A \arrow[Rightarrow]{d}[left]{\iter{\Phi_{A}}{n}\otimes 1} \arrow[Rightarrow]{r}{1^{\otimes n}\otimes f} &
      A^{\otimes n}\otimes B \arrow[Rightarrow]{d}[left]{\iter{\Phi_{A}}{n}\otimes 1} \arrow[Rightarrow]{r}{f^{\otimes n}\otimes 1} &
      B^{\otimes n}\otimes B \arrow[Rightarrow]{d}{\iter{\Phi_{B}}{n}\otimes 1}
      \\
      A\otimes A \arrow[Rightarrow]{d}[left]{\Phi_{A}} \arrow[Rightarrow]{r}{1\otimes f} &
      A\otimes B \arrow[Rightarrow]{r}{f\otimes 1} &
      B\otimes B \arrow[Rightarrow]{d}{\Phi_{B}}
      \\
      A \arrow[Rightarrow]{rr}{f} &
      &
      B \mathrlap{.}
    \end{tikzcd}
  \end{equation}

  Since \(\iter{\Phi_{A}}{n}\) is \(\aaa\)-strict,
  the top left square commutes up to an \(\aaa\otimes B\)-trivial homotopy by \Cref{thm:tensor-shm}.
  The composition of this homotopy with \(\Phi_{B}\circ(f\otimes 1)\) is \(\bbb\)-trivial by \Cref{thm:trivial-h-comp}
  because \(\Phi_{B}\) and~\(f\) are \(\bbb\)-strict and \(f_{(1)}(\aaa)\subset\bbb\).
  By induction, the top right square commutes up to a \(\bbb\otimes B\)-trivial homotopy,
  whose composition with~\(\Phi_{B}\) is \(\bbb\)-trivial.
  The bottom rectangle finally commutes up to a \(\bbb\)-trivial homotopy
  since \(f\) is \(\bbb\)-natural. The claim follows.  
\end{proof}

\section{Homotopy Gerstenhaber algebras}
\label{sec:hga}

\subsection{Definition of an hga}
\label{sec:def-hga}

Let \(A\) be an augmented dga. We say that \(A\) is a \newterm{homotopy Gerstenhaber algebra} (homotopy G-algebra, \newterm{hga})
if it is equipped with certain operations
\begin{equation}
  E_{k}\colon A \otimes A^{\otimes k}\to A,
  \qquad
  a\otimes b_{1}\otimes\dots\otimes b_{k}\mapsto E_{k}(a;b_{1},\dots,b_{k})
\end{equation}
of degree~\(\deg{E_{k}}=-k\) for~\(k\ge1\). To state the properties they satisfy,
it is convenient to use the additional operation~\(E_{0}=1_{A}\).
  All~\(E_{k}\) with~\(k\ge1\) take values in the augmentation ideal~\(\bar A\) and vanish
  if any argument is equal to~\(1\).
  For~\(k\ge1\) and all~\(a\),~\(b_{1}\),~\dots,~\(b_{k}\in A\) one has
  \begin{align}
    \label{eq:def-Ek-d}
    d(E_{k})(a;b_{\bullet})
    &\eqKS b_{1}\,E_{k-1}(a;b_{\bullet})
    + \sum_{m=1}^{k-1}(-1)^{m}\,E_{k-1}(a;b_{\bullet},b_{m}b_{m+1},b_{\bullet}) \\
    \notag &\qquad + (-1)^{k}\,E_{k-1}(a;b_{\bullet})\,b_{k}.
  \end{align}
  For~\(k\ge0\) and all~\(a_{1}\),~\(a_{2}\),~\(b_{1}\),~\dots,~\(b_{k}\in A\) one has
  \begin{equation}
    \label{eq:Eprodfirstarg}
    E_{k}(a_{1}a_{2};b_{\bullet}) \eqKS \!\!\sum_{k_{1}+k_{2}=k}\!\! E_{k_{1}}(a_{1};b_{\bullet})\,E_{k_{2}}(a_{2};b_{\bullet})
  \end{equation}
  where the sum is over all decompositions of~\(k\) into two non-negative integers.
  Finally, for~\(k\),~\(l\ge0\) and all~\(a\),~\(b_{1}\),~\dots,~\(b_{k}\),~\(c_{1}\),~\dots,~\(c_{l}\in A\) one has
  \begin{multline}
    \label{eq:formula-Ek-El}
    E_{l}(E_{k}(a;b_{\bullet});c_{\bullet}) \eqKS \\* \sum_{\substack{i_{1}+\dots+i_{k}+{}\\j_{0}+\dots+j_{k}=l}}
    \!\!(-1)^{\epsilon}\,
    E_{n}\bigl(a;\underbrace{c_{\bullet}}_{j_{0}},E_{i_{1}}(b_{1};c_{\bullet}),\underbrace{c_{\bullet}}_{j_{1}},
    \dots,\underbrace{c_{\bullet}}_{j_{k-1}},E_{i_{k}}(b_{k};c_{\bullet}),\underbrace{c_{\bullet}}_{j_{k}}\bigr),
  \end{multline}
  where the sum is over all decompositions of~\(l\) into \(2k+1\)~non-negative integers,
  \begin{equation}
    n = k + \sum_{t=0}^{k}j_{t}
    \qquad\text{and}\qquad
    \epsilon = \sum_{s=1}^{k}i_{s}\Bigl(k+\sum_{t=s}^{k}j_{t}\Bigr) + \sum_{t=1}^{k}t\,j_{t}.
  \end{equation}
A \newterm{morphism of hgas} is a morphism~\(f\colon A\to B\) of augmented dgas
that is compatible with the hga operations in the obvious way.  
  
Given an hga~\(A\), we can define
\begin{equation}
  \EE_{kl}\colon \BB_{k}A \otimes \BB_{l}A = (\desusp A)^{\otimes(k+l)} \to A
\end{equation}
for \(k\),~\(l\ge 0\) by
\begin{equation}
  \label{eq:def-Ekl}
  \EE_{kl} \, (\desusp)^{\otimes(k+l)} =
  \begin{cases}
    1_{A} & \text{if \(k=0\) and~\(l=1\),} \\
    E_{l} & \text{if \(k=1\),} \\
    0 & \text{otherwise}.
  \end{cases}
\end{equation}
The functions~\(\EE_{kl}\) assemble to a map
\begin{equation}
  \label{eq:def-EE}
  \EE\colon \BB A \otimes \BB A \to A,
\end{equation}
which is a twisting cochain
by~\eqref{eq:def-Ek-d} and~\eqref{eq:Eprodfirstarg} together with the normalization conditions.
Moreover, the identity~\eqref{eq:formula-Ek-El} implies that the induced dgc map
\begin{equation}
  \mu_{\BB A}\colon \BB A \otimes \BB A \to \BB A
\end{equation}
is associative and therefore turns \(\BB A\) into a dg~bialgebra.
Conversely, a dg~bialgebra structure on~\(\BB A\) whose associated
twisting cochain~\(\EE\) is of the form~\eqref{eq:def-Ekl} defines
an hga structure on~\(A\) with operations~\(E_{k}\).

\begin{remark}
  \label{rem:hga-braces}
  Our hga operations are related to the braces originally defined by Voronov and Gerstenhaber~%
  \cite[\S 8]{VoronovGerstenhaber:1995},~\cite[Sec.~1.2]{GerstenhaberVoronov:1995},~\cite[Sec.~3.2]{Voronov:2000}
  by the identity
  \begin{equation}
    \label{eq:E-braces}
    a\{b_{1},\dots,b_{k}\} = \EE_{1k}\bigl([\bar a]\otimes[\bar b_{1}|\dots|\bar b_{k}]\bigr)
     = (-1)^{\epsilon}\,E_{k}(a;b_{1},\dots,b_{k})
  \end{equation}
  for~\(k\ge1\) where
  \begin{equation}
    \label{eq:E-braces-epsilon}
    \epsilon = k\,\deg{a} + \sum_{m=1}^{k}(k-m)\,\deg{b_{m}}.
  \end{equation}
  (Compare formula~\eqref{eq:tw-fam-to-tw-sign}.)
  Our grading agrees with~\cite{Voronov:2000};
  in~\cite{VoronovGerstenhaber:1995} and~\cite{GerstenhaberVoronov:1995}
  the degrees of the desuspended arguments are used.\footnote{%
    The signs given in eqs.~(6) and~(7) of~\cite{GerstenhaberVoronov:1995} appear to be incorrect.}
\end{remark}

We observe that the \newterm{\(\cupone\)-product}
\begin{equation}
  a\cupone b = - E_{1}(a;b)
\end{equation}
is a homotopy from the product with commuted factors to the standard one,
\begin{equation}
  \label{eq:cupone-d}
  d(a\cupone b)+ da\cupone b+(-1)^{\deg{a}}\,a\cupone db =  ab - (-1)^{\deg{a}\deg{b}}\,ba
\end{equation}
satisfying the \newterm{Hirsch formula}
\begin{equation}
  \label{eq:hirsch-formula}
  ab \cupone c = (-1)^{\deg{a}}\,a(b\cupone c) + (-1)^{\deg{b}\deg{c}}(a\cupone c)\,b
\end{equation}
for~\(a\),~\(b\),~\(c\in A\).
As a consequence, the cohomology~\(H^{*}(A)\) 
is (graded) commutative and in fact a Gerstenhaber algebra with bracket
\begin{align}
  \label{eq:def-gerstenhaber-bracket}
  \bigl\{ [a], [b] \bigr\}
  &= \bigl[ E_{1}(a;b) - (-1)^{(\deg{a}-1)(\deg{b}-1)} E_{1}(b;a) \bigr] \\
  \notag &= (-1)^{\deg{a}-1}\bigl[\,a\cupone b + (-1)^{\deg{a}\deg{b}}\,b\cupone a\,\bigr]
\end{align}
for~\(a\),~\(b\in A\), see~\cite[\S 10]{VoronovGerstenhaber:1995}.

The main examples of hgas are the cochains on a simplicial set, see \Cref{sec:cochains-hga},
and the Hochschild cochains of an algebra, see the references given above. 
Any commutative dga is canonically an hga by setting \(E_{k}=0\) for all~\(k\ge1\).
The induced multiplication on~\(\BB A\) then is the shuffle product
discussed in \Cref{ex:shuffle-map}.

We say that an hga~\(A\) is \newterm{formal} if it is quasi-isomorphic to its cohomology~\(H^{*}(A)\), considered as an hga.

\subsection{Extended hgas}
\label{sec:extended-hga}

In his study of \(\cup_{i}\)-products on~\(\BB A\) for~\(i\ge1\),
Kadeishvili introduced operations~\(E^{i}_{kl}\) 
for an hga~\(A\) defined over~\(\kk=\Z_{2}\) \cite{Kadeishvili:2003}.
He called an hga equipped with these operations an `extended hga'.
We will only need the family~\(F_{kl} = E^{1}_{kl}\), but for coefficients in any~\(\kk\).
We therefore say that an hga is \emph{extended} if it has a family of operations
\begin{equation}
  F_{kl}\colon A^{\otimes k}\otimes A^{\otimes l}\to A
\end{equation}
of degree~\(\deg{F_{kl}}=-(k+l)\) for~\(k\),~\(l\ge 1\), satisfying the following conditions.
All operations~\(F_{kl}\) take values in the augmentation ideal~\(\bar A\) and vanish if any argument equals~\(1\in A\).
Their differential is given by
\begin{equation}
  \label{eq:d-Fkl}
  d(F_{kl})(a_{\bullet};b_{\bullet}) = A_{kl} + (-1)^{k}\,B_{kl}
\end{equation}
for all~\(a_{1}\),~\dots,~\(a_{k}\),~\(b_{1}\),~\dots,~\(b_{l}\in A\),
where
\begin{align}
  A_{1l} &= E_{l}(a_{1};b_{\bullet}), \\
  A_{kl} &\eqKS a_{1}\,F_{k-1,l}(a_{\bullet};b_{\bullet})
    + \sum_{i=1}^{k-1} (-1)^{i}\,F_{k-1,l}(a_{\bullet},a_{i}a_{i+1},a_{\bullet};b_{\bullet}) \\
  \notag &\qquad + \sum_{j=1}^{l} (-1)^{k}\, F_{k-1,j}(a_{\bullet};b_{\bullet})\,E_{l-j}(a_{k};b_{\bullet}) \\
\shortintertext{for~\(k\ge2\), and}
  B_{k1} &\eqKS -E_{k}(b_{1};a_{\bullet}), \\
  B_{kl} &\eqKS \sum_{i=0}^{k-1} E_{i}(b_{1};a_{\bullet})\,F_{k-i,l-1}(a_{\bullet};b_{\bullet})
    + \sum_{j=1}^{l-1} (-1)^{j}\,F_{k,l-1}(a_{\bullet};b_{\bullet},b_{j}b_{j+1},b_{\bullet}) \\
    \notag &\qquad + (-1)^{l}\,F_{k,l-1}(a_{\bullet};b_{\bullet})\,b_{l}
\end{align}
for~\(l\ge2\), compare~\cite[Def.~2]{Kadeishvili:2003}.

In particular, the operation~\(\cuptwo=-F_{11}\) is a \newterm{\(\cuptwo\)-product} for~\(A\) in the sense that
\begin{equation}
  \label{eq:cuptwo}
  d(\cuptwo)(a;b) = a\cupone b +(-1)^{\deg{a}\deg{b}}\,b\cupone a
\end{equation}
for all~\(a\),~\(b\in A\).
This implies that the Gerstenhaber bracket in~\(H^{*}(A)\) is trivial.

A \newterm{morphism of extended hgas} is a morphism of hgas that commutes with all operations~\(F_{kl}\), \(k\),~\(l\ge1\).

The following observation will be used in \Cref{sec:inverting-2}.

\begin{lemma}
  \label{thm:cuptwo}
  Let \(f\colon A\to B\) be a morphism of hgas
  where \(A\) is extended and \(B\) a commutative graded algebra, for example~\(B=H^{*}(A)\).
  Then for any cocycles~\(a\),~\(b\in A\),
  the value~\(f(a\cuptwo b)\)
  depends only on the cohomology classes of~\(a\) and~\(b\).
\end{lemma}

\begin{proof}
  We have to show that \(f(a\cuptwo b)\) vanishes
  if one cocycle is a coboundary. If \(a=dc\), then
  \begin{equation}
    a\cuptwo b = d\bigl(c\cuptwo b\bigr)
    - a\cupone b - (-1)^{\deg{a}\deg{b}}\,b\cupone a    
  \end{equation}
  maps to~\(0\in B\) since \(f\) vanishes on coboundaries and on \(\cupone\)-products.
  The same argument works for~\(b\).
\end{proof}

\subsection{Extended hgas as shc algebras}

We will need the following result.

\begin{theorem}
  \label{thm:shc-mod-k}
  \let\kkk\aaa
  Let \(A\) be an extended hga, and let \(\kkk\lhd A\) be the ideal generated
  by the values of all operations~\(E_{k}\) with~\(k\ge1\) as well as
  those of all operations~\(F_{kl}\) with~\((k,l)\ne(1,1)\).
  \begin{enumroman}
  \item 
    The extended hga~\(A\) is canonically an shc algebra.
    The structure maps~\(\Phi\),~\(\ha\) and~\(\hc\) commute with morphisms of extended hgas.
  \item
    \label{thm:Phin-k-strict}
    The shm map~\(\Phi\) is \(\kkk\)-strict. More generally,
    all iterations~\(\iter{\Phi}{n}\) with~\(n\ge0\) as well as the composition~\(\Phi\circ(1\otimes\Phi)\) are \(\kkk\)-strict.
  \item
    \label{thm:ha-k-trivial}
    The homotopy~\(\ha\) is \(\kkk\)-trivial.
  \item
    \label{thm:hcn-k}
    Modulo~\(\kkk\), we have for any~\(n\ge0\) and any~\(a_{\bullet}\),~\(b_{\bullet}\in A\) the congruence
    \begin{equation*}
      \hc_{(n)}(a_{\bullet}\otimes b_{\bullet}) \equiv
      \begin{cases}
	1 & \text{if \(n=0\),} \\
        \pm b_{1}\,(a_{1}\cuptwo b_{2})\,a_{2} & \text{if \(n=2\),} \\
	0 & \text{otherwise}.
      \end{cases}
    \end{equation*}
  \end{enumroman}
\end{theorem}

\begin{proof}
  The shc structure is constructed explicitly in the companion paper~\cite{Franz:hgashc}.
  Inspection of the definition of~\(\Phi\) there shows that it is \(\aaa\)-strict.
  The case~\(n=0\) of the iteration is void, and
  for~\(n\ge2\) it is a consequence of \Cref{thm:trivial-h-comp}\,\ref{thm:trivial-h-comp-0} (observed already in \Cref{sec:shc}),
  as is the case of the other composition.
  The statements about~\(\ha\) and~\(\hc\) follow again by looking at their definitions in~\cite{Franz:hgashc}.
\end{proof}

\section{Twisted tensor products}
\label{sec:twisted-tensor}

Let \(A\) be an augmented dga and \(C\) a coaugmented dgc.
For any~\(f\in\Hom(C,A)\) we set
\begin{equation}
  \delta_{f} = (1_{C}\otimes\mu_{A})\,(1_{C}\otimes f\otimes 1_{A})\,(\Delta_{C}\otimes 1_{A})\colon
  C\otimes A\to C\otimes A.
\end{equation}
The assignment
\begin{equation}
  \Hom(C,A) \to \End(C\otimes A),
  \qquad
  f \mapsto \delta_{f}
\end{equation}
is a morphism of dgas. As a consequence,
if \(t\in\Hom(C,A)\) is a twisting cochain, then
\begin{equation}
  \label{eq:def-twisted-diff}
  d_{\otimes}-\delta_{t} = \bigl( d_{C}\otimes1_{A}+ 1_{C}\otimes d_{A} \bigr) - \delta_{t}
\end{equation}
is a differential on~\(C\otimes A\). The resulting complex is called
a \newterm{twisted tensor product} and denoted by
\( 
  C\otimes_{t} A
\), 
compare \cite[Def.~II.1.4]{HusemollerMooreStasheff:1974}
or~\cite[Sec.~1.3]{Huebschmann:1989}.

\begin{lemma}
  \label{thm:homotop-tw}
  Let \(\aaa\lhd A\), and
  let \(h\colon C\to A\) be an \(\aaa\)-trivial homotopy from the twisting cochain~\(t\colon C\to A\)
  to~\(\tilde t\colon C\to A\). If \(C\) is cocomplete, then the map
  \begin{equation*}
    \delta_{h}\colon C\otimes_{\tilde t} A\to C\otimes_{t}A
  \end{equation*}
  is an isomorphism of complexes, congruent to the identity map modulo~\(C\otimes\aaa\).
\end{lemma}

\begin{proof}
  The inverse of~\(\delta_{h}\) is given by~\(\delta_{h^{-1}}\), see~\cite[Cor.~1.4.2]{Huebschmann:1989}.
  The congruence to the identity map follows directly from the \(\aaa\)-triviality.
\end{proof}

\begin{lemma}
  \label{thm:quiso-B-1-f-1}
  Let \(t\colon C\to A\) be a twisting cochain, and let \(g\colon C'\to C\) be a map of coaugmented dgcs.
  Then \(t\circ g\colon C'\to A\) is a twisting cochain and
  \begin{equation*}
    g\otimes 1_{A}\colon C'\otimes_{t\circ g} A\to C\otimes_{t}A
  \end{equation*}
  is a chain map.
\end{lemma}

\begin{proof}
  This follows directly from the definitions.
\end{proof}

\begin{lemma}
  \label{thm:homotopy-twisted}
  Let \(t\colon C\to A\) and~\(t'\colon C'\to A'\) be twisting cochains.
  \begin{enumroman}
  \item
    \label{thm:homotopy-twisted-1}
    Let \(f\colon A'\to A\) be a map of augmented dgas and \(g\colon C'\to C\) a map of coaugmented dgcs.
    If \(t\,g=f\,t'\), then 
    \begin{equation*}
      g\otimes f\colon C'\otimes_{t'}A'\to C\otimes_{t}A
    \end{equation*}
    is a chain map.
  \item
    \label{thm:homotopy-twisted-2}
    Let \(h\colon C'\to C\) be a coalgebra homotopy from~\(g\)
    to another map~\(\tilde g\colon C'\to C\) of coaugmented dgcs satisfying \(t\,\tilde g=f\,t'\).
    If \(t\,h=0\), then \(h\otimes f\) is a homotopy from~\(g\otimes f\) to~\(\tilde g\otimes f\).
  \end{enumroman}
\end{lemma}

\begin{proof}
  The first claim is again a direct consequence of the definitions. The second one follows
  from the identity \(\delta_{t}\,(h\otimes f)=-(h\otimes f)\,\delta_{t'}\), which uses the assumption~\(t\,h=0\).
\end{proof}

Let \(f\colon A\to B\) be a map of augmented dgas. Then \(f\circ t_{A}\colon \BB A\to B\) is a twisting cochain.
The associated twisted tensor product
\begin{equation}
  \label{eq:def-bar-one-sided}
  \BB(\kk,A,B) = \BB A\otimes_{f\circ t_{A}}B
\end{equation}
is the \newterm{one-sided bar construction}. Usually, the map~\(f\) will be understood from the context and not indicated.
We write the cohomology of the one-sided bar construction as
the differential torsion product
\begin{equation}
  \label{eq:def-dg-Tor}
  \Tor_{A}(\kk,B) = H^{*}\bigl(\Bar{B}{A}\bigr).
\end{equation}
Note that this is just a notation;
we are not concerned with whether the bar construction leads to a proper projective resolution in case \(\kk\) is not a field.
However, if \(A\) is free over~\(\kk\) and both \(A\) and~\(B\) have zero differentials, then \eqref{eq:def-dg-Tor} is the usual torsion product.

Given an shm map~\(g\colon B\Rightarrow B'\), we define
\begin{gather}
  \Gamma_{\!g}\colon \BB(\kk,A,B) = \BB A\otimes_{t_{A}}B \to \BB A\otimes_{g\circ t_{A}}B', \\
  \Gamma_{\!g}\bigl( [a_{1}|\dots|a_{k}]\otimes b\bigr) =
  \sum_{m=0}^{k} 
  [a_{1}|\dots|a_{m}]\otimes \ggg([a_{m+1}|\dots|a_{k}]\otimes b)
\end{gather}
where for any~\(k\ge0\) the map~\(\ggg\) of degree~\(0\) is defined as the composition
\begin{equation}
  \ggg\colon \BB_{k} B\otimes B \xrightarrow{1^{\otimes k}\otimes\desusp} \BB_{k+1} B \stackrel{g}{\longrightarrow} B'.
\end{equation}

The following is essentially taken from~\cite[Thm.~7]{Wolf:1977},
where also a version of \Cref{thm:quiso-B-1-f-1} for two-sided bar constructions is given.

\begin{lemma}
  \label{thm:quiso-B-1-1-g}
  Assume that \(g\colon B\Rightarrow B'\) is \(\bbb\)-strict for some~\(\bbb\lhd B'\). Then
  \(\Gamma_{\!g}\) as defined above
  is a chain map, congruent to~\(1_{\BB A}\otimes g_{(1)}\) modulo~\(\BB A\otimes\bbb\).
\end{lemma}

\begin{proof}
  This is a direct computation.
\end{proof}

\begin{remark}
  \label{rem:tw-tensor-quiso}
  Assume that all complexes involved are torsion-free over the principal ideal domain~\(\kk\)
  and (including the bar constructions) bounded below. If the map~\(g\) is a quasi-isomorphism,
  then the resulting maps in Lemmas~\ref{thm:quiso-B-1-f-1} and~\ref{thm:quiso-B-1-1-g} are quasi-isomorphisms.
  This follows from the Künneth theorem and a standard spectral sequence argument,
  compare the proof of \Cref{thm:Theta-additive}\,\ref{thm:Theta-additive-1} below.
\end{remark}

Assume now that \(A\to A'\) is a morphism of hgas.
It is convenient to introduce the map
\begin{align}
  \EEE\colon A'\otimes \BB A' &\to A', \\
  \notag a\otimes \bb &\mapsto
  \EE([\bar a],\bb) + \epsilon(a)\,\epsilon(\bb) =
  \begin{cases}
    a & \text{if \(\bb=\BBone_{\BB A'}\),} \\
    \EE([\bar a],\bb) & \text{if \(\bb\in\BB_{>0}A'\)}
  \end{cases}
\end{align}
of degree~\(0\). Following \cite{KadeishviliSaneblidze:2005}, we can then define the map
\begin{gather}
  \label{eq:def-prod-one-sided-bar}
  \circ\colon \BB(\kk,A,A')\otimes\BB(\kk,A,A') \to \BB(\kk,A,A'), \\
  \notag
  (\BarEl{a}{\aa}) \circ (\BarEl{b}{\bb})
    \eqKS \sum_{m=0}^{l} 
    \BarEl{\EEE(a;[b_{m+1},\dots,b_{l}])\,b}{\bigl(\aa\circ[b_{1}|\dots|b_{m}]\bigr)}
\end{gather}
where \(\aa=[a_{1}|\dots|a_{k}]\),~\(\bb=[b_{1}|\dots|b_{l}]\in \BB A\) and~\(a\),~\(b\in A'\).
Observe that the summand for~\(m=l\) is the componentwise product
\begin{equation}
  (-1)^{\deg{a}\deg{\bb}}\,\aa\circ\bb \otimes a\,b.
\end{equation}

\begin{proposition}[Kadeishvili--Saneblidze]
  \label{thm:def-prod-bar}
  Assume the notation introduced above.
  Then \(\Bar{A'}{A}\) is naturally an augmented dga
  with unit~\(1_{\BB A}\otimes1_{A'}\), augmentation~\(\epsilon_{\BB A}\otimes\epsilon_{A'}\)
  and product~\eqref{eq:def-prod-one-sided-bar}.
\end{proposition}

\begin{proof}
  In~\cite[Cor.~6.2,~7.2]{KadeishviliSaneblidze:2005} this is only stated for simply connected hgas.\footnote{%
  Note that the definition of an hga in~\cite[Def.~7.1]{KadeishviliSaneblidze:2005}
  uses Baues' convention (see \Cref{fn:baues}) and differs from ours
  (as does the definition of the differential on the bar construction \cite[p.~208]{KadeishviliSaneblidze:2005}).
  This results in a product on the one-sided bar construction~\(\BB(A',A,\kk)\).}
  It is, however, a formal consequence of the defining properties of any hga.
\end{proof}

\section{Simplicial sets}

Our basic reference for this material is \cite{May:1968}.
We write \(\setzero{n}=\{0,1,\dots,n\}\).

\subsection{Preliminaries}
\label{sec:simp-prelim}

Let \(X\) be a simplicial set.
We call \(X\) \newterm{reduced} if \(X_{0}\) is a singleton
and \newterm{\(1\)-reduced} if \(X_{1}\) is a singleton.
We abbreviate repeated face and degeneracy operators as
\begin{equation}
  \partial_{i}^{j} = \partial_{i}\circ\dots\circ\partial_{j},
  \qquad
  \partial_{i}^{i-1} = \operatorname{id},
  \qquad
  s_{I} = s_{i_{m}}\circ\dots\circ s_{i_{1}},
  \qquad
  s_{\emptyset} = \operatorname{id}
\end{equation}
for~\(i\le j\) and~\(I=\{i_{1}<\dots<i_{m}\}\).

We write \(C(X)\) and \(C^{*}(X)\) for the normalized
chain and cochain complex of~\(X\) with coefficients in~\(\kk\),
\cf~\cite[Sec.~VIII.6]{MacLane:1967}.
Then \(C(X)\) is a dgc with the Alexander--Whitney map as diagonal
and augmentation induced by the unique map~\(X\to *\), and \(C^{*}(X)\) is a dga with product~\(\transpp{\Delta_{C(X)}}\).

We say \(X\) has \newterm{polynomial cohomology}
(with respect to the chosen coefficient ring~\(\kk\))
if \(H^{*}(X)\) is a polynomial algebra on finitely many generators
of positive even degrees.
Note that \(X\) is of finite type over~\(\kk\) in this case.

For~\(0\le k\le n\) we define the ``partial diagonal''
\begin{align}
  \label{eq:def-Pnk}
  P^{n}_{k}\colon C_{n}(X) &\to C_{k}(X)\otimes C_{n-k}(X), \\*
  \notag \sigma &\mapsto \partial_{k+1}^{n}\sigma  \otimes \partial_{0}^{k-1}\sigma
  = \sigma(0,\dots,k)\otimes\sigma(k,\dots,n)
\end{align}
so that
\begin{equation}
  \label{eq:Pnk-diag}
  \Delta c = \sum_{k=0}^{n}P^{n}_{k}(c)
\end{equation}
for any~\(c\in C_{n}(X)\).
Note that each~\(P^{n}_{k}\) is well-defined on normalized chains.

For simplicial sets~\(X\) and~\(Y\), the shuffle map
\begin{equation}
  \shuffle=\shuffle_{X,Y}\colon C(X)\otimes C(Y)\to C(X\times Y)
\end{equation}
is a map of dgcs and also associative and commutative, \cf~\cite[Sec.~3.2]{Franz:2003}.
Commutativity for instance means that the diagram
\begin{equation}
  \begin{tikzcd}
    C(X)\otimes C(Y) \arrow{d}[left]{T_{C(X),C(Y)}} \arrow{r}{\shuffle_{X,Y}} & C(X\times Y) \arrow{d}{\tau_{X,Y}} \\
    C(Y)\otimes C(X) \arrow{r}{\shuffle_{Y,X}} & C(Y\times X)    
  \end{tikzcd}
\end{equation}
commutes,
where \(\tau_{X,Y}\colon X\times Y\to Y\times X\) swaps the factors.

\subsection{The extended hga structure on cochains}
\label{sec:cochains-hga}

Gerstenhaber and Voronov~\cite[Sec.~2.3]{GerstenhaberVoronov:1995}
have constructed an hga structure on the non-normalized cochain complex of a simplicial set~\(X\),
which descends to the normalized cochain complex~\(C^{*}(X)\).
There it can be given in terms of the interval cut operations 
\begin{equation}
  \label{eq:def-Ek-cochains}
  E_{k} = \transpp{\AWu{e_{k}}}
\end{equation}
corresponding to the surjections
\begin{equation}
  \label{eq:surjection-ek}
  e_{k} = (1,2,1,3,1,\dots,1,k+1,1),
\end{equation}
\cf~\cite[\S 1.6.6, Sec.~2]{BergerFresse:2004}.
Writing out the sign implicit in the transpose~\eqref{eq:def-Ek-cochains}, we have
\begin{equation}
  E_{k}(a;b_{1},\dots,b_{k})(c) = (-1)^{k\,(\deg{a}+\deg{b_{1}}+\dots+\deg{b_{k}})}\,(a\otimes b_{1}\otimes\dots\otimes b_{k})\,\AWu{e_{k}}(c).
\end{equation}
for~\(a\),~\(b_{\bullet}\in C^{*}(X)\) and~\(c\in C(X)\).

The operations~\(E_{k}\) vanish for~\(k\ge1\) if any argument is of degree~\(0\)
and never return a non-zero cochain of degree~\(0\).
This implies that the normalization condition~\eqref{eq:tw-normalization}
is satisfied independently of the chosen augmentation~\(C^{*}(X)\to\kk\).
This hga structure generalizes the multiplication on~\(B C^{*}(X)\) previously defined by Baues~\cite[\S IV.2]{Baues:1980}
for \(1\)-reduced~\(X\).\footnote{\label{fn:baues}
  More precisely, Baues' multiplication is obtained by transposing the factors of the product,
  so that \(\EE_{kl}\) vanishes for~\(l\ne1\), except for~\(\EE_{10}\).
  This also affects the components of the homotopy~\(\FF\)
  from~\citehgashc{Cor.~6.2}.}

Kadeishvili~\cite{Kadeishvili:2003} has observed that \(C^{*}(X)\)
is an extended hga with operations~\(F_{kl}\)
corresponding to the surjections
\begin{align}
  \label{eq:surjection-fkl}
  f_{kl} &= (k+1,1,k+1,2,k+1,\dots,k+1,k, \\ \notag &\qquad\quad k+1,k,k+2,k,\dots,k,k+l,k)
\end{align}
for~\(k\),~\(l\ge1\).\footnote{\label{fn:kadeishvili1}%
  Kadeishvili's choice for~\(f_{kl}\) \cite[pp.~116,~123]{Kadeishvili:2003}
  does not lead to the formula~\eqref{eq:d-Fkl} (or~\cite[Def.~2]{Kadeishvili:2003})
  for~\(d(F_{kl})\), but to the one with \(a\)-variables and \(b\)-variables interchanged.}
The associated \(\cuptwo\)-product is \(\cuptwo=-\transpp{\AWu{(2,1,2,1)}}\).

\subsection{Simplicial groups}

Let \(G\) be a simplicial group (for example, the singular simplices in a topological group) with multiplication~\(\mu\).
We write \(1_{p}\in G\) for the identity element of the group of \(p\)-simplices.
A \newterm{loop} in~\(G\) is a \(1\)-simplex~\(g\in G\) such that \(\partial_{0}g=\partial_{1}g=1_{0}\).

The dgc~\(C(G)\) is a dg~bialgebra
with unit given by the identity element of~\(G\) and multiplication
\begin{equation}
  C(G)\otimes C(G) \stackrel{\shuffle}{\longrightarrow} C(G\times G) \stackrel{\mu_{*}}{\longrightarrow} C(G).
\end{equation}
If \(G\) is commutative, then so is \(C(G)\).

Similarly, if \(G\) acts on the simplicial set~\(X\), then \(C(G)\) acts on \(C(X)\).
We write this action as \(a\ast c\) for~\(a\in C(G)\) and~\(c\in C(X)\).
If the \(G\)-action is trivial, then the \(C(G)\)-action factors through the augmentation~\(\epsilon\colon C(G)\to\kk\).
(Remember that we use normalized chains.)

For any loop~\(g\in G\) and any~\(0\le m\le n\) we define the map
\begin{equation}
  \label{eq:def-agm}
  \Ac{g}{m}\colon C_{n}(X)\to C_{n+1}(X),
  \qquad
  \sigma \mapsto 
  (s_{\setzero{n}\setminus m}\, g)\cdot s_{m}\sigma
\end{equation}
(which is again well-defined on normalized chains).
By the definition of the shuffle map we can write the action of the loop~\(g\in C(G)\) on~\(\sigma \in C(X)\) as
\begin{equation}
  \label{eq:shuffle-agm}
  g \ast \sigma  = \sum_{m=0}^{n} (-1)^{m}\,\ac{g}{m}{\sigma}.
\end{equation}

The diagonal of~\(C(X)\) is known to be \(C(G)\)-equivariant, \cf~\cite[Prop.~3.5]{Franz:2003}.
For loops, a more refined statement is the following.

\begin{lemma}
  \label{thm:shuffle-equiv-partial}
  Assume that \(g\in G\) is a loop, and let \(\sigma\in X_{n}\).
  Then
  \begin{equation*}
    P^{n+1}_{k}\bigl(\ac{g}{m}{\sigma}\bigr) =
    \begin{cases}
      (-1)^{k}\,(1\otimes \Ac{g}{m-k})\,P^{n}_{k}(\sigma) & \text{if \(k\le m\),} \\
      (\Ac{g}{m}\otimes 1)\,P^{n}_{k-1}(\sigma) & \text{if \(k>m\).}
    \end{cases}
  \end{equation*}
  for any~\(0\le m\le n\) and~\(0\le k\le n+1\).
\end{lemma}

\begin{proof}
  We have
  \begin{align}
    & P^{n+1}_{k}(\ac{g}{m}{\sigma}) = \partial_{k+1}^{n+1}\,\ac{g}{m}{\sigma} \otimes \partial_{0}^{k-1}\,\ac{g}{m}{\sigma} \\
    \notag &\qquad = \bigl(\partial_{k+1}^{n+1}\,s_{\setzero{n}\setminus m}\,g\bigr)\cdot\bigl(\partial_{k+1}^{n+1}\,s_{m}\,\sigma\bigr)
    \otimes \bigl(\partial_{0}^{k-1}\,s_{\setzero{n}\setminus m}\,g\bigr)\cdot\bigl(\partial_{0}^{k-1}\,s_{m}\,\sigma\bigr).
  \end{align}
  
  If \(k\le m\), then
  \begin{equation}
    \partial_{k+1}^{n+1}\,s_{\setzero{n}\setminus m}\,g = 1\in G_{k},
  \end{equation}
  hence
  \begin{align}
    P^{n+1}_{k}(\ac{g}{m}{\sigma}) &= \partial_{k+1}^{n+1}\,s_{m}\,\sigma \otimes \bigl(\partial_{0}^{k-1}\,s_{\setzero{n}\setminus m}\,g\bigr)\cdot\bigl(\partial_{0}^{k-1}\,s_{m}\,\sigma\bigr) \\
    \notag &= \partial_{k+1}^{n}\sigma \otimes \bigl(s_{\setzero{n-k}\setminus m-k}\,g\bigr)\cdot\bigl(\partial_{0}^{k-1}\,s_{m}\,\sigma\bigr) \\
    \notag &= (-1)^{k}\,(1\otimes \Ac{g}{m-k})\,P^{n}_{k}(\sigma).
  \end{align}
  
  In the case~\(k>m\) we similarly find
  \begin{equation}
    \partial_{0}^{k-1}\,s_{\setzero{n-1}\setminus m}\,g = 1\in G_{n-k}
  \end{equation}
  and
  \begin{align}
    P^{n+1}_{k}(\ac{g}{m}{\sigma}) &= \bigl(\partial_{k+1}^{n+1}\,s_{\setzero{n}\setminus m}\,g\bigr)\cdot\bigl(\partial_{k+1}^{n+1}\,s_{m}\,\sigma\bigr) \otimes \partial_{0}^{k-1}\,s_{m}\,\sigma \\
    \notag &= \bigl(s_{\setzero{k}\setminus m}\,g\bigr)\cdot\bigl(s_{m}\,\partial_{k}^{n}\sigma\bigr) \otimes \partial_{0}^{k-1}\,s_{m}\,\sigma \\
    \notag &= (\Ac{g}{m}\otimes 1)\,P^{n}_{k-1}(\sigma),
  \end{align}
  as claimed.
\end{proof}

\subsection{Universal bundles}
\label{sec:universal-bundles}

The standard reference for this material is \cite[\S 21]{May:1968},
where the notation~\(BG=\xbar{W}(G)\) and~\(EG=W(G)\) is used.

Let \(G\) be a simplicial group.
Its classifying space is the simplicial set~\(BG\) whose \(p\)-simplices
are elements of the Cartesian product
\begin{equation}
  [g_{p-1},\dots,g_{0}] \in G_{p-1}\times \dots \times G_{0} = BG_{p}.
\end{equation}
It is always reduced (with unique vertex~\(b_{0}\coloneqq[]\in BG_{0}\))
and \(1\)-reduced in case \(G\) is reduced.
The simplices in the total space of the universal \(G\)-bundle~\(\pi\colon EG\to BG\) are similarly given by
\begin{equation}
  e = \bigl(g_{p},[g_{p-1},\dots,g_{0}]\bigr)\in G_{p}\times BG_{p} = EG_{p}\,;
\end{equation}
the map~\(\pi\) is the obvious projection.
We write \(e_{0}=(1_{0},b_{0})\in EG_{0}\) for the canonical basepoint, which projects onto~\(b_{0}\).
Our conventions for face and degeneracy maps can be obtained from~\cite[pp.~71,~87]{May:1968}
by substituting the opposite group~\(G^{\mathrm{op}}\) for~\(G\).
More precisely, for~\(EG\) they are given by
\begin{align}
  \label{eq:d-k-EG}
  \MoveEqLeft{\partial_{k}\bigl(g_{p},[g_{p-1},\dots,g_{0}]\bigr) = {}} \\
  \notag & \bigl(\partial_{k} g_{p},\bigl[\partial_{k-1} g_{p-1},\dots,\partial_{1}g_{p-k+1},(\partial_{0}g_{p-k})g_{p-k-1},g_{p-k-2},\dots,g_{0}\bigr]\bigr), \\
  \MoveEqLeft{s_{k}\bigl(g_{p},[g_{p-1},\dots,g_{0}]\bigr) = {}} \\
  \notag & \bigl(s_{k} g_{p},\bigl[s_{k-1} g_{p-1},\dots,s_{0}g_{p-k},1_{p-k},g_{p-k-1},g_{p-k-2},\dots,g_{0}\bigr]\bigr)
\end{align}
for~\(0\le k\le p\); for~\(BG\) one drops the first component.
Note that for~\(k=0\) the right-hand side of formula~\eqref{eq:d-k-EG} is interpreted as
\(((\partial_{0}g_{p})g_{p-1},[g_{p-2},\dots,g_{0}])\) and
for~\(k=p\) as \((\partial_{p}g_{p},[\partial_{p-1}g_{p-1},\dots,\partial_{1}g_{1}])\).
We consider \(EG\) as a left \(G\)-space via
\begin{equation}
  h\cdot \bigl(g_{p},[g_{p-1},\dots,g_{0}]\bigr)
  = \bigl(h g_{p},[g_{p-1},\dots,g_{0}]\bigr)
\end{equation}
for~\(h\in G_{p}\).

There is a canonical map~\(S\colon EG\to EG\) of degree~\(1\) given by
\begin{equation}
  \label{eq:def-S}
  S\bigl(g_{p},[g_{p-1},\dots,g_{0}]\bigr) = \bigl(1_{p+1},[g_{p},g_{p-1},\dots,g_{0}]\bigr),
\end{equation}
\cf~\cite[p.~88]{May:1968}.
For all~\(e\in EG_{p}\) one has
\begin{alignat}{2}
  \label{eq:d-S-1}
  \partial_{0} S e &= e, \\
  \label{eq:d-S-2}
  \partial_{1} S e &= e_{0} &\qquad& \text{if~\(p=0\),} \\
  \label{eq:d-S-3}
  \partial_{k} S e &= S \partial_{k-1} e &\qquad& \text{if \(p>0\) and \(k>0\).}
\end{alignat}
This implies that \(S\) induces a chain homotopy on~\(C(EG)\), again called \(S\),
from the projection to~\(e_{0}\) to the identity on~\(EG\),
\begin{equation}
  \label{eq:homotopy-S}
  (d S + S d)(e) =
  \begin{cases}
    e-e_{0} & \text{if \(p=0\),} \\
    e &  \text{if \(p>0\),}
  \end{cases}
\end{equation}
for any simplex~\(e\in EG\),
and that it additionally satisfies
\begin{equation}
  \label{eq:properties-S}
  S S = 0
  \qquad\text{and}\qquad
  S e_0 = 0,
\end{equation}
compare~\cite[Prop.~2.7.1]{Franz:2001} or~\cite[Sec.~3.7]{Franz:2003}.

\goodbreak

\begin{lemma}
  \label{thm:P-S-sigma}
  Let \(c\in C(EG)\) be of degree~\(n\).
  \begin{enumroman}
  \item
    \label{thm:P-S-sigma-1}
    For any~\(0\le k\le n+1\) one has
    \begin{equation*}
      P^{n+1}_{k}(Sc) =
      \begin{cases}
	e_{0}\otimes Sc  & \text{if \(k=0\),} \\
          (S\otimes 1)\,P^{n}_{k-1}(c) & \text{if \(k>0\).}
      \end{cases}
    \end{equation*}
  \item
  \label{thm:diag-S-sigma}
    One has
    \begin{equation*}
      \Delta\,Sc = (S\otimes 1)\,\Delta c + e_{0}\otimes S c.
    \end{equation*}
  \end{enumroman}
\end{lemma}

\begin{proof}
  The first statement is immediate if \(c=e_{0}\) or~\(k=0\). For~\(n>0\) and~\(k>0\)
  it follows from the identities~\eqref{eq:d-S-1}--\eqref{eq:d-S-3}.
  Combining it with~\eqref{eq:Pnk-diag} gives the second claim,
  \cf~\cite[Prop.~2.7.1]{Franz:2001} or~\cite[Prop.~3.8]{Franz:2003}.
\end{proof}

\subsection{An Eilenberg\texorpdfstring{--}{-}Moore theorem}
\label{sec:bundles}

In this section we assume that \(\kk\) is a principal ideal domain.

The following result is suggested by work of Kadeishvili--Saneblidze~\cite[Cor.~6.2]{KadeishviliSaneblidze:2005}.

\begin{proposition}
  \label{thm:Eilenberg-Moore}
  Let \(F\stackrel{\iota}{\hookrightarrow}E\to B\) be a simplicial fibre bundle.
  If \(B\) is \(1\)-reduced and of finite type over~\(\kk\), then the map
  \begin{equation}
    \Bar{C^{*}(E)}{C^{*}(B)} \to C^{*}(F),
    \quad
    \BarEl{\gamma}{[\gamma_{1}|\dots|\gamma_{k}]} \mapsto
    \begin{cases}
      \iota^{*}(\gamma) & \text{if \(k=0\),} \\
      0 & \text{otherwise}
    \end{cases}
  \end{equation}
  is a quasi-isomorphism of dgas. In particular, there is an isomorphism of graded algebras
  \begin{equation*}
    H^{*}(F) \cong \Tor_{C^{*}(B)}\bigl(\kk,C^{*}(E)\bigr).
  \end{equation*}  
\end{proposition}

\begin{proof}
  By the usual Eilenberg--Moore theorem, the map is a quasi-isomorphism
  of complexes. For field coefficients, we can refer to~\cite[Thm.~3.2]{Smith:1967}.
  For general~\(\kk\), it follows by dualizing the homological quasi-isomorphism
  \cite[Sec.~6]{Gugenheim:1972}
  \begin{equation}
    C(F) \to \OM\bigl(\kk,C(B),C(E)\bigr)
  \end{equation}
  where the target is the one-sided cobar construction.

  Let us recall the argument:
  If we write \(G\) for the structure group of the bundle~\(E\to B\),
  then \(C(F)\) is a left \(C(G)\)-module.
  By the twisted Eilenberg--Zilber theorem
  \cite[Sec.~4]{Gugenheim:1972}, 
  there is a twisting cochain~\(t\colon C(B)\to C(G)\)
  and a homotopy equivalence
  \begin{equation}
    C(E) \simeq C(B)\otimes_{t}C(F)
  \end{equation}
  of left \(C(B)\)-comodules. Under this isomorphism,
  the map~\(\iota_{*}\colon C(F)\to C(E)\) corresponds
  to the canonical inclusion of~\(C(F)\) into the twisted tensor product
  with the unique base point of~\(B\) as first factor.

  We therefore get a homotopy equivalence of complexes
  \begin{align}
    \OM\bigl(\kk,C(B),C(E)\bigr) &\simeq \OM\bigl(\kk,C(B),C(B)\otimes_{t}C(F)\bigr) \\
    \notag &= \OM\bigl(\kk,C(B),C(B)\bigr) \otimes_{t} C(F)
  \end{align}
  between the one-sided cobar constructions, where we consider
  \(\OM(\kk,C(B),C(B))\) as a right \(C(B)\)-comodule,
  \cf~\cite[Def.~II.5.1]{HusemollerMooreStasheff:1974}.

  The canonical inclusion~\(\kk\hookrightarrow \OM(\kk,C(B),C(B))\)
  is a homotopy equivalence \cite[Prop.~II.5.2]{HusemollerMooreStasheff:1974},
  and \(\delta_{t}\) vanishes on its image. Because \(B\) is \(1\)-reduced,
  a spectral sequence argument shows
  that the map~\(C(F) \to \OM(\kk,C(B),C(B)) \otimes_{t} C(F)\) is a quasi-isomorphism
  of complexes,
  hence so is the natural map
  \begin{equation}
    C(F) \to \OM\bigl(\kk,C(B),C(E)\bigr),
    \qquad
    c \mapsto 1\otimes 1_{\OM C(B)}\otimes\iota_{*}(c).
  \end{equation}
  
  Everything we have said so far is valid for any coefficient ring.
  Since \(\kk\) is a principal ideal domain and \(B\) of finite type, the canonical map
  \begin{equation}
    (\desusp\bar C^{*}(B))^{\otimes k} \otimes C^{*}(E) \to
    \bigl( \kk \otimes (\desusp\bar C(B))^{\otimes k}\otimes C(E) \bigr)^{*}
  \end{equation}
  is a quasi-isomorphisms for any~\(k\ge0\) by the universal coefficient theorem,
  hence so is the composition
  \begin{equation}
    \bigBar{C^{*}(E)}{C^{*}(B)} \to \OM\bigl(\kk,C(B),C(E)\bigr)^{*} \to C^{*}(F).
  \end{equation}
  
  A look at \Cref{thm:def-prod-bar} finally shows
  that the quasi-isomorphism is multiplicative because any cochain on~\(B\)
  of positive degree restricts to~\(0\) on~\(F\). (Recall that we are working with
  normalized cochains.)
\end{proof}

If we define an increasing filtration on~\(\Bar{C^{*}(E)}{C^{*}(B)}\)
by the length of elements, then we get an (Eilenberg--Moore) spectral sequence
of algebras converging to~\(H^{*}(F)\) because the deformation terms
in the product formula given in \Cref{thm:def-prod-bar} lower the filtration degree.
By the Künneth theorem, the second page of this spectral sequence is of the form
\begin{equation}
  E_{2} = \Tor_{H^{*}(B)}\bigl(\kk,H^{*}(E)\bigr)
\end{equation}
with the usual product on~\(\Tor\),
provided that \(H^{*}(B)\) is free over~\(\kk\).

\begin{remark}
  Assume that the base~\(B\) has polynomial cohomology, say \(H^{*}(B)=\kk[y_{1},\dots,y_{n}]\).
  Let \(b_{1}\),~\dots,~\(b_{n}\in C^{*}(B)\) be representatives of the generators, and let
  \begin{equation}
    \bigwedge(x_{1},\dots,x_{n})
  \end{equation}
  be the exterior algebra on generators~\(x_{i}\) of degrees~\(\deg{x_{i}}=\deg{y_{i}}-1\).
  Since \(\BB\,C^{*}(Y)\) is a dg~bialgebra and the elements~\([b_{i}]\in \BB\,C^{*}(Y)\) primitive,
  the assignment
  \begin{equation}
    \bigwedge(x_{1},\dots,x_{n}) \to \BB\,C^{*}(Y),
    \qquad
    x_{i_{1}}\wedge\dots\wedge x_{i_{k}} \mapsto [b_{i_{1}}]\circ\dots\circ[b_{i_{k}}]
  \end{equation}
  is a dgc map (but not multiplicative in general) and in fact a quasi-isomorphism.
  Evaluating the product from the left to the right
  shows that the associated twisting cochain~\(\tGM\) is of the form~\(\tGM(x_{i})=b_{i}\) and
  \begin{align}
    \label{eq:twisiting-GM}
      \tGM(x_{i_{1}}\wedge\dots\wedge x_{i_{k}}) &= E_{1}(\cdots E_{1}(E_{1}(b_{i_{1}};b_{i_{2}});b_{i_{3}}); \cdots ; b_{i_{k}}) \\
      \notag &\quad = (-1)^{k-1}\,\bigl(((b_{i_{1}}\cupone b_{i_{2}})\cupone b_{i_{3}})\cupone\cdots\bigr)\cupone b_{i_{k}}
  \end{align}
  for~\(k\ge2\) and \(i_{1}<\dots<i_{k}\).
  A standard spectral sequence argument then implies that the twisted tensor product
  \begin{equation}
    \bigwedge(x_{1},\dots,x_{n}) \otimes_{\tGM} C^{*}(E)
  \end{equation}
  is quasi-isomorphic to~\(\Bar{C^{*}(E)}{C^{*}(B)}\) as a complex, hence computes \(H^{*}(F)\) as a graded \(\kk\)-module
  by the Eilenberg--Moore theorem. 
  We thus recover the model constructed by Gugenheim--May~\cite[Example~2.2 \& Thm.~3.3]{GugenheimMay:1974}.
\end{remark}

\begin{lemma}
  \label{thm:homog-reduced-groups}
  Let \(G\) be a connected simplicial group and \(K\subset G\) a connected subgroup.
  Write \(\check G\subset G\) for the reduced subgroup of simplices lying over~\(1\in G_{0}\),
  and define \(\check K\subset K\) analogously.
  Then the inclusion~\(\check G/\check K\hookrightarrow G/K\) is a homotopy equivalence,
  natural in the pair~\((G,K)\).
\end{lemma}

\begin{proof}
  The inclusions~\(\check G\hookrightarrow G\) and~\(\check K\hookrightarrow K\) are homotopy equivalences, compare \cite[Thm.~12.5]{May:1968}.
  The long exact sequence of homotopy groups implies that
  the map~\(\check G/\check K\hookrightarrow G/K\) is also a homotopy equivalence.
  Injectivity follows from the identity~\(\check K=K\cap \check G\), and naturality is clear.
\end{proof}

\begin{proposition}
  \label{thm:iso-bar-GK}
  Let \(G\) be a reduced simplicial group and \(K\) a reduced subgroup.
  There is an isomorphism of graded algebras
  \begin{equation*}
    H^{*}(G/K) \cong \Tor_{C^{*}(BG)}\bigl(\kk,C^{*}(BK)\bigr),
  \end{equation*}
  natural with respect to maps of pairs~\((G,K)\).
\end{proposition}

\begin{proof}
  The map~\(\pi\colon EG/K\to BG\) is a fibre bundle with fibre~\(G/K\).
  By \Cref{thm:Eilenberg-Moore}, the dgas~\(C^{*}(G/K)\) and~\(\Bar{C^{*}(EG/K)}{C^{*}(BG)}\) are naturally quasi-iso\-mor\-phic.
  The homotopy equivalence~\(BK=EK/K\to EG/K\) is a map over~\(BG\)
  and induces a quasi-isomorphism
  \begin{equation}
    \bigBar{C^{*}(EG/K)}{C^{*}(BG)}
    \to \bigBar{C^{*}(BK)}{C^{*}(BG)},
  \end{equation}
  which is multiplicative by the naturality of the hga structure on cochains.
\end{proof}

\begin{remark}
  Let \(G\) be a Lie group and \(K\subset G\) a closed subgroup. Then the projection~\(G\mapsto G/K\)
  is a principal \(K\)-bundle. Writing \(S(X)\) for the simplicial set of singular simplices
  in a topological space~\(X\), we therefore have \(S(G/K)=S(G)/S(K)\).
  The same holds if \(G\) is only a topological group, but the closed subgroup~\(K\)
  has the structure of a Lie group, \cf~\cite[Sec.~4.1]{Palais:1961}.
\end{remark}

\section{Homotopy Gerstenhaber formality of~\texorpdfstring{\(BT\)}{BT}}
\label{sec:BT-formality}

\subsection{Dga formality}
\label{sec:BT-formality-dga}

Let \(T\) be a simplicial torus of rank~\(n\). By this we mean a commutative simplicial group~\(T\)
such that \(H(T)\) is an exterior algebra on generators~\(x_{1}\),~\dots,~\(x_{n}\) of degree~\(1\).
For example, \(T\) can be the compact torus~\((S^{1})^{n}\), the algebraic torus~\((\C^{\times})^{n}\)
or the simplicial group~\(B\Z^{n}\).

As mentioned in the introduction,
Gugenheim--May~\cite[Thm.~4.1]{GugenheimMay:1974} have constructed a quasi-isomorphism of dgas
\begin{equation}
  \label{eq:CBT-HBT}
  C^{*}(BT) \to H^{*}(BT)
\end{equation}
annihilating all \(\cupone\)-products. An alternative approach was given by the author
in his doctoral dissertation~\cite[Prop.~2.2]{Franz:2001}, see also~\cite[Prop.~5.3]{Franz:2003}.
The goal of this section is to promote the latter construction to a quasi-isomorphism of hgas,
that is, one that annihilates all operations~\(E_{k}\) with~\(k\ge1\).
We will see that also all operations~\(F_{kl}\) with the exception of the \(\cuptwo\)-product are sent to~\(0\).

We write \(\Ll=H(T)\) and~\(\Sl=H(BT)\). The latter is the cocommutative coalgebra on cogenerators~\(y_{i}\in\Sl_{2}\)
that correspond to the~\(x_{i}\)'s under transgression. The \(y_{i}\)'s define a \(\kk\)-basis~\(y_{\alpha}\) of~\(\Sl\)
index by multi-indices~\(\alpha\in\N^{n}\). We also write \(y_{0}=1\).

Let \(t\colon\Sl\to\Ll\) be the (homological) twisting cochain that sends each~\(y_{i}\) to~\(x_{i}\)
and vanishes in other degrees. The twisted tensor product
\begin{equation}
  \Kl = \Ll\otimes_{t}\Sl
\end{equation}
is the Koszul complex. It is a dgc with \(\Ll\)-equivariant diagonal given by the tensor product of the componentwise diagonals.
For~\(a\in\Ll\) and~\(c\in\Sl\) we write \(a\cdot c\in\Kl\) instead of~\(a\otimes c\), reflecting the \(\Ll\)-action.
The differential on~\(\Kl\) is given by
\begin{equation}
  d(a\cdot y_{\alpha}) = (-1)^{\deg{a}}\sum_{i} a\wedge x_{i}\cdot y_{\alpha|i} = \sum_{i}x_{i}\wedge a \cdot y_{\alpha|i}
\end{equation}
where the sum runs over all~\(i\) such that~\(\alpha_{i}>0\), and ``\(\alpha|i\)'' means that the \(i\)-th component of~\(\alpha\) is decreased by~\(1\).

Let \(c_{1}\),~\dots,~\(c_{n}\in C_{1}(T)\) be linear combinations of loops in~\(G\) representing the generators~\(x_{i}\).
They define a quasi-isomorphism of dg~bialgebras
\begin{equation}
  \phi\colon \Ll\to C(T),
  \qquad
  x_{i_{1}}\wedge\dots\wedge x_{i_{k}} \mapsto c_{i_{1}}*\cdots*c_{i_{k}}
\end{equation}
for~\(i_{1}<\dots<i_{k}\).
Moreover, let \(\pi\colon ET\to BT\) be the universal \(T\)-bundle. Note that \(C(ET)\) is a \(\Ll\)-module via~\(\phi\).

Our map~\eqref{eq:CBT-HBT} will be the transpose of a quasi-isomorphism~\(\ffbar\colon\Sl\to C(BT)\).
The construction of the latter is based on a map
\begin{equation}
  \ff\colon\Kl\to C(ET)
\end{equation}
recursively defined by
\begin{alignat}{2}
  \ff(1) &= e_{0}, \\
  \ff(a\cdot c) &= a \ast \ff(c) &\qquad& \text{if \(\deg{a}>0\)}, \\
  \ff(c) &= S\,\ff(dc) &\qquad& \text{if \(\deg{c}>0\)}
\end{alignat}
for~\(c\in\Sl\) and~\(a\in\Ll\), where \(S\) is the homotopy defined in~\eqref{eq:def-S}.

\begin{proposition} 
  \label{thm:f-coalg}
  The map~\(\ff\) is a \(\Ll\)-equivariant quasi-isomorphism of dgcs.
\end{proposition}

For the convenience of the reader, we adapt the proof given in~\cite[Prop.~4.3]{Franz:2003}
to our slightly more general setting.\footnote{%
Using \cite[eq.~(2.12c)]{Franz:2001} or~\cite[eq.~(3.29a)]{Franz:2003},
one can see that our new construction coincides with the previous one
  if each~\(c_{i}\)
  lies entirely in the \(i\)-th factor of a circle decomposition of~\(T\).}

\begin{proof}
  It is clear from the definition that \(\ff\) commutes
  with the \(\Ll\)-action.  
  To show that it is a chain map, we proceed by induction on the degree of~\(a\cdot y\in\Kl\).
  For \(a\cdot y=1\) this is obvious. For~\(\deg{a}>0\) we have by equivariance and induction
  \begin{align}
    d\,\ff(a\cdot y) &= d\bigl(\phi(a)\ast \ff(y)\bigr) = \phi(da)\ast \ff(y) +(-1)^{\deg{a}}\,\phi(a)\ast d \ff(y) \\
    \notag &= \ff\bigl(da\cdot y + (-1)^{\deg{a}}a\cdot dy\bigr) = \ff\,d(a\cdot y).
  \end{align}
  For~\(\deg{y}>0\) we have by~\eqref{eq:homotopy-S} and induction
  \begin{equation}
    d\, \ff(y) = d\, S\, \ff(dy) = \ff(dy) - S\, d\, \ff(dy) = \ff(dy).
  \end{equation}
  
  To show that \(f\) is a map of coalgebras,
  we proceed once more by induction on~\(\deg{a\cdot y}\),
  the case~\(a\cdot y=1\) being trivial.
  If \(\deg a>0\), then again by equivariance and induction we have
  \begin{align}
    \Delta \ff(a\cdot y) &= \Delta \bigl(\phi(a)\ast \ff(y)\bigr)
    = \Delta \phi(a) \ast \Delta \ff(y)
    = \Delta \phi(a) \ast (\ff\otimes \ff)\Delta y \\
    \notag &= (\ff\otimes \ff) (\Delta a\cdot \Delta y) = (\ff\otimes \ff)\Delta(a\cdot y).
  \end{align}

  For~\(\alpha\ne0\) we therefore have by \Cref{thm:P-S-sigma}\,\ref{thm:diag-S-sigma} that
  \begin{align}
    \Delta\, \ff(y_{\alpha})
    &= \Delta\, S\, \ff(d y_{\alpha})
    = (S\otimes1)\,\Delta\, \ff(d y_{\alpha}) + e_{0}\otimes S\, \ff(d y_{\alpha}) \\
    \notag &= \sum_{i} (S\otimes1)\,\Delta \ff(x_{i}\cdot y_{\alpha|i}) + \ff(1) \otimes \ff(y_{\alpha}),
  \end{align}
  where the sum runs over the indices~\(i\) such that \(\alpha_{i}\ne0\).
  Using again the equivariance of the Alexander--Whitney map and induction, we get
  \begin{align}
    \Delta\, \ff(y_{\alpha})
    &= \sum_{i} (S\otimes1)\,\Delta c_{i}\ast \Delta\, \ff(y_{\alpha|i}) + \ff(1) \otimes \ff(y_{\alpha}) \\
    \notag &= \sum_{i}\sum_{\beta+\gamma=\alpha|i} (S\otimes1)\,\Delta c_{i}\ast\bigl(\ff(y_{\beta})\otimes \ff(y_{\gamma})\bigr) + \ff(1) \otimes \ff(y_{\alpha}).
  \end{align}
  Now \(\Delta c_{i}= c_{i}\otimes1+1\otimes c_{i}\),
  and \(S\, \ff(y_{\gamma})=0\) by~\eqref{eq:properties-S}, hence
  \begin{equation}
    \Delta \ff(y_{\alpha})
    = \sum_{i}\sum_{\beta+\gamma=\alpha|i} S( c_{i}\ast \ff(y_{\beta}))\otimes \ff(y_{\gamma}) + \ff(1) \otimes \ff(y_{\alpha}).
  \end{equation}
  We reorder the summands. For each~\(\gamma\ne\alpha\) whose
  components are all less than or equal to those of~\(\alpha\), we have
  one term of the form~\( c_{i}\ast \ff(y_{\beta|i})\) for each~\(\beta=\alpha-\gamma\)
  and each~\(i\) such that \(\beta_{i}\ne0\). This gives
  \begin{align}
  \Delta\, \ff(y_{\alpha})
    &= \sum_{\substack{\beta+\gamma=\alpha\\\gamma\ne\alpha}}
      \sum_{i} S( c_{i}\ast \ff(y_{\beta|i})\otimes \ff(y_{\gamma}) + \ff(1) \otimes \ff(y_{\alpha}) \\
    \notag &= \sum_{\substack{\beta+\gamma=\alpha\\\gamma\ne\alpha}}
      \ff(y_{\beta})\otimes \ff(y_{\gamma}) + \ff(1) \otimes \ff(y_{\alpha}) \\
    \notag &= \sum_{\beta+\gamma=\alpha} \ff(y_{\beta})\otimes \ff(y_{\gamma}),
  \end{align}
  as was to be shown.

  That \(\ff\) induces an isomorphism in homology is trivial.
\end{proof}

Since \(C(G)\) acts trivially on~\(BT\), the composition~\(\pi_{*}\ff\colon\Kl\to C(BT)\) descends to a map of dgcs
\begin{equation}
  \ffbar\colon \Sl = \kk\otimes_{\Ll}\Kl \to C(BT).
\end{equation}

\begin{proposition}
  \label{thm:quiso-ffbar}
  The transpose~\(\ffbar^{*}\colon C^{*}(BT)\to\Sl^{*}\) is a morphism of dgas
  that induces the identity in cohomology.
\end{proposition}

\begin{proof}
  Being the transpose of a dgc map, \(\ffbar^{*}\) clearly is a morphism of dgas.
  
  Let \(1\le i\le n\).
  By construction, \(\ffbar(y_{i})\)
  corresponds to~\(\phi(x_{i})\) under transgression:
  Let \(\iota\colon T\hookrightarrow ET\) be the inclusion
  of the fibre (over~\(b_{0}\)).
  Then
  \(\iota_{*}(\phi(x_{i})) = d\ff(y_{i})\) and \(\ffbar(y_{i}) = \pi_{*}\ff(y_{i})\).
  This means that \(H(\ffbar)\) the identity map on cogenerators, hence in general.
  By the universal coefficient theorem (or spectral sequence),
  the same conclusion holds for~\(H^{*}(f^{*})\).
\end{proof}

\subsection{Hga formality}

We say that a (non-degenerate) simplex~\(\sigma\in ET\) \newterm{appears} in an element of~\(C(ET)\)
if its coefficient in this chain is non-zero; an analogous definition applies to tensor products of chain complexes.

\begin{lemma}
  \label{thm:S-S-P-Sigma}
  Let \(0\le k\le n+1\), \(a\in\Ll_{1}\) and~\(c\in\Sl_{n}\). For any simplex~\(\sigma\in(ET)_{n+1}\) appearing in~\(\ff(a\cdot c)\) we have
  \begin{equation*}
    (S\otimes S)\,P^{n+1}_{k}(\sigma) = 0.
  \end{equation*}
\end{lemma}

\begin{proof}
  By construction and formula~\eqref{eq:shuffle-agm},
  the simplex~\(\sigma\) is of the form~\(\ac{g}{m}{\tau}\) for some loop~\(g\in T_{1}\),
  some~\(0\le m\le n\) and some \(n\)-simplex~\(\tau\) appearing in~\(\ff(c)\).

  If \(n=0\), then \(\tau=e_{0}\). Hence
  \begin{equation}
    P^{1}_{0}(\sigma) = e_{0}\otimes g \ast e_{0}
    \qquad\text{and}\qquad
    P^{1}_{1}(\sigma) = g \ast e_{0}\otimes e_{0}
  \end{equation}
  by \Cref{thm:shuffle-equiv-partial}, and our claim follows from the second identity in~\eqref{eq:properties-S}.
  
  Now consider the case~\(n>0\). The definition of the map~\(\ff\) implies that \(\tau\) is of the form~\(S\rho\) where
  \(\rho\) is a simplex appearing in~\(\ff(\tilde a\cdot\tilde c)\) with \(\tilde a\in\Ll_{1}\) and~\(\tilde c\in\Sl_{n-2}\).

  Assume \(k\le m\). Then
  \begin{equation}
    P^{n+1}_{k}(\sigma) = P^{n+1}_{k}(\ac{g}{m}{S\rho})
    = (-1)^{k}\,(1\otimes\Ac{g}{m-k})\,P^{n}_{k}(S\rho)
  \end{equation}
  where we have once again used \Cref{thm:shuffle-equiv-partial}.
  In the case~\(k=0\) we obtain
  \begin{equation}
    (S\otimes S)\,P^{n+1}_{0}(\sigma) = Se_{0}\otimes S\,\ac{g}{m}{S\rho} = 0
  \end{equation}
  by \Cref{thm:P-S-sigma}\,\ref{thm:P-S-sigma-1} and~\eqref{eq:properties-S}.
  If \(k>0\), then
  \begin{equation}
    (S\otimes S)\,P^{n+1}_{k}(\sigma) = (-1)^{k}\, (S\,S\otimes S\,\Ac{g}{m-k})\,P^{n-1}_{k-1}(\rho) = 0
  \end{equation}
  again by \Cref{thm:P-S-sigma}\,\ref{thm:P-S-sigma-1} and the first identity in~\eqref{eq:properties-S}.

  In the case~\(k>m\), we have
  \begin{equation}
    P^{n+1}_{k}(\sigma) = (\Ac{g}{m}\otimes 1)\,P^{n}_{k-1}(S\rho).
  \end{equation}
  For~\(k=1\) this gives
  \begin{equation}
    (S\otimes S)\,P^{n+1}_{1}(\sigma) 
    = - S\,\ac{g}{m}{e_{0}}\otimes S\,S\rho = 0.
  \end{equation}
  If \(k>1\), we finally get
  \begin{align}
      (S\otimes S)\,P^{n+1}_{k}(\sigma) &= (S\,\Ac{g}{m}\,S\otimes S)\,P^{n-1}_{k-2}(\rho) \\
      \notag &= (S\,\Ac{g}{m}\otimes 1)\,(S\otimes S)\,P^{n-1}_{k-2}(\rho) = 0
  \end{align}
  by induction.
\end{proof}

For~\(0\le k<l\le n\) we define
\begin{align}
    Q^{n}_{k,l}\colon C_{n}(ET) &\to C_{n-l+k+1}(ET)\otimes C_{l-k}(BT), \\
    \notag \sigma &\mapsto \partial_{k+1}^{l-1}\sigma \otimes \pi_{*}\,\partial_{0}^{k-1}\,\partial_{l+1}^{n}\sigma \\
    \notag &\qquad = \sigma(0,\dots,k,l,\dots,n)\otimes\pi_{*}\sigma(k,\dots,l).
\end{align}
This operation is related to the \(\cupone\)-product since for \(\sigma\in C_{n}(ET)\) we have
\begin{equation}
  \label{eq:AW121-Qnkl}
  (1\otimes\pi_{*})\,\AWu{(1,2,1)}(\sigma) = \sum_{0\le k<l\le n}(-1)^{(n-l)(l-k)+k}\,Q^{n}_{k,l}(\sigma),
\end{equation}
compare \cite[\S 2.2.8]{BergerFresse:2004}.

\begin{lemma}
  \label{eq:Qnkl-sigma}
  Let \(0\le k<l\le n\) and~\(a\cdot c\in\Kl_{n}\). For any \(n\)-simplex~\(\sigma\in ET\) appearing in~\(\ff(a\cdot c)\) we have
  \begin{equation*}
    Q^{n}_{k,l}(\sigma) = 0.
  \end{equation*}
\end{lemma}

\begin{proof}
  We proceed by induction on~\(n\), the case~\(n=0\) being void.
  For the induction step from~\(n\) for~\(n+1\), we start by considering the case~\(\deg{a}=0\), which entails \(n\ge1\).
  The definition of~\(\ff\) then implies that \(\sigma\) is of the form~\(\sigma=S\tau\) 
  for some~\(n\)-simplex \(\tau\in ET\)
  that appears in~\(\ff(\tilde a\cdot\tilde c)\) for some~\(\tilde a\in\Ll_{1}\) and some~\(\tilde c\in\Sl_{n-1}\).

  If \(1\le k<l\le n+1\), we get
  \begin{align}
    Q^{n+1}_{k,l}(\sigma) &= \partial_{k+1}^{l-1} \, S \tau \otimes \pi_{*}\,\partial_{0}^{k-1}\,\partial_{l+1}^{n+1}\, S \tau
    = S \, \partial_{k}^{l-2} \tau \otimes \pi_{*}\,\partial_{0}^{k-1} \,S \, \partial_{l}^{n}\tau \\
    \notag &= S \, \partial_{k}^{l-2} \tau \otimes \pi_{*}\,\partial_{0}^{k-2} \, \partial_{l}^{n}\tau
    = (S\otimes 1)\,Q^{n}_{k-1,l-1}(\tau) = 0
  \end{align}
  by induction.
  
  For~\(0<l\le n+1\) we have
  \begin{align}
    Q^{n+1}_{0,l}(\sigma) &= \partial_{1}^{l-1} \, S \tau \otimes \pi_{*}\,\partial_{l+1}^{n+1}\, S \tau
    = S \, \partial_{0}^{l-2} \tau \otimes \pi_{*}\, \partial_{l+1}^{n+1} S\tau \\
    \notag &= S \, \partial_{0}^{l-1} S\tau \otimes \pi_{*}\, \partial_{l+1}^{n+1} S\tau
    = \pm (1\otimes \pi_{*})\,T\,(1\otimes S)\,P^{n+1}_{l}(S\tau) \\
    \shortintertext{where \(T\) denotes the transposition of factors,}
    \notag &= \mp (1\otimes\pi_{*})\,T\,(S\otimes S)\,P^{n}_{l-1}(\tau) = 0
  \end{align}
  by Lemmas~\ref{thm:P-S-sigma}\,\ref{thm:P-S-sigma-1} and~\ref{thm:S-S-P-Sigma}.

  Now we turn to the case~\(\deg{a}>0\). Then a simplex appearing in~\(\ff(a\cdot c)\) is of the form~\(\sigma=\ac{g}{m}{\tau}\)
  for some loop~\(g\in T_{1}\), some \(n\)-simplex~\(\tau\in ET\) appearing in~\(\ff(c)\) and some~\(0\le m\le n\). We have
  \begin{align}
    \label{eq:Qkl-a}
      Q^{n+1}_{k,l}(\sigma) &= Q^{n+1}_{k,l}\bigl(\ac{g}{m}{\tau}\bigr) = Q^{n+1}_{k,l}\bigl(s_{\setzero{n}\setminus m}\,g\cdot s_{m}\tau\bigr) \\
      \notag &= \partial_{k+1}^{l-1}\bigl(s_{\setzero{n}\setminus m}\,g\cdot s_{m}\tau\bigr) \otimes \pi_{*}\,\partial_{0}^{k-1}\,\partial_{l+1}^{n+1}\,s_{m}\tau.
  \end{align}

  Assume \(l>m\). Then
  \begin{equation}
    \partial_{0}^{k-1}\,\partial_{l+1}^{n+1}\,s_{m}\tau = \partial_{0}^{k-1}\,s_{m}\,\partial_{l}^{n}\tau =
    \begin{cases}
      s_{m-k}\,\partial_{0}^{k-1}\,\partial_{l}^{n}\tau & \text{if~\(k\le m\),} \\
      \partial_{0}^{k-2}\,\partial_{l}^{n}\tau & \text{if~\(k>m\).}
    \end{cases}
  \end{equation}
  In the first case we obtain a degenerate simplex, so that \eqref{eq:Qkl-a} vanishes.
  In the second case we have
  \begin{align}
      Q^{n+1}_{k,l}(\sigma) &= 
      \bigl(s_{\setzero{n-l+k+1}\setminus m}\,g\bigr)\cdot \bigl(s_{m}\,\partial_{k}^{l-2}\tau\bigr) \otimes \pi_{*}\,\partial_{0}^{k-2}\,\partial_{l}^{n}\tau \\
      \notag &= (\Ac{g}{m}\otimes 1)\Bigl(\partial_{k}^{l-2}\tau\otimes\pi_{*}\,\partial_{0}^{k-2}\,\partial_{l}^{n}\tau \Bigr) \\
      \notag &= (\Ac{g}{m}\otimes 1) \, Q^{n}_{k-1,l-1}(\tau) = 0
  \end{align}
  by induction.

  Finally consider \(l\le m\). Then
  \begin{align}
      Q^{n+1}_{k,l}(\sigma) &= 
      \bigl(\partial_{k+1}^{l-1}\,s_{\setzero{n}\setminus m}\,g\bigr)\cdot \bigl(\partial_{k+1}^{l-1}\,s_{m}\tau\bigr) \otimes \pi_{*}\,\partial_{0}^{k-1}\,\partial_{l+1}^{n+1}\,s_{m}\tau \\
      \notag &= \bigl(s_{\setzero{n-l+k+1}\setminus m-l+k+1}g\bigr)\cdot \bigl(s_{m+k-l+1}\,\partial_{k+1}^{l-1}\tau\bigr) \otimes \pi_{*}\,\partial_{0}^{k-1}\,\partial_{l+1}^{n}\tau \\
      \notag &=  \Ac{g}{m-l+k+1}\bigl( \partial_{k+1}^{l-1}\tau \bigr) \otimes \pi_{*}\,\partial_{0}^{k-1}\,\partial_{l+1}^{n}\tau \\
      \notag &= (\Ac{g}{m-l+k+1}\otimes 1)\, Q^{n}_{k,l}(\tau) = 0
  \end{align}
  by induction. This completes the induction step and the proof.
\end{proof}

We write \(\setone{n}=\{1,\dots,n\}\).
We say that a surjection~\(u\colon\setone{k+l}\to\setone{l\vphantom{+}}\) \newterm{has an enclave}
between positions~\(i\) and~\(i'\ge i+2\) if \(u(i)=u(i')\) and if the values~\(u(j)\) for~\(i<j<i'\) do not also appear at positions~\(\le i\) or~\(\ge i'\).
For example, the surjection~\((1,2,3,2,1,4)\) has exactly two enclaves, namely between positions~\(1\) and~\(5\) and between~\(2\) and~\(4\).

\begin{proposition}
  \label{thm:enclave-0}
  If the surjection~\(u\colon\setone{k+l}\to\setone{l\vphantom{+}}\) has an enclave, then
  \begin{equation*}
    \AWu{u}\,f = 0.
  \end{equation*}
\end{proposition}

\begin{proof}
  We start with a general observation. Let \(\sigma\) be a simplex in some simplicial set.
  If \(u\) has an enclave, then
  it follows from the definition of interval cut operations~\cite[Sec.~2.2]{BergerFresse:2004}
  that any tensor product of simplices appearing in~\(\AWu{u}(\sigma)\)
  can be obtained from a term~\(\tau\otimes\rho\) appearing in~\(\AWu{(1,2,1)}(\sigma)\)
  by applying an interval cut to~\(\tau\) (at one choice of positions, not at all positions as in~\cite[\S 2.2.6]{BergerFresse:2004}),
  another one to~\(\rho\) and permuting the factors of the result.

  Now let \(c\in\Sl_{n}\) for some~\(n\ge0\).
  By definition and naturality we have
  \begin{equation}
    \AWu{u}\,\ffbar(c) = \AWu{u}\,\pi_{*}\,\ff(c) = (\pi_{*}\otimes\pi_{*})\,\AWu{u}\,\ff(c).
  \end{equation}
  Our previous remarks together with~\eqref{eq:AW121-Qnkl} show that it suffices to prove
  that \(Q^{n}_{k,l}(\sigma)\) vanishes for any~\(\sigma\in ET\) appearing in~\(\ff(c)\) and any~\(0\le k<l\le n\).
  But this has been done in \Cref{eq:Qnkl-sigma}.
\end{proof}

\begin{theorem}
  \label{thm:ffbar-hga-formal}
  The map~\(\ffbar^{*}\colon C^{*}(BT)\to H^{*}(BT)\) is a quasi-isomorphism of hgas
  that additionally annihilates all extended hga operations~\(F_{kl}\) with~\((k,l)\ne(1,1)\).
  In particular, \(C^{*}(BT)\) is formal as an hga.
\end{theorem}

\begin{proof}
  We know from \Cref{thm:quiso-ffbar} that \(\ffbar^{*}\) is a quasi-isomorphism of dgas.
  The hga operations~\(E_{k}\) with~\(k\ge1\) as well as the operations~\(F_{kl}\) with~\((k,l)\ne(1,1)\)
  are defined by surjections having enclaves, see~\eqref{eq:surjection-ek} and~\eqref{eq:surjection-fkl}.
  Hence the claim follows by dualizing \Cref{thm:enclave-0}.
\end{proof}

\subsection{The case where \texorpdfstring{\(2\)}{2} is invertible}
\label{sec:inverting-2}

  It would greatly simplify the discussion of the next sections if the formality map~\(\ffbar^{*}\)
  also annihilated the operation~\(F_{11}=-\cuptwo\). However, this is impossible to achieve
  for the transpose of a quasi-isomorphism \(f\colon H(BT)\to C(BT)\),
  independently of the coefficient ring~\(\kk\). This can be seen as follows.

  Take a non-zero~\(y\in H_{2}(BT)\) and set \(w=f(y)\in C_{2}(BT)\). Choose a cochain \(a\in C^{2}(BT)\)
  such that \(a(w)\ne0\). Let \(\sigma\) be a \(2\)-simplex appearing in~\(w\) with coefficient~\(w_{\sigma}\ne0\)
  and such that \(a(\sigma)\ne0\).
  Define \(b\in C^{2}(BT)\) by~\(b(\sigma)=1\) and \(b(\tau)=0\) for~\(\tau\ne\sigma\). Then
  \begin{equation}
    (a\cuptwo b)(w) = \sum_{\tau} w_{\tau}\,a(\tau)\,b(\tau) = w_{\sigma}\,a(\sigma) \ne 0,
  \end{equation}
  where we have used the identity~\((a\cuptwo b)(\sigma)=a(\sigma)\,b(\sigma)\), 
  \cf~\cite[\S 2.2.8]{BergerFresse:2004}.
  Hence \(f^{*}(a\cuptwo b)\ne0\), and analogously \(f^{*}(b\cuptwo a)\ne0\).
  Note that \(a\) may be a cocycle, but \(b\) is not.
  (If \(\sigma=[\,g\,|\,1_{0}\,]\) for a loop~\(1_{1}\ne g\in T_{1}\), then \(b\bigl(d\,[\,s_{0}g^{-1}\,|\,g\,|\,1_{0}\,]\bigr)\ne0\).)

  In general one cannot even expect \(\ffbar^{*}\)
  to annihilate all \(\cuptwo\)-products of cocycles as they are related to Steenrod squares.
  For~\(\kk=\Z_{2}\) and any non-zero \([a]\in H^{2}(BT)\) one has
  \begin{equation}
    [a] = \Sq^{0}[a] = [a\cuptwo a]\ne0.
  \end{equation}
The situation changes if we can invert \(2\).

\def\gone{a}
\def\gtwo{b}

\begin{proposition}
  \label{thm:ffbar-cuptwo}
  Assume that \(2\) is invertible in~\(\kk\). Then one can choose representatives~\(( c_{i})\)
  such that
  \(\transpp{\ffbar}\) additionally annihilates all \(\cuptwo\)-products of cocycles.
\end{proposition}

\begin{proof}
  Let \(\iota\colon T\to T\) be the group inversion.
  Being a morphism of groups, it induces involutions of~\(ET\) and~\(BT\),
  which we denote by the same letter.
  Recall that \(\iota_{*}\) changes the sign of all generators~\(x_{i}\in H_{1}(T)\) and
  all cogenerators~\(y_{i}\in H_{2}(BT)\). 
  Starting from any set of representatives~\(( c_{i})\), we set
  \begin{equation}
    \label{eq:def-tildecci-in-proof}
    \tildecc_{i} = \textstyle{\frac{1}{2}}\, c_{i}-\textstyle{\frac{1}{2}}\,\iota_{*} c_{i},
  \end{equation}
  so that \(\iota_{*}\tildecc_{i}=-\tildecc_{i}\).
  We construct \(\ff\) and~\(\ffbar\) based on these representatives.
  The equivariance of~\(\ff\) with respect to the involutions follows inductively from the recursive definition, 
  and it entails that of~\(\ffbar\). 

  Now let \(\gone\) and~\(\gtwo\) be cocycles. By \Cref{thm:cuptwo},
  the value~\(\transpp{\ffbar}(\gone\cuptwo\gtwo)\) only depends
  on the cohomology classes of~\(\gone\) and~\(\gtwo\). In particular,
  we may assume that \(\gone\) is of even degree~\(2k\) and \(\gtwo\) of degree~\(2l\).
  Then \(\gone\cuptwo\gtwo\) is of degree~\(2(k+l-1)\), whence
  \begin{equation}
    \iota^{*}\transpp{\ffbar}(\gone\cuptwo\gtwo)
    = -(-1)^{k+l}\transpp{\ffbar}(\gone\cuptwo\gtwo).
  \end{equation}
  On the other hand, we have
  \begin{equation}
    \iota^{*}(\gone\cuptwo\gtwo)
    =\iota^{*}(\gone)\cuptwo\iota^{*}(\gtwo)
  \end{equation}
  by naturality. Now \(\iota^{*}(\gone)\) is cohomologous to~\((-1)^{k}\,\gone\)
  and \(\iota^{*}(\gtwo)\) cohomologous to~\((-1)^{l}\,\gone\), which implies that
  \begin{equation}
    \transpp{\ffbar}(\iota^{*}(\gone\cuptwo\gtwo))
    = (-1)^{k+l}\,\transpp{\ffbar}(\gone\cuptwo\gtwo).
  \end{equation}
  Since \(2\) is invertible in~\(\kk\), this can only happen if the \(\cuptwo\)-product vanishes.
\end{proof}

\section{The kernel of the formality map}
\label{sec:kernel}

{
\let\alpha a
\let\beta b
\let\gamma c

The kernel of the formality map~\(\transpp{\ffbar}\colon C^{*}(BT)\to H^{*}(BT)\) constructed
in the previous section depends on the choice of representatives. It will be convenient
to consider instead an ideal~\(\ax_{X}\lhd C^{*}(X)\) for any simplicial set~\(X\)
such that \(\ax_{BT}\subset\ker\transpp{\ffbar}\) independently of any choices
and also \(\kappa^{*}(\ax_{Y})\subset\ax_{X}\) for any map~\(\kappa\colon X\to Y\).

Let \(X\) be a simplicial set.
Given elements~\(\beta_{0}\),~\dots,~\(\beta_{k}\in C^{*}(X)\), we write the repeated \(\cupone\)-product as
\begin{equation}
  \RC_{0}(\beta_{0}) = \beta_{0}
  \qquad\text{and}\qquad
  \RC_{k}(\beta_{0},\dots,\beta_{k}) = -\RC_{k-1}(\beta_{0},\dots,\beta_{k-1})\cupone\beta_{k}
\end{equation}
for~\(k\ge1\), compare the Gugenheim--May twisting cochain~\eqref{eq:twisiting-GM}.

We define
\begin{equation}
 \ax=\ax_{X}\lhd C^{*}(X) 
\end{equation}
to be the ideal generated by the following elements
where \(a\),~\(b\),~\(a_{\bullet}\),~\(b_{\bullet}\),~\(c_{\bullet}\in C^{*}(X)\):
\begin{enumarabic}
\item
  \label{def-ax-1}
  all elements of odd degree,
\item all coboundaries,
\item
  all elements~\(E_{k}(a;b_{1},\dots,b_{k})\) with~\(k\ge1\),
\item
  \label{def-ax-4}
  all elements~\(F_{kl}(a_{1},\dots,a_{k};b_{1},\dots,b_{l})\) with~\((k,l)\ne(1,1)\),
\item
  \label{def-ax-5}
  all elements~\(a\cuptwo E_{k}(b;c_{1},\dots,c_{k})\) with~\(k\ge2\),
\item
  \label{def-ax-6}
  all elements~\(a\cuptwo \RC_{k}(b_{0},\dots,b_{k})\) with~\(k\ge0\)
  where \(a\) and~\(b_{\bullet}\) are cocycles.
\end{enumarabic}

\begin{lemma}
  Let \(\kappa\colon X\to Y\) be a map of simplicial sets.
  Then \(\kappa^{*}(\ax_{Y})\subset\ax_{X}\).
\end{lemma}

\begin{proof}
  This follows directly from the naturality of the extended hga operations.
\end{proof}

Let us write
\begin{equation}
  [\alpha,\beta] = \alpha\beta-(-1)^{\deg\alpha\deg\beta}\,\beta\alpha
\end{equation}
for the commutator of~\(\alpha\),~\(\beta\in C^{*}(X)\).

\begin{lemma}
  \label{thm:commutator-ax}
  For all~\(\alpha\),~\(\beta\in C^{*}(X)\),
  \begin{equation*}
    [\alpha,\beta] \equiv 0 \pmax.
  \end{equation*}
\end{lemma}

\begin{proof}
  This results from the definition of~\(\ax\) and the identity
  \begin{equation}
    d(\cupone)(\alpha;\beta) = [\alpha,\beta].
    \qedhere
  \end{equation}
\end{proof}

\begin{lemma}
    \label{thm:cupone-right-derivation-cup}
    The \(\cupone\)-product is a right derivation of the commutator. That is,
    \begin{equation*}
      [\alpha,\beta]\cupone\gamma = (-1)^{\deg\alpha}[\alpha,\beta\cupone\gamma] + (-1)^{\deg\beta\deg\gamma}[\alpha\cupone\gamma,\beta]
    \end{equation*}
    for all~\(\alpha\),~\(\beta\),~\(\gamma\in C^{*}(X)\).
\end{lemma}

\begin{proof}
  This is a consequence of the Hirsch formula~\eqref{eq:hirsch-formula}.
\end{proof}

\goodbreak

\begin{lemma}
  \label{thm:cuptwo-prod}
  Let \(\alpha\),~\(\beta\),~\(\gamma\in C^{*}(X)\).
  \begin{enumroman}
  \item
    \label{thm:cuptwo-prod-1}
    Modulo~\(\ax\), the \(\cuptwo\)-product is both a left and a right derivation of the commutator. That is,
    \begin{align*}
      \alpha \cuptwo (\beta\,\gamma) &\equiv (\alpha\cuptwo\beta)\,\gamma + (-1)^{\deg\alpha\deg\beta}\,\beta\,(\alpha\cuptwo\gamma) \pmax, \\
      (\alpha\,\beta) \cuptwo \gamma &\equiv (-1)^{\deg\beta\deg\gamma}\,(\alpha\cuptwo\gamma)\,\beta + \alpha\,(\beta\cuptwo\gamma) \pmax.
    \end{align*}
  \item
    \label{thm:cuptwo-prod-2}
    One has
    \begin{equation*}
      \alpha\cuptwo[\beta,\gamma] \equiv
      [\alpha,\beta]\cuptwo\gamma \equiv 0 \pmax.
    \end{equation*}
  \end{enumroman}
\end{lemma}

\begin{proof}
  The first part follows from the identities
  \begin{align}
    d(F_{12})(\alpha;\beta,\gamma) &\eqKS E_{2}(\alpha;\beta,\gamma)
    - F_{11}(\alpha;\beta)\,\gamma + F_{11}(\alpha,\beta\,\gamma) - F_{11}(\alpha;\gamma), \\
    d(F_{21})(\alpha,\beta;\gamma) &\eqKS \alpha\,F_{11}(\beta;\gamma) - F_{11}(\alpha\,\beta;\gamma) + 
      F_{11}(\alpha;\gamma)\,\beta - E_{2}(\gamma;\alpha,\beta),
  \end{align}
  see~\eqref{eq:d-Fkl}.
  It implies the formulas
  \begin{align}
    \alpha \cuptwo [\beta,\gamma] &\equiv [\alpha\cuptwo\beta,\gamma] + (-1)^{\deg\alpha\deg\beta}[\beta,\alpha\cuptwo\gamma] \pmax, \\
    [\alpha,\beta] \cuptwo \gamma &\equiv (-1)^{\deg\beta\deg\gamma}[\alpha\cuptwo\gamma,\beta] + [\alpha,\beta\cuptwo\gamma] \pmax,
  \end{align}
  which together with \Cref{thm:commutator-ax} entail the second claim.
\end{proof}

\begin{remark}
  \label{rem:ax-kerf}
  So far we have only used parts~\ref{def-ax-1}--\ref{def-ax-4} of the definition of~\(\ax\).
  The elements listed there are also contained in~\(\ker\transpp{f}\) by \Cref{thm:ffbar-hga-formal}
  and because \(H^{*}(BT)\) is concentrated in even degrees.
  Lemmas~\ref{thm:commutator-ax}--\ref{thm:cuptwo-prod} therefore hold as well for~\(X=BT\) and \(\ker\transpp{f}\) instead of~\(\ax\),
  for any choice of representatives~\(c_{i}\in C_{1}(T)\).
\end{remark}

\begin{lemma}
  \label{thm:cuptwo-Ek}
  Let \(\alpha\),~\(\beta\),~\(\gamma_{1}\),~\dots,~\(\gamma_{k}\in C^{*}(BT)\) with~\(k\ge2\). Then
  \begin{equation*}
    \alpha\cuptwo E_{k}(\beta;\gamma_{1},\dots,\gamma_{k}) \equiv E_{k}(\beta;\gamma_{1},\dots,\gamma_{k}) \cuptwo \alpha \equiv 0 \pmkerf.
  \end{equation*}
\end{lemma}

\begin{proof}
  When the surjection~\(e_{k}\) with~\(k\ge2\) is split into two, then at least one part will have an enclave.
  By the composition rule for the surjection operad, this implies that each surjection appearing in~\(f_{11}\circ_{2}e_{k}\) or~\(f_{11}\circ_{1}e_{k}\)
  again has an enclave.
  This gives the desired identities by \Cref{thm:enclave-0}.
\end{proof}

\begin{lemma}
  \label{thm:cuptwo-cupone-deg2}
  Let \(\alpha\),~\(\beta\),~\(\gamma\in C^{*}(BT)\). If \(\alpha\) is cocycle of degree~\(\deg\alpha\le2\), then
  \begin{equation*}
    \alpha\cuptwo(\beta\cupone\gamma) \equiv (\beta\cupone\gamma)\cuptwo\alpha \equiv 0 \pmkerf.
  \end{equation*}
\end{lemma}

\begin{proof}  
  We consider the element~\(g_{12}=(2,3,1,3,1,2,1)\) in the surjection operad,
  following Kadeishvili~\cite[Rem.~2]{Kadeishvili:2003}.\footnote{%
    Kadeishvili takes \(g_{12}=(1,2,1,3,1,3,2)\) and~\(g_{21}=(1,2,3,2,3,1,3)\) instead. Assuming his definition of~\(E^{1}_{pq}\)
    (see \Cref{fn:kadeishvili1}),
    this gives the formula for~\(d(G_{12})\) stated in~\cite[Rem.~2]{Kadeishvili:2003}
    with the term~\((a\cuptwo c)\cupone b\) replaced by~\((a\cuptwo b)\cupone c\).
    In the formula for~\(d(G_{21})\), the double \(\cuptwo\)-product should read \((a\cuptwo c)\cupone b\).}
  It satisfies
  \begin{equation}
    d\,g_{12} =  e_{1}\circ_{1} f_{11} + (1\,2)\cdot(e_{1}\circ_{2}f_{11}) - f_{11}\circ_{2}e_{1}-f_{12} + (2\,3)\cdot f_{12}
  \end{equation}
  and the corresponding interval cut operation therefore
  \begin{equation}
    \label{eq:d-G12}
    d(G_{12})(\alpha,\beta,\gamma) \equiv \mp \alpha\cuptwo(\beta\cupone\gamma) \pmkerf.
  \end{equation}

  Because there are three \(1\)'s appearing in~\(g_{12}\), the corresponding interval cut operation on a simplex~\(\sigma\)
  has the property that the first simplex~\(\sigma_{(1)}\) in the resulting tensor product involves three intervals, each of which contributes at least one vertex.
  Now the first occurrence of~\(1\) in~\(g_{12}\) is surrounded by two occurrences of~\(3\). Hence the associated
  interval for this~\(1\) must involve at least two vertices for otherwise the simplex~\(\sigma_{(3)}\) made up of the \(3\)-intervals
  would contain twice the same vertex and therefore be degenerate. Hence \(\sigma_{(1)}\) is of dimension at least~\(3\).

  Dually, \(G_{12}(\alpha,\beta,\gamma)\) vanishes for~\(\deg\alpha\le2\), which implies that the left-hand side
  of~\eqref{eq:d-G12} is a coboundary if \(\alpha\) is additionally a cocycle.
  This proves that the first term in the statement is congruent to~\(0\).

  The second part follows analogously by looking at~\(g_{21}=(3,1,3,2,3,2,1)\), which satisfies
  \begin{equation}
    d\,g_{21} = (2\,3)\cdot(e_{1}\circ_{1}f_{11}) + e_{1}\circ_{2}f_{11} - f_{11}\circ_{1} e_{1} + f_{21} - (1\,2)\cdot f_{21}.
    \qedhere
  \end{equation}
\end{proof}

\begin{proposition}
  \label{thm:cuptwo-rep-cupone}
  For all cocycles~\(\alpha\),~\(\beta_{0}\),~\dots,~\(\beta_{k}\in C^{*}(BT)\), \(k\ge1\),
  we have
  \begin{equation*}
    \alpha\cuptwo \RC_{k}(\beta_{0},\dots,\beta_{k}) \equiv \RC_{k}(\beta_{0},\dots,\beta_{k})\cuptwo\alpha \equiv 0 \pmkerf.
  \end{equation*}
\end{proposition}

\begin{proof}
  We show that the first term in the statement lies in~\(\ax\); the proof for the second is analogous.
  
  Write \(\beta=\RC_{k}(\beta_{0},\dots,\beta_{k})\) and  
  assume first that \(\alpha=d\gamma\) is a coboundary.
  Then
  \begin{equation}
    d(\gamma\cuptwo\beta) = d\gamma\cuptwo\beta \pm \gamma\cuptwo d\beta
    \pm \gamma\cupone \beta \pm \beta\cupone\gamma,
  \end{equation}
  hence
  \begin{equation}
    \label{eq:dC1}
    \alpha\cuptwo\beta \equiv \mp \gamma\cuptwo d\beta \pmkerf.
  \end{equation}
  Since \(\beta_{1}\),~\dots,~\(\beta_{k}\) are cocycles, we have
  \begin{equation}
    \label{eq:dC2}
    d\,\beta = \sum_{i=1}^{k} \pm \RC_{k-i}\bigl(\bigl[\RC_{i-1}(\beta_{0},\dots,\beta_{i-1}),\beta_{i}\bigr],\beta_{i+1},\dots,\beta_{k}\bigr)
  \end{equation}
  By a repeated application of \Cref{thm:cupone-right-derivation-cup}
  we see that each term on the right-hand side of~\eqref{eq:dC2} is a sum of commutators,
  so that the right-hand side of~\eqref{eq:dC1} vanishes by the analogue of \Cref{thm:cuptwo-prod}\,\ref{thm:cuptwo-prod-2}.
  This proves the claim for~\(\alpha=d\gamma\).
  
  As a consequence, we may replace \(\alpha\) by any cocycle cohomologous to it.
  Because \(H^{*}(BT)\) is generated in degree~\(2\), we may in particular assume
  that \(\alpha\) is the product of cocycles of degree~\(2\). By the analogue of \Cref{thm:cuptwo-prod}\,\ref{thm:cuptwo-prod-1},
  it is enough to consider the case where \(\alpha\) is a single degree-\(2\) cocycle,
  where \Cref{thm:cuptwo-cupone-deg2} applies.
\end{proof}

\begin{corollary}
  \label{cor:ax-kerf}
  Assume that \(2\) is invertible in~\(\kk\), and let \(\transpp{f}\colon C^{*}(BT)\to H^{*}(BT)\)
  be a formality map as in \Cref{thm:ffbar-cuptwo}. Then \(\kkk_{BT}\subset\ker\transpp{f}\).
\end{corollary}

\begin{proof}
  We have already observed in \Cref{rem:ax-kerf}
  that \(\ker\transpp{f}\) contains the elements mentioned in parts~\ref{def-ax-1}--\ref{def-ax-4}
  of the definition of~\(\kkk_{BT}\). Part~\ref{def-ax-5} is covered by \Cref{thm:cuptwo-Ek},
  and part~\ref{def-ax-6} by \Cref{thm:cuptwo-rep-cupone} for~\(k\ge1\) and by \Cref{thm:ffbar-cuptwo} for~\(k=0\).
\end{proof}

\begin{remark}
  \label{rem:ax-kerf-int}
  \Cref{cor:ax-kerf} remains valid for any~\(\kk\) and any formality map constructed in \Cref{sec:BT-formality}
  if part~\ref{def-ax-6} in the definition of~\(\kkk\) is restricted to~\(k\ge1\), that is,
  if \(\kkk\) is only required to contain
  \begin{enumarabic}
  \item[(6\('\))]
    all elements~\(a\cuptwo \RC_{k}(b_{0},\dots,b_{k})\) with~\(k\ge1\)
    where \(a\) and~\(b_{\bullet}\) are cocycles.
  \end{enumarabic}
\end{remark}

}

\section{Spaces and shc maps}
\label{sec:spaces-shc}

Let \(X\) be a simplicial set, and let \(\kkk=\kkk_{X}\lhd C^{*}(X)\)
be the ideal defined in the preceding section.
We want to relate \(\kkk\) to the canonical shc structure on~\(C^{*}(X)\).
From \Cref{thm:shc-mod-k} we conclude
that the structure map~\(\Phi\colon C^{*}(X)\otimes C^{*}(X)\Rightarrow C^{*}(X)\) is \(\kkk\)-strict.
Moreover, the associativity homotopy~\(\ha\colon \Phi\circ(\Phi\otimes1)\simeq\Phi\circ(1\otimes\Phi)\)
is \(\kkk\)-trivial, but the commutativity homotopy~\(\hc\colon \Phi\circ T_{A,A} \simeq \Phi\) is not in general.
This failure requires extra attention.

We introduce the following terminology:
Let \(A\) and~\(B\) be dgas, \(\bbb\lhd B\) and \(m\ge0\).
An shm homotopy~\(h\colon A^{\otimes m}\to B\) is called \newterm{\(\bbb\)-trivial on cocycles} if
\begin{equation}
  h_{(n)}(a_{11}\otimes\dots\otimes a_{1m},\dots,a_{n1}\otimes\dots\otimes a_{nm}) \equiv 0 \pmod\bbb
\end{equation}
for all~\(n\ge1\) and all cocycles~\(a_{11}\),~\dots,~\(a_{nm}\in A\).
Similarly, an shc map~\(f\colon A^{\otimes m}\Rightarrow B\) is called
\newterm{\(\bbb\)-natural on cocycles} if there is a homotopy~\(h\colon A^{\otimes 2m}\to B\)
that is \(\bbb\)-trivial on cocycles and makes the diagram~\eqref{eq:def-natural-shc-map} commute.

\begin{lemma}
  \label{thm:shm-h-trivial-cocycles}
  Let \(h\),~\(k\colon A^{\otimes m}\to B\) be shm homotopies, \(\bbb\)-trivial on cocycles.
  \begin{enumroman}
  \item
    \label{thm:shm-h-trivial-cocycles-cup}
    \label{thm:shm-h-trivial-cocycles-inverse}
    The shm homotopies~\(h\cup k\) and~\(h^{-1}\) are again \(\bbb\)-trivial on cocycles.
  \item
    \label{thm:shm-h-trivial-cocycles-comp}
    If \(f\colon B\to C\) is a \(\ccc\)-strict shm map such that \(f_{(1)}(\bbb)\subset\ccc\), then \(f\circ h\) is \(\ccc\)-trivial on cocycles.
  \item
    If \(T\colon A^{\otimes m}\to A^{\otimes m}\) is some permutation of the factors, then \(h\circ T\) is \(\bbb\)-trivial on cocycles.
  \end{enumroman}
\end{lemma}

\begin{proof}
  The first claim follows from the definition of the cup product and
  the formula for the inverse given in \Cref{thm:tw-h-equiv-rel}\,\ref{thm:inverse-trivial}.
  The second part is analogous to \Cref{thm:trivial-h-comp}\,\ref{thm:trivial-h-comp-1},
  and the last claim is trivial.
\end{proof}

Because \(\kkk\) contains all \(\cuptwo\)-products of cocycles,
both \(\hc\) and the homotopy~\(\invhc=\hc\circ T\) in the other direction 
are \(\kkk\)-trivial on cocycles. We need to extend this observation.

\begin{lemma}
  \label{thm:hc-1Phin-trivial}
  For any~\(n\ge0\), the shm homotopy
  \begin{equation*}
    \hc\circ\bigl(1\otimes\iter{\Phi}{n}\bigr)\colon C^{*}(X)\otimes C^{*}(X)^{\otimes n}\to C^{*}(X)
  \end{equation*}
  is \(\kkk\)-trivial on cocycles, and the same holds with~\(\invhc\) instead of~\(\hc\).
\end{lemma}

\begin{proof}
  By naturality we may assume that \(\kappa\) is the identity map of~\(X=BT\).
  
  Let \(l\ge0\), and let \(b_{1}\),~\dots,~\(b_{l}\in C^{*}(BT)^{\otimes n}\) with~\(b_{i}=b_{i,1}\otimes\dots\otimes b_{i,n}\)
  where all~\(b_{i,j}\) are cocycles.
  We claim that \(\iter{\Phi}{n}_{(l)}(b_{1},\dots,b_{l})\)
  is a linear combination of products of terms of the following two kinds:
  Repeated \(\cupone\)-products~\(\RC_{k}(c_{0},\dots,c_{k})\) of cocycles with~\(k\ge0\),
  or \(E_{k}\)-terms with~\(k\ge2\) (and not necessarily cocycles as arguments).

  This follows by induction: \(\iter{\Phi}{n}_{(0)}=0\), and \(\iter{\Phi}{n}_{(1)}(b_{1})=b_{1,1}\cdots b_{1,n}\)
  is a product of cocycles~\(b_{1,j}=\RC_{0}(b_{1,j})\). For the induction step, we observe from the formula for~\(\Phi\)
  and the composition formula~\eqref{eq:twc-composition} for shm maps
  that we get products of terms~\(E_{m}(\dots)\) with some value~\(\iter{\Phi}{n}_{(l)}(b_{1},\dots,b_{l})\)
  plugged into the first argument and cocycles into the remaining arguments.
  
  We consider each factor~\(E_{m}\) separately.
  For~\(m=0\) the induction hypothesis applies and for~\(m\ge2\) there is nothing to show. So assume \(m=1\).  
  By induction and the Hirsch formula~\eqref{eq:hirsch-formula}, we may assume that the first argument
  is a repeated \(\cupone\)-product of cocycles or a term~\(E_{k}\) with~\(k\ge2\).
  In the former case we get another repeated \(\cupone\)-product of cocycles.
  In the latter case the identity~\eqref{eq:formula-Ek-El} shows that we end up with a sum of terms~\(E_{k'}\) with~\(k'\ge k\ge2\).
  This completes the proof of the claim.

  Now consider \(\hc_{(m)}\) for~\(m\ge1\), or rather
  its description modulo~\(\kkk\) given in \Cref{thm:shc-mod-k}\,\ref{thm:hcn-k}.
  We have to plug cocycles into the first arguments and values \(\iter{\Phi}{n}_{(l)}(b_{1},\dots,b_{l})\)
  as above into the second arguments.
  Because the \(\cuptwo\)-product is a derivation modulo~\(\kkk\) by \Cref{thm:cuptwo-prod}\,\ref{thm:cuptwo-prod-1},
  we only have to consider terms of the following two kinds in light of our previous discussion:
  firstly, terms~\(a\cuptwo b\)
  where \(a\) is a cocycle and \(b\) a repeated \(\cupone\)-product of cocycles, and secondly,
  \(\cuptwo\)-prod\-ucts where the second argument is an~\(E_{k}\)-term with~\(k\ge2\).
  Both kinds of terms are contained in~\(\kkk\) by definition.

  The proof for~\(\invhc\) is analogous.
\end{proof}

\begin{proposition}
  \label{thm:Phin-ax-natural}
  For any~\(n\ge0\)
  the iteration~\(\iter{\Phi}{n}\colon C^{*}(X)^{\otimes n}\Rightarrow C^{*}(X)\)
  is an shc map that is \(\kkk\)-natural on cocycles.
\end{proposition}

\begin{proof}
  Munkholm~\cite[Prop.~4.5]{Munkholm:1974} has shown that \(\iter{\Phi}{n}\) is an shc map.
  The non-trivial part of the proof is to construct a homotopy
  \begin{equation}
    h^{[n]}\colon \Phi\circ\bigl(\iter{\Phi}{n}\otimes\iter{\Phi}{n}\bigr) \simeq \iter{\Phi}{n}\circ \Phi^{\otimes n}\circ T_{n}
  \end{equation}
  where \(T_{n}\colon A^{\otimes n}\otimes A^{\otimes n}\to (A\otimes A)^{\otimes n}\)
  is the reordering of the \(2n\)~factors~\(A=C^{*}(X)\) corresponding to the permutation
  \begin{equation}
    \begin{pmatrix}
      1 & 2 & \dots & n & n+1 & n+2 & \dots & 2n \\  
      1 & 3 & \dots & 2n-1 & 2 & 4 & \dots & 2n
    \end{pmatrix}.
  \end{equation}  
  We follow Munkholm's arguments and verify that in our setting they lead to a homotopy that is \(\kkk\)-trivial on cocycles.
  There is nothing to show for~\(n\le1\).

  We start with the case~\(n=2\), see~\cite[p.~31]{Munkholm:1974}. The homotopies labelled \(h_{1}\),~\(h_{2}\),~\(h_{4}\) and~\(h_{5}\)
  by Munkholm are \(\kkk\)-trivial because \(\Phi\circ(1\otimes\Phi)\) is \(\kkk\)-strict and the homotopy~\(\ha\) is \(\kkk\)-trivial,
  see Lemmas~\ref{thm:shc-mod-k},~\ref{thm:shm-h-trivial-cocycles}\,\ref{thm:shm-h-trivial-cocycles-comp} and~\ref{thm:trivial-h-comp}.
  So consider the homotopy
  \begin{equation}
    \label{eq:h3}
    h_{3} = \Phi\circ(1_{A}\otimes\Phi)\circ(1_{A}\otimes\invhc\otimes 1_{A}).
  \end{equation}
  We have shown above
  that \(\invhc\) is \(\kkk\)-trivial on cocycles.
  Together with the \(\kkk\)-strictness of~\(\Phi\circ(1\otimes\Phi)\) 
  this implies by \Cref{thm:shm-h-trivial-cocycles}
  that \(h_{3}\) is \(\kkk\)-trivial on cocycles, too, and therefore also the sought-after homotopy
  \begin{equation}
    \label{eq:h-h12345}
    h^{[2]} = h_{1}\cup h_{2}\cup h_{3}\cup h_{4}\cup h_{5}.
  \end{equation}

  For the induction step we have another set of homotopies~\(h_{1}\) to~\(h_{5}\), see~\cite[p.~32]{Munkholm:1974}.
  The homotopies~\(h_{1}\) and~\(h_{4}\) are \(\kkk\)-trivial by \Cref{thm:comp-shm-tensor} and \Cref{thm:trivial-h-comp} because \(\Phi\) is \(\kkk\)-strict.
  The homotopy~\(h_{3}\) is actually not needed. In fact, the identity
  \begin{equation}
    \bigl(\iter{\Phi}{n}\otimes1\bigr) \circ T_{A,A^{\otimes n}} =  T_{A,A} \circ \bigl(1\otimes\iter{\Phi}{n}\bigr)
  \end{equation}
  (see~\cite[\S 3.6\,(iii)]{Munkholm:1974})
  and~\cite[Prop.~3.3\,(ii)]{Munkholm:1974} (or \Cref{thm:comp-shm-tensor}) implies that
  \begin{multline}
    \bigl(1_{A}\otimes T_{A,A}\otimes 1_{A}\bigr) \circ \bigl(\iter{\Phi}{n}\otimes 1_{A}\otimes\iter{\Phi}{n}\otimes 1_{A}\bigr) \\*
    = \bigl(\iter{\Phi}{n}\otimes\iter{\Phi}{n}\otimes1_{A}\otimes1_{A}\bigr) \circ \bigl(1_{A^{\otimes n}}\otimes T_{A,A^{\otimes n}}\otimes 1_{A}\bigr),
  \end{multline}
  which means that the homotopy relation labelled~``\(\stackrel{3}{\simeq}\)'' in~\cite[p.~32]{Munkholm:1974} is an equality.
  That the homotopy~\(h_{5}\) is \(\kkk\)-trivial on cocycles uses that so is \(h^{[n]}\) by induction,
  that \(\Phi\) is \(\kkk\)-strict and also \Cref{thm:shm-h-trivial-cocycles}. 

  To show that
  \begin{equation}
    h_{2} = h^{[2]}\circ \bigl(\iter{\Phi}{n}\otimes 1_{A}\otimes\iter{\Phi}{n}\otimes 1_{A}\bigr)
  \end{equation}
  is \(\kkk\)-trivial on cocycles,
  we may by~\eqref{eq:Hom-cup-dgc} and \Cref{thm:shm-h-trivial-cocycles}\,\ref{thm:shm-h-trivial-cocycles-cup}
  consider the composition with each factor in~\eqref{eq:h-h12345} separately.
  (Recall that \(h\circ f=h\,\BB f\) for an shm homotopy~\(h\) and an shm map~\(f\).)
  The homotopies~\(h_{1}\),~\(h_{2}\),~\(h_{4}\) and~\(h_{5}\) for the case~\(n=2\)
  are \(\kkk\)-trivial on all arguments and therefore pose no problem.
  
  It remains to look at the homotopy
  \begin{align}
    \label{eq:h3comp}
    k_{1} &= h_{3}\circ \bigl(\iter{\Phi}{n}\otimes 1_{A}\otimes\iter{\Phi}{n}\otimes 1_{A}\bigr) \\
    \notag &= \Phi\circ(1_{A}\otimes\Phi)\circ(1_{A}\otimes\invhc\otimes 1_{A})\circ \bigl(\iter{\Phi}{n}\otimes 1_{A}\otimes\iter{\Phi}{n}\otimes 1_{A}\bigr).
  \shortintertext{We want to compare it to the homotopy}
    k_{2} &= \Phi\circ(1_{A}\otimes\Phi)\circ\Bigl(\iter{\Phi}{n}\otimes \bigl(\invhc\circ\bigl(1_{A}\otimes\iter{\Phi}{n}\bigr)\bigr)\otimes 1_{A}\Bigr).
  \shortintertext{Denoting reduction modulo~\(\kkk\) by a bar above a map, we have}
    \bar k_{1} &= \iter{\mu_{A/\kkk}}{3}\circ (1_{A/\kkk}\otimes\barinvhc\otimes 1_{A/\kkk})\circ \bigl(\iter{\bar\Phi}{n}\otimes 1_{A}\otimes\iter{\Phi}{n}\otimes \bar 1_{A}\bigr), \\
    \bar k_{2} &= \iter{\mu_{A/\kkk}}{3}\circ \Bigl(\iter{\bar\Phi}{n}\otimes \bigl(\barinvhc\circ\bigl(1_{A}\otimes\iter{\Phi}{n}\bigr)\bigr)\otimes \bar 1_{A}\Bigr).
  \end{align}
  Because \(\iter{\bar\Phi}{n}\) is a dga map, \(\bar k_{1}\) and \(\bar k_{2}\) agree, see \Cref{thm:tensor-shm-h}.
  Moreover, the homotopy \(\barinvhc\circ(1\otimes\iter{\Phi}{n})\) is \(0\)-trivial (that is, trivial) on cocycles by \Cref{thm:hc-1Phin-trivial},
  which together with \Cref{thm:shm-h-trivial-cocycles}\,\ref{thm:shm-h-trivial-cocycles-comp}
  implies that \(\bar k_{2}\) has the same property.
  Putting these facts together, we obtain that \(k_{1}\) is \(\kkk\)-trivial on cocycles.  
  This completes the proof.
\end{proof}

Let \(n\ge0\) and choose cocycles~\(a_{1}\),~\dots,~\(a_{n}\in C^{*}(X)\) of even positive degrees.
We write \(\aa=(a_{1},\dots,a_{n})\) and consider the shm map
\begin{equation}
  \label{eq:def-Lambda-a}
  \Lambda_{\aa}\colon \kk[\xx]\coloneqq\kk[x_{1}]\otimes\dots\otimes\kk[x_{n}] \stackrel{\lambda_{\aa}}{\longrightarrow} C^{*}(X)^{\otimes n} \stackrel{\iter{\Phi}{n}}{\Longrightarrow} C^{*}(X)
\end{equation}
where \(x_{1}\),~\dots,~\(x_{n}\) are indeterminates of degrees~\(\deg{x_{i}}=\deg{a_{i}}\) and
\(\lambda_{\aa}\) is the tensor product of the dga maps sending each~\(x_{i}\) to~\(a_{i}\).

\begin{remark}
  The map~\(\Lambda_{\aa}\) can be expressed in terms of the hga operations on~\(C^{*}(X)\).
  It is not the same as Wolf's explicit shm~map \cite[Sec.~3]{Wolf:1977}, which only uses \(\cupone\)-products.
\end{remark}

\begin{proposition}
  \label{thm:Lambda-shc-natural}
  If \(2\) is invertible in~\(\kk\), then the map
  \( 
    \Lambda_{\aa}\colon \kk[\xx]\to C^{*}(X)
  \) 
  is a \(\kkk\)-strict and \(\kkk\)-natural shc map.
\end{proposition}

\begin{proof}
  Since \(\iter{\Phi}{n}\) is \(\kkk\)-strict, so is \(\Lambda_{\aa}\) by \Cref{thm:trivial-h-comp}\,\ref{thm:trivial-h-comp-0}.
  It remains to consider the diagram
  \begin{equation}
    \begin{tikzcd}[column sep=large]
      \kk[\xx]\otimes\kk[\xx] \arrow{d}[left]{\lambda_{\aa}\otimes\lambda_{\aa}} \arrow{r}{\mu^{\otimes n}\circ T_{n}} & \kk[\xx] \arrow{d}{\lambda_{\aa}} \\
      C^{*}(X)^{\otimes n}\otimes C^{*}(X)^{\otimes n} \arrow[Rightarrow]{d}[left]{\iter{\Phi}{n}\otimes\iter{\Phi}{n}} \arrow[Rightarrow]{r}{\Phi^{\otimes n}\circ T_{n}\,} & C^{*}(X)^{\otimes n} \arrow[Rightarrow]{d}{\iter{\Phi}{n}} \\
      C^{*}(X)\otimes C^{*}(X) \arrow[Rightarrow]{r}{\Phi} & C^{*}(X) \mathrlap{.}
    \end{tikzcd}
  \end{equation}

  Each dga~map~\(\kk[x_{i}]\to C^{*}(X)\), \(x_{i}\mapsto a_{i}\) is in fact a \(\kkk\)-natural shc map
  because we can choose \(b=-\frac{1}{2}\,a_{i}\cuptwo a_{i}\in\kkk\) in the statement of~\citehgashc{Prop.~7.2}.
  Then the shc map~\(\lambda_{\aa}\) is \(\kkk^{\boxtimes n}\)-natural by \Cref{thm:tensor-prod-shc} and induction.
  Because \(\iter{\Phi}{n}\) is \(\kkk\)-strict, its composition~\(h_{1}\) with the homotopy making the top diagram commute is \(\kkk\)-trivial by \Cref{thm:trivial-h-comp}\,\ref{thm:trivial-h-comp-1}.

  Composed with the top left arrow, the homotopy making the bottom square commute
  is \(\kkk\)-trivial by~\Cref{thm:Phin-ax-natural}. The cup product of this composed homotopy~\(h_{2}\)
  with~\(h_{1}\) then is \(\kkk\)-trivial as well by \Cref{thm:tw-h-equiv-rel}\,\ref{thm:cup-trivial-h}. This proves the claim since
  \begin{equation}
    (\lambda_{\aa}\otimes\lambda_{\aa})\circ\bigl(\iter{\Phi}{n}\otimes\iter{\Phi}{n}\bigr)
    = (\lambda_{\aa}\circ\iter{\Phi}{n})\otimes(\lambda_{\aa}\circ\iter{\Phi}{n}) = \Lambda_{\aa}\otimes\Lambda_{\aa}
  \end{equation}
  by \Cref{thm:comp-shm-tensor}.
\end{proof}

If \(H^{*}(X)\cong\kk[\xx]\) is polynomial and each \(a_{i}\) represents \(x_{i}\) under this isomorphism,
then \(\Lambda_{\aa}\) is a quasi-isomorphism. Note that it depends both on the choice of the generators~\(x_{i}\)
and of their representatives~\(a_{i}\).

\begin{theorem}
  \label{thm:h-square-twc}
  Assume that \(2\) is invertible in~\(\kk\), and
  let \(\phi\colon Y\to X\) be a map of simplicial sets with polynomial cohomology.
  Let \(\aa\) and~\(\bb\) be representatives of some generators of~\(H^{*}(X)\) and~\(H^{*}(Y)\), respectively.
  Then the diagram
  \begin{equation*}
    \begin{tikzcd}
      H^{*}(X) \arrow[Rightarrow]{d}[left]{\Lambda_{\aa}} \arrow{r}{\phi^{*}} & H^{*}(Y) \arrow[Rightarrow]{d}{\Lambda_{\bb}} \\
      C^{*}(X) \arrow{r}{\phi^{*}} & C^{*}(Y)
    \end{tikzcd}
  \end{equation*}
  commutes up to a \(\kkk_{Y}\)-trivial shm homotopy.
\end{theorem}

The corresponding result in~\cite[Sec.~7]{Munkholm:1974} is the heart of Munkholm's paper,
and for our proof of \Cref{thm:main} in the next section \Cref{thm:h-square-twc} will also be crucial.

\begin{proof}
  We write \(f=\phi^{*}\), \(\aa=(a_{1},\dots,a_{n})\) and \(\kkk=\kkk_{Y}\).
  By assumption, we have \(H^{*}(X)=\kk[x_{1},\dots,x_{n}]\).
  We define \(\tilde\aa=(H^{*}(f)(a_{1}),\dots,H^{*}(f)(a_{n}))\) and consider the diagram
  \begin{equation}
    \begin{tikzcd}[column sep=large]
      & H^{*}(X) \arrow{dl}[above]{\lambda_{\aa}\quad\mkern0mu} \arrow{dr}{\!\!\lambda_{\tilde\aa}} \\
      C^{*}(X)^{\otimes n} \arrow[Rightarrow]{d}{\iter{\Phi_{X}}{n}} \arrow{r}[below]{f^{\otimes n}} & C^{*}(Y)^{\otimes n} \arrow[Rightarrow]{d}{\iter{\Phi_{Y}}{n}} & H^{*}(Y)^{\otimes n} \arrow[Rightarrow]{l}{\Lambda_{\bb}^{\otimes n}} \arrow{d}{\iter{\mu}{n}} \\
     C^{*}(X) \arrow{r}[below]{f} & C^{*}(Y) & \arrow[Rightarrow]{l}{\Lambda_{\bb}} H^{*}(Y) \mathrlap{.}
    \end{tikzcd}
  \end{equation}
  The composition from~\(H^{*}(X)\) to~\(C^{*}(X)\) equals \(\Lambda_{\aa}\), and the one from~\(H^{*}(X)\) to~\(H^{*}(Y)\) is \(H^{*}(f)\).
  The left square commutes strictly by the naturality of the hga structure.
  \Cref{thm:Phi-n-homotopy} implies that the right square commutes up to a \(\kkk\)-trivial homotopy
  since \(\Lambda_{\bb}\) is \(\kkk\)-strict and \(\kkk\)-natural by \Cref{thm:Lambda-shc-natural}.

  The composition
  \begin{equation}
    \label{eq:Aicomp1}
    \kk[x_{i}] \longrightarrow H^{*}(Y) \stackrel{\Lambda_{\bb}}{\Longrightarrow} C^{*}(Y)
  \end{equation}
  is a \(\kkk\)-strict shm map, and
  the composition
  \begin{equation}
    \label{eq:Aicomp2}
    \kk[x_{i}] \longrightarrow C^{*}(X) \stackrel{f}{\longrightarrow} C^{*}(Y)
  \end{equation}
  is the dga~map sending \(x_{i}\) to~\(f(a_{i})\).
  Since both \((\Lambda_{\bb})_{(1)}(\tilde a_{i})\) and~\(f(a_{i})\) represent the even-degree element~\(\tilde a_{i}\in H^{*}(B)\)
  and \(\kkk\) contains all elements of odd degree,
  these two maps are homotopic via a \(\kkk\)-trivial shm homotopy by~\citehgashc{Prop.~7.1}.

  The two ways to go from~\(H^{*}(X)\) to~\(C^{*}(Y)^{\otimes n}\) in the diagram represent the tensor products of the shm~maps just discussed.
  This implies by induction and \Cref{thm:trivial-h-comp} that also the triangle commutes up to a \(\kkk^{\boxtimes n}\)-trivial homotopy.
  Its composition with~\(\iter{\Phi_{Y}}{n}\) is \(\kkk\)-trivial by \Cref{thm:trivial-h-comp} as \(\iter{\Phi_{Y}}{n}\) is \(\kkk\)-strict.
  \Cref{thm:tw-h-equiv-rel} 
  concludes the proof.
\end{proof}

\section{Homogeneous spaces}
\label{sec:homog}

We are now ready to prove \Cref{thm:main}.
We assume in this and the next section that \(\kk\) is a principal ideal domain in which \(2\) is invertible.

Let \(G\) be a connected Lie group and \(\iota\colon K\hookrightarrow G\) a closed connected subgroup
such that the order of the torsion subgroup of~\(H^{*}(G;\Z)\) is invertible in~\(\kk\) and analogously for~\(K\).
This implies that \(BG\) and~\(BK\) have polynomial cohomology over~\(\kk\)
(and in fact is equivalent to it), see~\cite[Rem.~IV.8.1]{HusemollerMooreStasheff:1974}.
While we will make use of a maximal torus~\(T\subset K\) in our proof,
\(G\) could more generally be any topological group such that \(BG\) has polynomial cohomology in the sense of \Cref{sec:simp-prelim}.
By \Cref{thm:homog-reduced-groups} we may assume both~\(BG\) and~\(BK\) to be \(1\)-reduced.
For simplicity, we denote the induced maps~\(C^{*}(BG)\to C^{*}(BK)\) and~\(H^{*}(BG)\to H^{*}(BK)\) both by~\(\iota^{*}\).

Our goal is to construct an isomorphism of graded algebras
\begin{equation}
  \label{eq:map-Tor-HGK}
  H^{*}(G/K) \cong \Tor_{H^{*}(BG)}\bigl(\kk,H^{*}(BK)\bigr),
\end{equation}
natural in the pair~\((G,K)\).
Recall from \Cref{thm:iso-bar-GK} that there is a natural isomorphism of graded algebras
\begin{equation}
  \label{eq:map-Tor-CGK}
  H^{*}(G/K) \cong \Tor_{C^{*}(BG)}\bigl(\kk,C^{*}(BK)\bigr),
\end{equation}
It suffices therefore to connect the two bar constructions underlying the torsion products in~\eqref{eq:map-Tor-HGK} and~\eqref{eq:map-Tor-CGK}.
We start by establishing an isomorphism of graded \(\kk\)-modules,
proceeding in a way similar to Munkholm~\cite[\S 7.4]{Munkholm:1974}.
Remember that we have defined one-sided bar constructions as twisted tensor products in~\eqref{eq:def-bar-one-sided}.

Let \(\aa\) and~\(\bb\) be representatives of generators of~\(H^{*}(BG)\) and~\(H^{*}(BK)\), respectively.
We write the induced shm~quasi-isomorphism~\(\Lambda_{\aa}\colon H^{*}(BG)\Rightarrow C^{*}(BG)\)
introduced in~\eqref{eq:def-Lambda-a} as~\(\Laa\)
and analogously \(\Lbb=\Lambda_{\bb}\colon H^{*}(BK)\Rightarrow C^{*}(BK)\).

We define the map
\begin{equation}
  \Theta_{G,K}\colon \BB(\kk,H^{*}(BG),H^{*}(BK)) \to \BB(\kk,C^{*}(BG),C^{*}(BK))
\end{equation}
as the composition of the chain maps
\begin{equation}
  \begin{tikzcd}
    \BB H^{*}(BG) \otimes_{\iota^{*}\circ t_{H^{*}(BG)}} H^{*}(BK) \arrow{d}{\FLbb} \\
    \BB H^{*}(BG) \otimes_{\Lbb\circ\iota^{*}\circ t_{H^{*}(BG)}} C^{*}(BK) \arrow{d}{\delta_{h}} \\
    \BB H^{*}(BG) \otimes_{\iota^{*}\circ\Laa\circ t_{H^{*}(BG)}} C^{*}(BK) \arrow{d}{\BB\Laa\otimes 1} \\
    \BB C^{*}(BG) \otimes_{\iota^{*}\circ t_{C^{*}(BG)}} C^{*}(BK) \mathrlap{,}
  \end{tikzcd}
\end{equation}
given by Lemmas~\ref{thm:quiso-B-1-1-g},~\ref{thm:homotop-tw} and~\ref{thm:quiso-B-1-f-1}, respectively,
where the \(\kkk_{BK}\)-trivial twisting cochain homotopy~\(h\) in the second step comes from \Cref{thm:h-square-twc}.
Note that \(\Theta_{G,K}\) depends
on the chosen representative cocycles~\(\aa\) and~\(\bb\).

\begin{lemma}
  \label{thm:Theta-modulo-k}
  Modulo~\(\BB C^{*}(BG) \otimes\kkk_{BK}\) we have
  \begin{equation*}
    \Theta_{G,K} \equiv \BB\Laa \otimes \Lbbone.
  \end{equation*}
\end{lemma}

Recall that \(\Lbbone\) is the quasi-isomorphism of complexes
\begin{equation}
  H^{*}(BK)\cong\kk[y_{1},\dots,y_{n}] \to C^{*}(BK),
  \qquad
  y_{1}^{k_{1}}\cdots y_{n}^{k_{n}} \mapsto b_{1}^{k_{1}}\cdots b_{n}^{k_{n}}.
\end{equation}

\begin{proof}
  The congruence follows from Lemmas~\ref{thm:homotop-tw},~\ref{thm:quiso-B-1-f-1} and~\ref{thm:quiso-B-1-1-g},
  given that \(\Lbb\) is a \(\kkk_{BK}\)-strict shm map and \(h\) a \(\kkk_{BK}\)-trivial homotopy.
\end{proof}

\begin{proposition} \( \)
  \label{thm:Theta-additive}
  \begin{enumroman}
  \item
    \label{thm:Theta-additive-1}
    The map
    \begin{equation*}
      H^{*}(\Theta_{G,K})\colon \Tor_{H^{*}(BG)}\bigl(\kk,H^{*}(BK)\bigr) \to \Tor_{C^{*}(BG)}\bigl(\kk,C^{*}(BK)\bigr)
    \end{equation*}
    is an isomorphism of graded \(\kk\)-modules.
  \item
    \label{thm:Theta-additive-2}
    The Eilenberg--Moore spectral sequence for the fibration~\(G/K\hookrightarrow BK\to BG\)
    collapses at the second page.
  \end{enumroman}
\end{proposition}

\begin{proof}
  Both \(\Laa\) and~\(\Lbb\) are quasi-isomorphisms, and so is \(\BB\Laa\).
  It follows from \Cref{thm:Theta-modulo-k} as in \Cref{rem:tw-tensor-quiso}
  that \(\Theta_{G,K}\) induces an isomorphism between the second pages of these spectral sequences
  and therefore between the torsion products.
  
  Because the spectral sequence for~\(\BB(\kk,H^{*}(BG),H^{*}(BK))\) collapses at this stage,
  so does the one for~\(\BB(\kk,C^{*}(BG),C^{*}(BK))\), which is the Eilenberg--Moore spectral sequence of the fibration.
\end{proof}

\begin{remark}
  We assume that \(2\) is a unit in~\(\kk\) to ensure that the shc maps~\(\Laa\) and~\(\Lbb\)
  are natural with respect to~\(\kkk_{BG}\) and~\(\kkk_{BK}\), respectively,
  see the proof of \Cref{thm:Lambda-shc-natural}.
  If this were not the case, then the congruence in \Cref{thm:Theta-modulo-k} would still hold
  modulo~\(\BB C^{*}(BG)\otimes C^{>0}(BK)\), and this is enough to prove \Cref{thm:Theta-additive}.
  We thus recover Munkholm's collapse theorem
  for spaces with polynomial cohomology \cite[Thm.]{Munkholm:1974}.
\end{remark}

We now turn to the multiplicativity and naturality of~\(H^{*}(\Theta_{G,K})\).
Here our approach is inspired by Wolf~\cite[p.~331]{Wolf:1977}.
Let \(\kappa\colon T\to K\) be a morphism of simplicial groups where \(T\) is some torus.
We also choose a formality map~\(\ffbar^{*}\colon C^{*}(BT)\to H^{*}(BT)\) as in \Cref{thm:ffbar-cuptwo},
keeping in mind that \(\ffbar^{*}\) annihilates \(\kkk_{BT}\) by \Cref{cor:ax-kerf}.
\def\Pi{\ffbar^{*}}%
Based on~\(\kappa\) and~\(\Pi\)
we define the map
\begin{equation}
  \Psi_{\kappa}\colon \BB(\kk,C^{*}(BG),C^{*}(BK)) \to \BB(\kk,C^{*}(BG),H^{*}(BT))
\end{equation}
as the composition
\begin{equation}
  \label{eq:def-Psi}
  \begin{tikzcd}
    \BB C^{*}(BG) \otimes_{\iota^{*}\circ t_{C^{*}(BG)}} C^{*}(BK) \arrow{d}{1\otimes\kappa^{*}} \\
    \BB C^{*}(BG) \otimes_{\kappa^{*}\iota^{*}\circ t_{C^{*}(BG)}} C^{*}(BT) \arrow{d}{1\otimes \QBT } \\
    \BB C^{*}(BG) \otimes_{\QBT\kappa^{*}\iota^{*}\circ t_{C^{*}(BG)}} H^{*}(BT) \mathrlap{.}
  \end{tikzcd}
\end{equation}

\begin{lemma} \( \)
  \label{thm:HPsi}
  \begin{enumroman}
  \item
    \label{thm:HPsi-1}
    \(\Psi_{\kappa}\) is a morphism of dgas.
  \item
    \label{thm:HPsi-2}
    If \(\kappa\) is the inclusion of a maximal torus into~\(K\), then \(H^{*}(\Psi_{\kappa})\)
    is injective, hence so is the map
    \( 
      H^{*}(G/K) \to H^{*}(G/T)
    \). 
  \item
    \label{thm:HPsi-3}
    The composition~\(\Psi_{\kappa}\,\Theta_{G,K}\) 
    is the map
    \begin{equation*}
      \begin{tikzcd}
	\BB H^{*}(BG)\otimes_{\iota^{*}\circ t_{H^{*}(BG)}} H^{*}(BK) \arrow{d}{\BB\Laa\otimes\kappa^{*}} \\
        \BB C^{*}(BG) \otimes_{\QBT\kappa^{*}\iota^{*}\circ t_{C^{*}(BG)}} H^{*}(BT) \mathrlap{.}
      \end{tikzcd}
    \end{equation*}
  \end{enumroman}
\end{lemma}

The idea of reducing to a maximal torus 
goes back to Baum~\cite[Lemma~7.2]{Baum:1968}.
Note that \Cref{thm:homotopy-twisted}\,\ref{thm:homotopy-twisted-1}
confirms that \(\BB\Laa\otimes\kappa^{*}\) is a chain map because \(\BB\Laa\) is \(\kkk_{BG}\)-trivial,
so that \(\QBT\,\kappa^{*}\,\iota^{*}\,t_{C^{*}(BG)}\,\BB\Laa=\kappa^{*}\,\iota^{*}\,t_{H^{*}(BG)}\colon\BB H^{*}(BG)\to H^{*}(BT)\).

\begin{proof}
  The first map in~\eqref{eq:def-Psi} is a dga~map by naturality
  and the second one by inspection of the product formula~\eqref{eq:def-prod-one-sided-bar}.
  This proves the first claim.

  If \(T\subset K\) is a maximal torus, then \(H^{*}(K/T)\) is concentrated in even degrees,
  as is \(H^{*}(BK)\) by assumption. Hence the Serre spectral sequence for the fibration \(K/T\hookrightarrow BT\to BK\)
  degenerates at the second page. By the Leray--Hirsch theorem, this implies
  that \(H^{*}(BT)\) is a free module over~\(H^{*}(BK)\) with~\(\kappa^{*}(H^{*}(BK))\) 
  being a direct summand.

  As a consequence, the induced map
  \begin{equation}
    \Tor_{H^{*}(BG)}\bigl(\kk,H^{*}(BK)\bigr) \xrightarrow{\Tor_{1}(1,\kappa^{*})} \Tor_{H^{*}(BG)}\bigl(\kk,H^{*}(BT)\bigr)
  \end{equation}
  is injective. This is the map between the second pages of the Eilenberg--Moore spectral sequences for~\(G/K\) and~\(G/T\), respectively.
  Because these spectral sequences degenerate at this level by \Cref{thm:Theta-additive}\,\ref{thm:Theta-additive-2},
  this implies that the map~\(1\otimes\kappa^{*}\) in~\eqref{eq:def-Psi} is injective in cohomology.

  Another standard spectral sequence argument shows that the map~\(1\otimes \QBT \) in~\eqref{eq:def-Psi} is a quasi-isomorphism since \(\QBT \) is so.
  Together with the naturality of the isomorphism~\eqref{eq:map-Tor-CGK} this shows the second claim.

  The last part is a consequence of \Cref{thm:Theta-modulo-k}.
\end{proof}

\begin{theorem}
  \label{thm:Theta-mult}
  The isomorphism~\(H^{*}(\Theta_{G,K})\) is multiplicative.
\end{theorem}

\begin{proof}
  Let \(\kappa\colon T\hookrightarrow K\) be the inclusion of a maximal torus. By \Cref{thm:HPsi} 
  it suffices to prove that the composition~\(\Psi_{\kappa}\,\Theta_{G,K}=\BB\Laa\otimes\kappa^{*}\) is multiplicative up to homotopy.
  Clearly, \(\kappa^{*}\) is multiplicative.
  
  We claim that \(\BB\Laa\) is multiplicative up to a \(\kkk_{BG}\)-trivial coalgebra homotopy
  \begin{equation}
    h\colon \BB H^{*}(BG)\otimes \BB H^{*}(BG) \to \BB C^{*}(BG).
  \end{equation}
  To see this, we consider the diagram
  \begin{equation}
    \!\!\!
    \begin{tikzcd}[column sep=2em]
      \BB H^{*}(BG)\otimes \BB H^{*}(BG) \arrow{d}{\BB\Laa\otimes \BB\Laa} \arrow{r}{\shuffle} & \BB\bigl(H^{*}(BG)\otimes H^{*}(BG)\bigr) \arrow{d}{\BB(\Laa\otimes \Laa)} \arrow{r}{\BB\mu} & \BB H^{*}(BG) \arrow{d}{\BB\Laa} \\
      \BB C^{*}(BG)\otimes \BB C^{*}(BG) \arrow{r}{\shuffle} & \BB\bigl(C^{*}(BG)\otimes C^{*}(BG)\bigr) \arrow{r}{\BB\Phi} & \BB C^{*}(BG) \mathrlap{.}
    \end{tikzcd}
  \end{equation}
  The composition along the top row is the multiplication in~\(\BB H^{*}(BG)\), and by~\citehgashc{Prop.~4.3}
  the one along the bottom row equals the product in~\(\BB C^{*}(BG)\).
  The left square commutes by naturality of the shuffle map (\Cref{thm:shuffle-natural-shm}).
  The right square commutes up to a \(\kkk_{BG}\)-trivial coalgebra homotopy because \(\Laa\)
  is a \(\kkk_{BG}\)-natural shc map by \Cref{thm:Lambda-shc-natural}.
  The claim follows.
  
  By transposing factors,
  we can pass from the tensor product
  \begin{equation}
    \bigl( \BB H^{*}(BG)\otimes_{\iota^{*}\circ t_{H^{*}(BG)}} H^{*}(BK) \bigr)
    \otimes
    \bigl( \BB H^{*}(BG)\otimes_{\iota^{*}\circ t_{H^{*}(BG)}} H^{*}(BK) \bigr)
  \end{equation}
  to the single twisted tensor product
  \begin{equation}
    \bigl( \BB H^{*}(BG)\otimes\BB H^{*}(BG) \bigr) \otimes_{t'} \bigl( \BB H^{*}(BK)\otimes\BB H^{*}(BK) \bigr)
  \end{equation}
  whose twisting cochain~\(t'\) vanishes except for
  \begin{equation}
    t'([x]\otimes 1)=\iota^{*}(x)\otimes 1,
    \qquad
    t'(1\otimes[x])=1\otimes\iota^{*}(x)
  \end{equation}
  with~\(x\in H^{*}(BG)\). We want to show that the two chain maps
  \begin{equation}
    \bigl(\BB\Laa\otimes\kappa^{*}\bigr)\,\bigl(\mu_{\BB H^{*}(BG)}\otimes\mu_{H^{*}(BK)}\bigr)
    = \BB\Laa\,\mu_{\BB H^{*}(BG)}\otimes\kappa^{*}\mu_{H^{*}(BK)}
  \end{equation}
  and
  \begin{multline}
    \bigl(\mu_{\BB C^{*}(BG)}\otimes\mu_{H^{*}(BK)}\bigr)\,\bigl(\BB\Laa\otimes\BB\Laa\otimes\kappa^{*}\otimes\kappa^{*}\bigr) \\*
    = \mu_{\BB C^{*}(BG)}\,(\BB\Laa\otimes\BB\Laa)\otimes\kappa^{*}\mu_{H^{*}(BK)}
  \end{multline}
  are homotopic.
  Given that the coalgebra homotopy~\(h\) is \(\kkk_{BG}\)-trivial,
  we can appeal to \Cref{thm:homotopy-twisted}\,\ref{thm:homotopy-twisted-2}.
\end{proof}

\begin{theorem}
  \label{thm:iso-natural}
  Let \(\phi\colon (G,K)\to(G',K')\) be a map of pairs, both satisfying our assumptions,
  and choose representatives~\(\aa'\) and~\(\bb'\) for generators of~\(H^{*}(BG')\) and~\(H^{*}(BK')\), respectively.
  Then the following diagram commutes.
  \begin{equation*}
    \begin{tikzcd}[column sep=huge]
      \Tor_{H^{*}(BG')}\bigl(\kk,H^{*}(BK')\bigr) \arrow{d} \arrow{r}{\Tor_{\phi^{*}}(1,\phi^{*})} & \Tor_{H^{*}(BG)}\bigl(\kk,H^{*}(BK)\bigr) \arrow{d} \\
      H^{*}(G'/K') \arrow{r} & H^{*}(G/K)
    \end{tikzcd}
  \end{equation*}
\end{theorem}

\begin{proof}
  Let \(T\subset K\) again be a maximal torus.
  We consider the diagram
  \begin{equation}
    \label{eq:naturality-diag}
    \begin{tikzcd}[column sep=huge]
      \BB(\kk,H^{*}(BG'),H^{*}(BK')) \arrow{d}{\Theta_{G',K'}} \arrow{r}{\BB(1,\phi^{*},\phi^{*})} & \BB(\kk,H^{*}(BG),H^{*}(BK)) \arrow{d}{\Theta_{G,K}} \\
      \BB(\kk,C^{*}(BG'),C^{*}(BK')) \arrow{d}{1\otimes\kappa^{*}\phi^{*}} \arrow{r}{\BB(1,\phi^{*},\phi^{*})} & \BB(\kk,C^{*}(BG),C^{*}(BK)) \arrow{d}{1\otimes\kappa^{*}} \\
      \BB(\kk,C^{*}(BG'),C^{*}(BT)) \arrow{d}{1\otimes \QBT } \arrow{r}{\BB(1,\phi^{*},1)} & \BB(\kk,C^{*}(BG),C^{*}(BT)) \arrow{d}{1\otimes \QBT } \\
      \BB(\kk,C^{*}(BG'),H^{*}(BT)) \arrow{r}\arrow{r}{\BB(1,\phi^{*},1)} & \BB(\kk,C^{*}(BG),H^{*}(BT)) \mathrlap{.}
    \end{tikzcd}
  \end{equation}
  We have to show that the top square in the diagram commutes in cohomology.
  By \Cref{thm:HPsi}\,\ref{thm:HPsi-2}, it suffices to consider
  the prolongations of the maps in question to~\(\BB(\kk,C^{*}(BG),H^{*}(BT))\).
  
  The composition along the path via~\(\BB(\kk,H^{*}(BG),H^{*}(BK))\) gives the map
  \begin{equation}
    \label{eq:comp1}
    \BB\Laa\,\BB\phi^{*}\otimes\kappa^{*}\phi^{*}
  \end{equation}
  by \Cref{thm:HPsi}\,\ref{thm:HPsi-3}.
  Since the middle square in~\eqref{eq:naturality-diag} commutes by naturality and the bottom square by construction,
  the same result shows that
  the path via~\(\BB(\kk,C^{*}(BG'),C^{*}(BK'))\) gives
  \begin{equation}
    \label{eq:comp2}
    \BB\phi^{*}\,\BB\Lambda^{G'}\otimes\kappa^{*}\,\phi^{*}.
  \end{equation}  

  By \Cref{thm:h-square-twc} there is a \(\kkk_{BG}\)-triv\-ial homotopy~\(h\)
  between the shm maps \(\phi^{*}\circ\Lambda^{G'}\) and~\(\Laa\circ\phi^{*}\).
  In other words, \(\BB h\) is a \(\kkk_{BG}\)-trivial coalgebra homotopy between~\(\BB\Laa\,\BB\phi^{*}\) and~\(\BB\phi^{*}\,\BB\Lambda^{G'}\).
  This implies by \Cref{thm:homotopy-twisted}\,\ref{thm:homotopy-twisted-2}
  that \(\BB h\otimes\kappa^{*}\phi^{*}\) is a homotopy between the maps~\eqref{eq:comp1} and~\eqref{eq:comp2}
  and completes the proof.
\end{proof}

\begin{corollary}
  The isomorphism~\eqref{eq:map-Tor-HGK} does not depend on the chosen representatives~\(\aa\) and~\(\bb\).
\end{corollary}

\begin{proof}
  Take \(\phi\colon (G,K)\to(G,K)\) to be the identity map in \Cref{thm:iso-natural}.
\end{proof}

\begin{remark}
  Theorems~\ref{thm:Theta-mult} and~\ref{thm:iso-natural} actually hold not just for principal ideal domains,
  but for all coefficient rings~\(\kk\) in which \(2\) is invertible.
  (We remark that already Gugenheim--May~\cite[\S 4]{GugenheimMay:1974}
  allow Noetherian rings and Husemoller--Moore--Stasheff~\cite{HusemollerMooreStasheff:1974}
  arbitrary coefficients.)
  The only change required is to replace cochain complexes with chain complexes,
  which are now homotopy Gerstenhaber coalgebras. Imitating our arguments,
  one obtains quasi-isomorphisms of dgcs
  \begin{multline}
    \label{eq:quiso-homological}
    C(G/K) \to \OM(\kk,C(BG),C(EG/K)) \\* \leftarrow \OM(\kk,C(BG),C(BK)) \to \OM(\kk, H(BG),H(BK)).
  \end{multline}
  The cobar constructions are dgcs by the homological analogue of \Cref{thm:def-prod-bar}.
  The dual maps to \eqref{eq:quiso-homological} induce the isomorphism~\eqref{eq:map-Tor-HGK}
  by the universal coefficient spectral sequence since \(H(BG)\) and~\(H(BK)\)
  are free of finite type over~\(\kk\).
  We have chosen the cohomological setting in this paper because we expect it to be more accessible.
\end{remark}

\begin{remark}
  \label{rem:PU2}
  Baum~\cite[Ex.~4]{Baum:1968} has observed that there is no multiplicative isomorphism of the form~\eqref{eq:map-Tor-HGK}
  for the projective unitary group~\(PU(n)=U(n)/U(1)\) with~\(n\equiv2\pmod{4}\) and~\(\kk=\Z_{2}\).
  This is readily verified for~\(PU(2)\):
  Recall that the torsion product of graded commutative algebras is bigraded with the \(\Tor\)-degree being non-positive.
  In the case at hand one obtains
  \begin{equation}
    \def\Ztwo{\Z_{2}}
    \begin{array}{ccc|c}
      & \Ztwo & & 4 \\
      & \Ztwo & \Ztwo & 2 \\
      & & \Ztwo & 0 \\
      \hline
      -2 & -1 & 0 & \mathrlap{\;\;\;.}
    \end{array}
  \end{equation}
  Because the product respects bidegrees, the non-zero element in bidegree~\((-1,2)\) squares to~\(0\).
  This does of course not happen for the generator~\(x\in H^{1}(PU(2))\) as \(PU(2)\cong SO(3)\approx\RP^{3}\).
  
  The same counterexample shows that one cannot expect the isomorphism~\eqref{eq:map-Tor-HGK} to be natural if \(2\) is not invertible in~\(\kk\).
  Consider the diagonal map
  \begin{equation}
    PU(2) = U(2) \bigm/ U(1) \to \bigl(U(2)\times U(2)\bigr)\bigm/\bigl(U(1)\times U(1)\bigr) = PU(2)\times PU(2),
  \end{equation}
  which induces the cup product in cohomology.
  Naturality of the isomorphism~\eqref{eq:map-Tor-HGK} would predict that
  the image of~\(x\otimes x\) in~\(H^{2}(PU(2))\) vanishes, which again is not the case.
\end{remark}

\begin{remark}
  \label{rem:gen-homog}
  May--Neumann~\cite{MayNeumann:2002} have observed that \Cref{thm:intro:additive} extends to generalized homogeneous spaces,
  that is, to homotopy fibres of maps~\(\phi\colon Y\to X\) where \(X\) and~\(Y\) take the roles of~\(BG\) and~\(BK\), respectively.
  The same is true for \Cref{thm:Theta-mult}: Assume that \(X\) and~\(Y\) have polynomial cohomology and
  that there is a map~\(\kappa\colon BT\to Y\) where \(BT\) is the classifying space of some torus such that \(H^{*}(BT)\) is a free \(H^{*}(Y)\)-module.
  If \(X\) is, for example, \(1\)-reduced, then the dga~\(\BB(\kk,C^{*}(X),C^{*}(Y))\) computes the cohomology of the homotopy fibre~\(F\) of~\(\phi\),
  and the same argument as before shows that in this case there is an isomorphism of graded algebras
  \begin{equation}
    H^{*}(F) = \Tor_{H^{*}(X)}\bigl(\kk,H^{*}(Y)\bigr)
  \end{equation}
  under our assumption that \(2\) is invertible in~\(\kk\).
\end{remark}

\section{Examples}
\label{sec:examples}

Like Cartan's \Cref{thm:Cartan},
our main result (\Cref{thm:main}) reduces the task of computing the cohomology ring
of a homogeneous space~\(G/K\) to a purely algebraic problem,
provided that one understands the map~\(H^{*}(BG)\to H^{*}(BK)\).
We illustrate this with two examples that recover and (in the case~\(n>n_{1}+\dots+n_{k}\) with~\(k>1\))
generalize computations that can be found in~\cite{BrunerCatanzaroMay:2013},~%
\cite[Sec.~XI.4]{GreubHalperinVanstone:1976},~\cite[Thm.~3.10]{He:2020} and~\cite[Ch.~3]{MimuraToda:1991}.
We continue to assume that \(2\) is invertible in the given principal ideal domain~\(\kk\).

\subsection{Unitary groups}
\label{sec:example-U}

We consider the homogeneous space
\begin{equation}
  U(n)\bigm/ U(n_{1})\times\dots\times U(n_{k})
\end{equation}
where \(k\),~\(n\),~\(n_{1}\),~\dots,~\(n_{k}\) are positive integers such that \(n\ge n_{1}+\dots+n_{k}\).
For~\(k=1\) this is a complex Stiefel manifold. For~\(n=n_{1}+\dots+n_{k}\) we get
a (complete or partial) complex flag variety, in particular a complex Grassmannian for~\(k=2\).

Recall that
\begin{equation}
  H^{*}(BU(n)) = \kk[c_{1},\dots,c_{n}]
\end{equation}
is a polynomial ring in the Chern classes~\(c_{j}\) of degree~\(2j\).
The total Chern class \(c=1+c_{1}+\dots+c_{n}\) restricts to the product
\begin{equation}
  c^{(1)}\cdots c^{(k)} \in H^{*}\bigl(U(n_{1})\times\dots\times U(n_{k})\bigr)
\end{equation}
of the total Chern classes of the factors. Hence
\begin{equation}
  \label{eq:restr-cj}
  \iota^{*}(c_{j}) = \!\!\sum_{j_{1}+\dots+j_{k}=j}\!\! c^{(1)}_{j_{1}}\cdots c^{(k)}_{j_{k}}
\end{equation}
for~\(1\le j\le n\), \cf~\cite[Thm.~3.1]{BrunerCatanzaroMay:2013} or~\cite[Thm.~3.5.8\,(3)]{MimuraToda:1991}.
Here we allow \(j_{i}=0\) by setting \(c^{(i)}_{0}=1\).

Let \(\nn=n_{1}+\dots+n_{k}\) be the rank of~\(K=U(n_{1})\times\dots\times U(n_{k})\).
If this equals the rank~\(n\) of~\(G=U(n)\), then \(H^{*}(G/K)\) is concentrated
in even degrees \cite[Thm.~7.3.21\,(1)]{MimuraToda:1991}. As in the proof of \Cref{thm:HPsi}\,\ref{thm:HPsi-2},
this implies by the Leray--Hirsch theorem that \(H^{*}(BK)\) is free over~\(H^{*}(BG)\), so that
\begin{align}
  \label{eq:ex-U-HGK-same-rank}
  H^{*}(G/K) &= \Tor_{H^{*}(BG)}\bigl(\kk,H^{*}(BK)\bigr) = \kk \otimes_{H^{*}(BG)} H^{*}(BK) \\
  \notag &= 
  \kk\bigl[c^{(1)}_{1},\dots,c^{(1)}_{n_{1}},\dots,c^{(k)}_{1},\dots,c^{(k)}_{n_{k}}\bigr]
  \bigm/ \bigl\langle\iota^{*}(c_{1}),\dots,\iota^{*}(c_{\nn})\bigr\rangle
\end{align}
as a graded \(\kk\)-algebra.

If \(\nn<n\), then the only difference to the previous case is that we have
\(\iota^{*}(c_{j})=0\) for~\(j>\nn\). Hence
\begin{align}
  H^{*}(G/K) &= \Tor_{\kk[c_{1},\dots,c_{n}]}\bigl(\kk,H^{*}(BK)\bigr) \\
  \notag &= \Tor_{\kk[c_{1},\dots,c_{\nn}]}\bigl(\kk,H^{*}(BK)\bigr)
  \otimes \Tor_{\kk[c_{\nn+1},\dots,c_{n}]}(\kk,\kk) \\
  \notag &= R \otimes \bigwedge(x_{2\nn+1},\dots,x_{2n-1})
\end{align}
where \(R\) is the algebra from the last line of~\eqref{eq:ex-U-HGK-same-rank}
and each exterior generator~\(x_{i}\) is of degree~\(i\).
Note that \(R=\kk\) if \(k=1\), that is, in case of a Stiefel manifold.

There is a canonical inclusion~\(\phi\colon U(n)\hookrightarrow U(n')\) for~\(n\le n'\).
The case~\(k=1\) of~\eqref{eq:restr-cj} means that each~\(c_{j}\in H^{*}(BU(n'))\) maps to~\(c_{j}\in H^{*}(BU(n))\)
if~\(j\le n\) and to~\(0\) otherwise.
Let \(G'/K'=U(n')/U(n'_{1})\times\dots\times U(n'_{k})\) be a second quotient with~
\(n_{i}\le n'_{i}\) for all~\(i\).
If \(\phi\) restricts to the canonical inclusion~\(U(n_{i})\hookrightarrow U(n'_{i})\) for each~\(1\le i\le k\),
then the naturality part of \Cref{thm:main} implies that the induced map~\(H^{*}(G'/K')\to H^{*}(G/K)\)
is given as follows: Each generator~\(c^{(i)}_{j}\in H^{*}(G'/K')\) is sent to its counterpart
in~\(H^{*}(G/K)\) if~\(j\le n_{i}\) and to~\(0\) otherwise.
Each exterior generator~\(x_{2j-1}\) is similarly sent to ``itself''
if~\(j\le\nn\) and to~\(0\) otherwise.

\subsection{Special orthogonal groups}

We now turn to the homogeneous space
\begin{equation}
  SO(n) \bigm/ SO(n_{1})\times\dots\times SO(n_{k})
\end{equation}
where \(k\),~\(n\),~\(n_{1}\),~\dots,~\(n_{k}\) are again positive integers such that \(n\ge n_{1}+\dots+n_{k}\).
As in the unitary case, we obtain real Stiefel manifolds, Grassmannians and other
flag varieties as special cases. We set \(G=SO(n)\) and~\(K=SO(n_{1})\times\dots\times SO(n_{k})\).
Depending on whether \(n\) is even or odd, we write \(n=2m\) or~\(n=2m+1\), and similarly for the~\(n_{i}\).
We assume that~\(n_{1}\),~\dots,~\(n_{l}\) are even and \(n_{l+1}\),~\dots,~\(n_{k}\) odd.
We finally abbreviate the rank~\(m_{1}+\dots+m_{k}\) of~\(K\) to~\(\mm\).

Since \(2\) is assumed to be a unit in~\(\kk\), we have
\begin{align}
  \label{eq:HBSOn}
  H^{*}(BSO(n)) = \begin{cases}
    \kk[p_{1},\dots,p_{m-1},e] & \text{if \(n\) is even,} \\
    \kk[p_{1},\dots,p_{m}] & \text{if \(n\) is odd} \\
  \end{cases}
\end{align}
where \(p_{j}\) is the \(j\)-th Pontryagin class of degree~\(4j\),
and for even~\(n\) the Euler class~\(e\) of degree~\(n=2m\) squares to~\(p_{m}\).
The Künneth theorem gives \(H^{*}(BK)\).

Analogously to the total Chern class,
the total Pontryagin class~\(1+p_{1}+\dots+p_{m}\) restricts to the
product~\(p^{(1)}\cdots p^{(k)}\) of the total Pontryagin classes of the factors.
In other words,
\begin{equation}
  \iota^{*}(p_{j}) = \!\!\sum_{j_{1}+\dots+j_{k}=j}\!\! p^{(1)}_{j_{1}}\cdots p^{(k)}_{j_{k}}
\end{equation}
for~\(1\le j\le n\), where again we set \(p^{(i)}_{0}=0\).
For even~\(n\), the Euler class~\(e\) restricts to the product~\(e^{(1)}\cdots e^{(k)}\)
of the Euler classes of the factors if \(\mm=m\) and otherwise to~\(0\) (since so does \(p_{m}\)).
Compare \cite[Cor.~7.3\,(iii)]{BrunerCatanzaroMay:2013}.

We start with the case where \(G\) and~\(K\) have the same rank~\(\mm=m\).
This happens if and only if all~\(n_{i}\) are even and add up to~\(n\).
As before, this implies that \(H^{*}(BK)\) is free over~\(H^{*}(BG)\), so that we have
\begin{equation}
  H^{*}(G/K) = H^{*}(BK)
  \bigm/ \bigl\langle\iota^{*}(p_{1}),\dots,\iota^{*}(p_{m-1}),\iota^{*}(e))\bigr\rangle.
\end{equation}

Now assume \(\mm<m\).
Then \(\iota^{*}(p_{j})=0\) for~\(j>\mm\), and also \(\iota^{*}(e)=0\) if \(n\) is even.
Let \(S\subset H^{*}(BK)\) be the subalgebra generated by all Pontryagin classes~\(p^{(i)}_{j}\)
including \(p^{(i)}_{m_{i}}\) for each even~\(n_{i}\).
From~\eqref{eq:HBSOn} we see that \(H^{*}(BK)\) is a free \(S\)-module with basis
\begin{equation}
  e^{I} = \prod_{i\in I} e^{(i)}
\end{equation}
where \(I\) runs through the subsets of~\(\{1,\dots,l\}\).

The polynomial algebra~\(\kk[p_{1},\dots,p_{\mm}]\subset H^{*}(BG)\) acts on~\(S\) in the same way
as \(H^{*}(BG)\) acted on~\(H^{*}(BK)\) in the unitary example from \Cref{sec:example-U},
except that all degrees are now doubled.
This implies that \(S\) is a free module over~\(\kk[p_{1},\dots,p_{\mm}]\), hence the same holds for~\(H^{*}(BK)\).
For odd~\(n\) we therefore get
\begin{align}
  \label{eq:ex-SO-HGK}
  H^{*}(G/K) &= \Tor_{\kk[p_{1},\dots,p_{\mm}]}\bigl(\kk,H^{*}(BK)\bigr) \otimes
  \Tor_{\kk[p_{\mm+1},p_{m}]}(\kk,\kk) \\
  \notag &= R \otimes \bigwedge(x_{4\mm+3},\dots,x_{4m-1})
\end{align}
where
\begin{equation}
  R = H^{*}(BK) \bigm/ \bigl\langle \iota^{*}(p_{1}),\dots,\iota^{*}(p_{\mm}) \bigr\rangle.
\end{equation}
As before, the subscripts of the exterior generators in~\eqref{eq:ex-SO-HGK} indicate degrees.
For even~\(n\), we similarly get
\begin{equation}
  H^{*}(G/K) = R \otimes \bigwedge(x_{4\mm+3},\dots,x_{4m-5},y_{2m-1})
\end{equation}
where the additional exterior generator is induced from the Euler class~\(e\).

The behaviour of these isomorphisms under maps are analogous to the unitary case.
We omit the details.

\appendix

\section{Completing the proof of Proposition \ref{thm:tensor-shm}}
\label{sec:pf-tensor-prod-Ai}

In this appendix we complete the proof of \Cref{thm:tensor-shm}.
Given two shm maps~\(f\colon A\Rightarrow A'\) and~\(g\colon B\Rightarrow B'\),
we justify that the maps~\(h_{(n)}\) introduced in that proof satisfy
the defining identity~\eqref{eq:tw-h-family-3} for a twisting homotopy family corresponding to an shm homotopy
from~\((1\otimes g)\circ(f\otimes 1)\) to~\((f\otimes1)\circ(1\otimes g)\).
Explicitly, we have to show
\begin{align}
  \label{eq:app:tw-h-family-3}
  d(h_{(n)})(a_{\bullet}\otimes b_{\bullet}) &\eqKS \sum_{k=1}^{n-1}(-1)^{k}\,h_{(n-1)}(a_{\bullet}\otimes b_{\bullet},a_{k}a_{k+1}\otimes b_{k}b_{k+1},a_{\bullet}\otimes b_{\bullet}) \\*
  \notag &\; + \sum_{k=1}^{n} \bigl((1\otimes g)\circ(f\otimes 1)\bigr)_{(k)}(a_{\bullet}\otimes b_{\bullet})\,h_{(n-k)}(a_{\bullet}\otimes b_{\bullet}) \\*
  \notag &\; - \sum_{k=0}^{n-1} (-1)^{k}\,h_{(k)}(a_{\bullet}\otimes b_{\bullet})\, \bigl((f\otimes1)\circ(1\otimes g)\bigr)_{(n-k)}(a_{\bullet}\otimes b_{\bullet})
\end{align}
for~\(n\ge0\) and~\(a_{\bullet}\otimes b_{\bullet}\in A\otimes B\). Recall that \(h_{(0)}=\eta_{A'}\otimes\eta_{B'}\) and
\begin{equation}
  h_{(n)}(a_{\bullet}\otimes b_{\bullet}) \eqKS \sum_{k,l\ge1}\mkern-0mu\sum_{\substack{i_{1}+\dots+i_{k}+\\j_{1}+\dots+j_{l}=n}}\mkern-5mu
  (-1)^{\epsilon}\,F \otimes G
\end{equation}
for~\(n\ge1\), where the second sum is over all decompositions of~\(n\) into~\(k+l\) positive integers,
\begin{align}
  F &= \iter{\mu}{k}\Bigl(f_{(i_{1})}(a_{\bullet}),\dots,f_{(i_{k-1})}(a_{\bullet}),f_{(i_{k}+l)}\bigl(a_{\bullet},\iter{\mu}{j_{1}}(a_{\bullet}),\dots,\iter{\mu}{j_{l}}(a_{\bullet})\bigr)\Bigr), \\
  G &= \iter{\mu}{l}\Bigl(g_{(k+j_{1})}\bigl(\iter{\mu}{i_{1}}(b_{\bullet}),\dots,\iter{\mu}{i_{k}}(b_{\bullet}),b_{\bullet}\bigr),g_{(j_{2})}(b_{\bullet}),\dots,g_{(j_{l})}(b_{\bullet})\Bigr), \\
  \epsilon &= \sum_{s=1}^{k}s\,(i_{s}-1) + \sum_{t=1}^{l}(l-t)(j_{t}-1) + k\,(l-1)+1.
\end{align}
The twisting families corresponding
to the compositions of shm maps~\((1\otimes g)\circ(f\otimes 1)\) and~\((f\otimes1)\circ(1\otimes g)\)
are given by~\eqref{eq:1-g-f-1} and~\eqref{eq:f-1-1-g}, respectively.
We only present a sketch of the proof.
The computation is elementary, but lengthy because of the many cases to consider.

\medskip

The case \(d(h_{(0)})=0\) being clear, we have to compute \(d(h_{(n)})(a_{\bullet}\otimes b_{\bullet})\) for~\(n\ge1\).
Recall from the discussion before~\eqref{eq:example-diff-maps}
that both the composition and the tensor product of maps obey the graded Leibniz rule.
The only maps in~\eqref{eq:def-Hn} that are not chain maps are the components~\(f_{(n)}\) and~\(g_{(n)}\) of~\(f\) and~\(g\), respectively.
From the formula~\eqref{eq:tw-fam-2} we see that two kinds of terms appear
when applying the differential to~\(f_{(n)}\):
those obtained by splitting the arguments and those obtained by multiplying two of them. More precisely,
we say that a term~\(f_{(n)}(a_{\bullet})\) is \newterm{split at position~\(m\)}
if it is split between the \(m\)-th and the \((m+1)\)-st argument,
that is, if we consider the term~\(f_{(m)}(a_{\bullet})f_{(n-m)}(a_{\bullet})\) of \(d(f_{(n)})(a_{\bullet})\).
We similarly say that two arguments are \newterm{multiplied at position~\(m\)}
if we consider the term~\(f_{(n-1)}(a_{\bullet},a_{m}a_{m+1},a_{\bullet})\) where
the arguments at positions~\(m\) and~\(m+1\) are multiplied.
The same applies to the components of~\(g\). We claim that when computing \(d(h_{(n)})\),
all terms on the right-hand side of~\eqref{eq:app:tw-h-family-3} are indeed produced
and all other terms that come up cancel in pairs.

Below is a list all terms that appear in the computation.
In each case, we indicate
whether the corresponding terms cancel against other terms
or pair with terms on the other side of the equation.
The notation ``\textbf{X}~\(\to\)~\textbf{Y}'' means
that the terms~\textbf{X} cancel or pair with the terms~\textbf{Y}.

\medbreak

\noindent\textbf{Terms produced by~\(d(h_{(n)})\)}
\medskip
\begin{caselist}
\item Splitting of a term~\(f_{(i_{s})}\)
  \begin{caselist}
  \item Term~\(f_{(i_{s})}\), \(1\le s<k\), at any position (if \(k\ge 2\))
    \label{tensor:11}
    \cancelswith{tensor:41}
  \item Term~\(f_{(i_{k}+l)}\) at position~\(1\le m<i_{k}\) (if \(i_{k}\ge 2\))
    \label{tensor:11bis}
    \cancelswith{tensor:41bis}
  \item Term~\(f_{(i_{k}+l)}\) at position~\(m=i_{k}\)
    \begin{caselist}
    \item \(j_{1}=1\) and \(l=1\)
      \label{tensor:120}
      \cancelswith{tensorh:21}
    \item \(j_{1}=1\) and \(l>1\)
      \label{tensor:121}
      \cancelswith{tensor:221}
    \item \(j_{1}>1\)
      \label{tensor:122}
      \cancelswith{tensor:321}
    \end{caselist}
  \item Term~\(f_{(i_{k}+l)}\) at position~\(i_{k}+1\le m<i_{k}+l\) (if \(l\ge 2\))
    \label{tensor:13}
    \cancelswith{htensor:2}
  \end{caselist}
  
\item Splitting of a term~\(g_{(j_{t})}\)
  \begin{caselist}
  \item Term~\(g_{(k+j_{1})}\) at position~\(1\le m<k\) (if \(k\ge 2\))
    \label{tensor:21}
    \cancelswith{tensorh:1}
  \item Term~\(g_{(k+j_{1})}\) at position~\(m=k\)
    \begin{caselist}
    \item \(i_{k}=1\) and \(k=1\)
      \label{tensor:220}
      \cancelswith{htensor:11}
    \item \(i_{k}=1\) and \(k>1\)
      \label{tensor:221}
      \cancelswith{tensor:121}
    \item \(i_{k}>1\)
      \label{tensor:222}
      \cancelswith{tensor:421}
    \end{caselist}
  \item Term~\(g_{(k+j_{1})}\) at position~\(k+1\le m<k+j_{1}\) (if \(j_{1}\ge 2\))
    \label{tensor:23}
    \cancelswith{tensor:33}
  \item Term~\(g_{(j_{t})}\), \(1<t\le l\), at any position (if \(l\ge 2\))
    \label{tensor:23bis}
    \cancelswith{tensor:33bis}
  \end{caselist}
  
\item Multiplication of two arguments of a term~\(f_{(i_{s})}\)
  \begin{caselist}
  \item Term~\(f_{(i_{s})}\), \(1\le s<k\), at any position
    \label{tensor:31}
    \cancelswith{tensor1:1}
  \item Term~\(f_{(i_{k}+l)}\) at position~\(1\le m< i_{k}\)
    \label{tensor:31bis}
    \cancelswith{tensor1:1bis}
  \item Term~\(f_{(i_{k}+l)}\) at position~\(m=i_{k}\)
    \begin{caselist}
    \item \(i_{k}=1\) and \(k=1\)
      \label{tensor:320}
      \cancelswith{htensor:12}
    \item \(i_{k}=1\) and \(k>1\)
      \label{tensor:321}
      \cancelswith{tensor:122}
    \item \(i_{k}>1\)
      \label{tensor:322}
      \cancelswith{tensor:422}
    \end{caselist}
  \item Term~\(f_{(i_{k}+l)}\) at position~\(m=i_{k}+1\) (if \(l\ge 2\))
    \label{tensor:33}
    \cancelswith{tensor:23}
  \item Term~\(f_{(i_{k}+l)}\) at position~\(i_{k}+1<m<i_{k}+l\) (if \(l\ge 3\))
    \label{tensor:33bis}
    \cancelswith{tensor:23bis}
  \end{caselist}
  
\item Multiplication of two arguments of a term~\(g_{(j_{t})}\)  
  \begin{caselist}
  \item Term~\(g_{(k+j_{1})}\) at position~\(1\le m<k-1\) (if \(k\ge 3\))
    \label{tensor:41}
    \cancelswith{tensor:11}
  \item Term~\(g_{(k+j_{1})}\) at position~\(m=k-1\) (if \(k\ge 2\))
    \label{tensor:41bis}
    \cancelswith{tensor:11bis}
  \item Term~\(g_{(k+j_{1})}\) at position~\(m=k\)
    \begin{caselist}
    \item \(j_{1}=1\) and \(l=1\)
      \label{tensor:420}
      \cancelswith{tensorh:22}
    \item \(j_{1}=1\) and \(l>1\)
      \label{tensor:421}
      \cancelswith{tensor:222}
    \item \(j_{1}>1\)
      \label{tensor:422}
      \cancelswith{tensor:322}
    \end{caselist}
  \item Term~\(g_{(k+j_{1})}\) at position~\(k+1\le m<k+j_{1}\)
    \label{tensor:43}
    \cancelswith{tensor1:2}
  \item Term~\(g_{(j_{t})}\), \(1<t\le l\), at any position
    \label{tensor:43bis}
    \cancelswith{tensor1:2bis}
  \end{caselist}
\end{caselist}

\bigbreak
\noindent\textbf{Terms appearing in~\(h_{(n-1)}(\dots,a_{m}a_{m+1}\otimes b_{m}b_{m+1},\dots)\)}
\medskip
\begin{caselist}[resume]
\item \(m\le i_{1}+\dots+i_{k-1}\) (if \(k\ge 2\))
  \label{tensor1:1}
  \cancelswith{tensor:31}
\item \(i_{1}+\dots+i_{k-1} < m\le i_{1}+\dots+i_{k}\)
  \label{tensor1:1bis}
  \cancelswith{tensor:31bis}
\item \(i_{1}+\dots+i_{k} < m \le i_{1}+\dots+i_{k}+j_{1}\)
  \label{tensor1:2}
  \cancelswith{tensor:43}
\item \(i_{1}+\dots+i_{k}+j_{1} < m\) (if \(l\ge 2\))
  \label{tensor1:2bis}
  \cancelswith{tensor:43bis}
\end{caselist}

\bigbreak
\noindent\textbf{Terms appearing in~\(\bigl((1\otimes g)\circ(f\otimes 1)\bigr)_{(k)}\cdot h_{(n-k)}\)}
\medskip
\begin{caselist}[resume]
\item \(k<n\)
  \label{tensorh:1}
  \cancelswith{tensor:21}
\item \(k=n\)
  \begin{caselist}
  \item \(i_{n}=1\)
    \label{tensorh:21}
    \cancelswith{tensor:120}
  \item \(i_{n}>1\)
    \label{tensorh:22}
    \cancelswith{tensor:420}
  \end{caselist}
\end{caselist}

\bigbreak
\noindent\textbf{Terms appearing in~\(h_{(k)}\cdot\bigl((f\otimes 1)\circ(1\otimes g)\bigr)_{(n-k)}\)}
\medskip
\begin{caselist}[resume]
\item \(k=0\)
  \begin{caselist}
  \item \(j_{1}=1\)
    \label{htensor:11}
    \cancelswith{tensor:220}
  \item \(j_{1}>1\)
    \label{htensor:12}
    \cancelswith{tensor:320}
  \end{caselist}
\item \(k>0\)
  \label{htensor:2}
  \cancelswith{tensor:13}
\end{caselist}

\medskip

To illustrate how the proof proceeds,
let us discuss two cases in detail. We write \(\ii=(i_{1},\dots,i_{k})\),~\(\jj=(j_{1},\dots,j_{l})\)
and \(\epsilon(\ii,\jj)\) for the sign exponent~\(\epsilon\) from~\eqref{eq:def-epsilon-ij}.
We compute sign exponents modulo~\(2\), indicated by~``\(\equiv\)''.

\medskip

\noindent\thpair{tensor:11}{tensor:41}
The case~\ref{tensor:11} with a splitting of~\(f_{(i_{s})}(a_{\bullet})\)
at position~\(1\le p<i_{s}\) produces the term~\((-1)^{\epsilon'} F'\otimes G'\) with
\begin{align}
  F' &\eqKS \iter{\mu}{k+1}\Bigl(f_{(i_{1})}(a_{\bullet}),\dots,f_{(p)}(a_{\bullet}),f_{(q)}(a_{\bullet}),\dots,f_{(i_{k-1})}(a_{\bullet}), \\*
  \notag &\qquad\qquad\qquad\qquad\qquad\qquad\qquad f_{(i_{k}+l)}\bigl(a_{\bullet},\iter{\mu}{j_{1}}(a_{\bullet}),\dots,\iter{\mu}{j_{l}}(a_{\bullet})\bigr)\Bigr), \\
  \shortintertext{where \(p+q=i_{s}\),}
  G' &= G \\
  \notag &\eqKS \iter{\mu}{l}\Bigl(g_{(k+j_{1})}\bigl(\iter{\mu}{i_{1}}(b_{\bullet}),\dots,\iter{\mu}{i_{k}}(b_{\bullet}),b_{\bullet}\bigr),g_{(j_{2})}(b_{\bullet}),\dots,g_{(j_{l})}(b_{\bullet})\Bigr) \\
  \shortintertext{and}
  \epsilon' &\equiv \epsilon(\ii,\jj) + (1-i_{1}) + \dots + (1-i_{s-1}) + p \\
  \notag &\equiv \epsilon(\ii,\jj) + i_{1} + \dots + i_{s-1} + s + p + 1.
\end{align}
On the other hand, the case~\ref{tensor:41} gives the term~\((-1)^{\epsilon''} F''\otimes G''\) with
\begin{align}
  F'' &= F \\
  \notag &\eqKS \iter{\mu}{k}\Bigl(f_{(i_{1})}(a_{\bullet}),\dots,f_{(i_{k-1})}(a_{\bullet}),f_{(i_{k}+l)}\bigl(a_{\bullet},\iter{\mu}{j_{1}}(a_{\bullet}),\dots,\iter{\mu}{j_{l}}(a_{\bullet})\bigr)\Bigr) \\
  G'' &\eqKS \iter{\mu}{l}\Bigl(g_{(k-1+j_{1})}\bigl(\iter{\mu}{i_{1}}(b_{\bullet}),\dots,\iter{\mu}{i_{m}+i_{m+1}}(b_{\bullet}),\dots,\iter{\mu}{i_{k}}(b_{\bullet}),b_{\bullet}\bigr), \\*
  \notag &\qquad\qquad\qquad\qquad\qquad\qquad\qquad\qquad\qquad\quad g_{(j_{2})}(b_{\bullet}),\dots,g_{(j_{l})}(b_{\bullet})\Bigr), \\
  \epsilon'' &\equiv \epsilon(\ii,\jj) + (1-i_{1}) + \dots + (1-i_{k-1}) + (1-i_{k}-l) + m + 1 \\
  \notag &\equiv \epsilon(\ii,\jj) + i_{1} + \dots + i_{k} + k + l + m + 1.
\end{align}
We rewrite this second case in terms
of~\(\ii''=(i_{1},\dots,i_{s-1},p,q,i_{s+1},\dots,i_{k})\) (of length~\(k''=k+1\)) 
and~\(s=m\). Then \(F''=F'\), \(G''=G'\) and
\begin{align}
  \epsilon(\ii'',\jj) &\equiv \epsilon(\ii,\jj) - s(p+q-1) + s(p-1) + (s+1)(q-1) \\
  \notag &\qquad  + (i_{s+1}-1) + \dots + (i_{k}-1) + (l-1) \\
  \notag &\equiv \epsilon(\ii,\jj) + i_{s+1} + \dots + i_{k} + p + k + l.
\end{align}
Hence
\begin{align}
  \epsilon'' &\equiv \epsilon(\ii'',\jj) + i_{1} + \dots + p + q + \dots + i_{k} + k'' + l + s + 1 \\
  \notag &\equiv \epsilon(\ii,\jj) + p + s + 1 \equiv \epsilon' + 1,
\end{align}
which means that the terms produced by these two cases have opposite signs
and therefore cancel out.

\medskip

\noindent\thpair{tensor:21}{tensorh:1}
The case~\ref{tensor:21} produces the term~\((-1)^{\epsilon'}F'\otimes G'\) with
\begin{align}
  F' &= F \\
  \notag &\eqKS \iter{\mu}{k}\Bigl(f_{(i_{1})}(a_{\bullet}),\dots,f_{(i_{k-1})}(a_{\bullet}),f_{(i_{k}+l)}\bigl(a_{\bullet},\iter{\mu}{j_{1}}(a_{\bullet}),\dots,\iter{\mu}{j_{l}}(a_{\bullet})\bigr)\Bigr), \\
  G' &\eqKS \iter{\mu}{l+1}\Bigl(g_{(m)}\bigl(\iter{\mu}{i_{1}}(b_{\bullet}),\dots,\iter{\mu}{i_{m}}(b_{\bullet})\bigr), \\*
  \notag &\qquad\quad g_{(k-m+j_{1})}\bigl(\iter{\mu}{i_{m+1}}(b_{\bullet}),\dots,
  \iter{\mu}{i_{k}}(b_{\bullet}),b_{\bullet}\bigr),g_{(j_{2})}(b_{\bullet}),\dots,g_{(j_{l})}(b_{\bullet})\Bigr), \\
  \epsilon' &\equiv \epsilon(\ii,\jj) + (1-i_{1}) + \dots + (1-i_{k-1}) + (1-i_{k}-l) + m \\
  \notag &\equiv \epsilon(\ii,\jj) + i_{1} + \dots + i_{k} +k + l + m.
\end{align}
We now consider the case~\ref{tensorh:1} with~\(k=m\). Taking \eqref{eq:1-g-f-1} into account,
we see that it produces terms of the form~\((-1)^{\epsilon''}F'\otimes G'\) with
\begin{align}
  F'' &\eqKS \iter{\mu}{m}\Bigl(f_{(i_{1})}(a_{\bullet}),\dots,f_{(i_{m})}(a_{\bullet})\Bigr)\cdot \iter{\mu}{k-m}\Bigl(f_{(i_{m+1})}(a_{\bullet}),\dots, \\*
  \notag &\qquad\qquad\qquad\qquad\quad\;\; f_{(i_{k-1})}(a_{\bullet}),f_{(i_{k}+l)}\bigl(a_{\bullet},\iter{\mu}{j_{1}}(a_{\bullet}),\dots,\iter{\mu}{j_{l}}(a_{\bullet})\bigr)\Bigr), \\
  G'' &\eqKS g_{(m)}\bigl(\iter{\mu}{i_{1}}(b_{\bullet}),\dots,\iter{\mu}{i_{m}}(b_{\bullet})\bigr)
  \cdot \iter{\mu}{l}\Bigl(g_{(k-m+j_{1})}\bigl(\iter{\mu}{i_{m+1}}(b_{\bullet}),\dots, \\*
  \notag &\qquad\qquad\qquad\qquad\qquad\qquad\qquad \iter{\mu}{i_{k}}(b_{\bullet}),b_{\bullet}\bigr),g_{(j_{2})}(b_{\bullet}),\dots,g_{(j_{l})}(b_{\bullet})\Bigr), \\
  \epsilon'' &\equiv \sum_{s=1}^{m}(s-1)(i_{s}-1) + \epsilon(\ii'',\jj) + (m-1)\Bigl(\sum_{s=m+1}^{k}(i_{s}-1) + l\Bigr)
\end{align}
where \(\ii''=(i_{m+1},\dots,i_{k})\). The first summand of~\(\epsilon''\)
is \eqref{eq:1-g-f-1-epsilon}, and last one arises
because we have moved the second factor of~\(F''\) (which comes from~\(h\)) past the first factor of~\(G''\)
(which comes from~\((1\otimes g)\circ(f\otimes 1)\)).
We have \(F''=F'\), \(G''=G'\) and
\begin{align}
  \epsilon'' &\equiv \sum_{s=1}^{k}s(i_{s}-1) + \sum_{s=1}^{m}(i_{s}-1) + \sum_{s=m+1}^{k}m(i_{s}-1) \\
  \notag &\qquad + \sum_{t=1}^{l}(l-t)(j_{t}-1) + k(l-1) -m(l-1) + 1 \\
  \notag &\qquad + \sum_{s=m+1}^{k}(m-1)(i_{s}-1) + l(m-1) \\
  \notag &\equiv \epsilon(\ii,\jj) + \sum_{s=1}^{k}(i_{s}-1) -m(l-1) + l(m-1) \\
  \notag &\equiv \epsilon(\ii,\jj) + i_{1} + \dots + i_{k} + k + l + m.
\end{align}
Hence the terms produced by these two cases agree, including the signs.

\end{document}